\documentclass[12pt,centertags,oneside]{amsart}
\usepackage{amsmath,amstext,amsthm,a4,amssymb,amscd}
\usepackage{mathrsfs,dsfont}
\usepackage{fancyhdr}
\usepackage{epsf}
\usepackage{mathtools}
\usepackage%[backref=page]
{hyperref}
\hypersetup{
    colorlinks = true,
    linkcolor={magenta},
    citecolor={magenta},
     }
\usepackage{charter}
\usepackage{typearea}
\usepackage{xcolor}
\usepackage{listings}
\usepackage{enumitem}
\usepackage[utf8]{inputenc}

\usepackage[a4paper,width=16.2cm,top=2.5cm,bottom=2cm,
footskip=1cm]{geometry}

\numberwithin{equation}{section}
\newtheorem{theorem}{Theorem}[section]
\newtheorem{lemma}[theorem]{Lemma}
\newtheorem{proposition}[theorem]{Proposition}
\newtheorem{corollary}[theorem]{Corollary}

\theoremstyle{definition}
\newtheorem{definition}[theorem]{Definition}
\newtheorem{example}[theorem] {Example}

\newtheorem{remark}[theorem]{Remark}
\newtheorem{notation}[theorem]{Notation}
\newtheorem{condition}[theorem]{Condition}
\newtheorem{assumption}[theorem]{Assumption}

\newtheorem*{examples}{Examples}
\numberwithin{equation}{section}

\newcommand{\field}[1]{\mathbb{#1}}
\newcommand{\Z}{\field{Z}}
\newcommand{\R}{\field{R}}
\newcommand{\C}{\field{C}}
\newcommand{\N}{\field{N}}

\newcommand{\cali}[1]{\mathscr{#1}}
\newcommand{\cC}{\cali{C}} \newcommand{\cA}{\cali{A}}
\newcommand{\cO}{\cali{O}} 
 
 \newcommand{\cL}{\cali{L}}
 
\newcommand{\cK}{\cali{K}} \newcommand{\cT}{\cali{T}}
\newcommand{\cF}{\cali{F}}
\newcommand{\calig}[1]{\mathcal{#1}}

 \newcommand\mO{\calig{O}}
\newcommand\mQ{\calig{Q}} 
 \newcommand{\cP}{\cali{P}}
 \newcommand\mT{\calig{T}} 

\newcommand{\boldsym}[1]{\boldsymbol{#1}}
\newcommand\bb{\boldsym{b}}

\newcommand\bbf{\boldsym{f}}

\newcommand{\imat}{\sqrt{-1}}
\DeclareMathOperator{\End}{End}

\DeclareMathOperator{\Ker}{Ker}

\DeclareMathOperator{\Dom}{Dom}

\DeclareMathOperator{\rank}{rk}
\DeclareMathOperator{\Id}{Id}
\DeclareMathOperator{\supp}{supp}

\DeclareMathOperator{\Ric}{Ric}
\DeclareMathOperator{\spec}{Spec}

\newcommand{\db}{\overline\partial}
\newcommand{\spin}{$\text{spin}^c$ }
\newcommand{\norm}[1]{\lVert#1\rVert}
\newcommand{\abs}[1]{\lvert#1\rvert}
\newcommand{\om}{\omega}
\newcommand{\ov}{\overline}
\newcommand{\var}{\varepsilon}
\newcommand{\wi}{\widetilde}

\newcommand{\comment}[1]{}

\allowdisplaybreaks

\begin{document}
%%%%%%%%%%%%%%%%%%%%%%%%%%%%%%%%%%%%%%%

%%%%    Title
\title{Berezin-Toeplitz Quantization of non-compact manifolds}

\date{\today}

%%%   Author
\author{Louis Ioos}
\address{CY Cergy Paris Université, 95300 Pontoise,
France}
\email{louis.ioos@cyu.fr}
\thanks{L.\ I.\ was partially 
supported by DIM-R\'egion Ile-de-France, by
the European Research Council Starting grant 757585
and by the ANR-23-CE40-0021-01 JCJC
project QCM}
%-------
\author{Wen Lu}
\address{School of Mathematics and Statistics,
\& Hubei Key Laboratory of Engineering Modeling 
 and Scientific Computing,
% Huazhong University of Science and Technology, 
% \mbox{\quad}\,Wuhan 430074, China
% \newline
%  \mbox{\quad}\,Hubei Key Laboratory of Engineering Modeling 
%  and Scientific Computing,
%  \newline  \mbox{\quad}\,
 Huazhong University of Science and Technology, 
 Wuhan 430074, China}
\email{wlu@hust.edu.cn}
\thanks{W.\ L.\ supported by National Natural Science Foundation
of China (Grant Nos. 11401232, 11871233)}
%-------
\author{Xiaonan Ma}
\address{Chern Institute of Mathematics and LPMC, Nankai University, Tianjin 30071, P.R. China}
\email{xiaonan.ma@nankai.edu.cn}
\thanks{X.\ M.\ was partially supported by
Nankai Zhide Foundation, ANR-14-CE25-0012-01,  and
funded through the Institutional Strategy of
the University of Cologne within the German Excellence Initiative.}
%-------
\author{George Marinescu}
\address{Universit{\"a}t zu K{\"o}ln,  Mathematisches Institut,
    Weyertal 86-90,   50931 K{\"o}ln, Germany
    \newline
    \mbox{\quad}\,Institute of Mathematics `Simion Stoilow',
	Romanian Academy,
Bucharest, Romania}
\email{gmarines@math.uni-koeln.de}
\thanks{G.\ M.\ partially supported by DFG funded
project SFB TRR 191}


\begin{abstract}
We develop Berezin-Toeplitz quantization in a non-compact complex geometric
setting.  Let $(X,\Theta)$ be a Hermitian manifold, $(L,h^L)$ a positive
holomorphic line bundle, and $(E,h^E)$ a holomorphic Hermitian vector bundle.
Assuming that the Kodaira Laplacian on $(0,1)$-forms with values in
$L^p\!\otimes E$ has a spectral gap growing linearly in $p$, we prove that the
Bergman projection onto the $L^2$-holomorphic space
$H^0_{(2)}(X,L^p\!\otimes E)$ enjoys the usual off-diagonal decay and admits a
full asymptotic expansion on compact subsets as $p\to\infty$.  As a consequence,
for every smooth symbol $f\in \mathcal C^\infty_{\mathrm{const}}(X,\End(E))$
(constant outside a compact set), the associated Toeplitz operators
$T_{f,p}=P_p f P_p$ form a closed algebra and satisfy a complete composition
expansion, yielding a star-product on $\mathcal C^\infty_{\mathrm{const}}(X,\End(E))$
and the expected semiclassical commutator formula.
We also give intrinsic criteria characterizing Toeplitz families with compactly
supported kernels.

We then provide geometric conditions guaranteeing the spectral gap on large
classes of non-compact manifolds, via fundamental $L^2$-estimates for $\bar\partial$
on complete Hermitian manifolds (including bounded-geometry 
complete K\"ahler manifolds, K\"ahler-Einstein manifolds,
pseudoconvex/weakly $1$-complete, and quasi-projective manifolds).
Finally, for compactly supported bounded symbols, we prove a Szeg\H{o}-type theorem
describing the eigenvalue distribution of the compact Toeplitz operators $T_{f,p}$
as $p\to\infty$.
\end{abstract}

\maketitle

\tableofcontents

\section{Introduction}

Berezin-Toeplitz quantization is one of the most concrete and flexible
quantization procedures. On a compact K\"ahler manifold $(X,\omega)$ 
endowed with a prequantum line bundle
$(L,h^L)$ such that $\omega=\frac{\sqrt{-1}}{2\pi}R^L$, and 
a holomorphic Hermitian vector bundle $(E,h^E)$, one considers the
quantum spaces $H^0(X,L^p\!\otimes E)$, $p\in\N^*$, and the orthogonal
Bergman projections $P_p$ onto them.  To every observable
$f\in\cC^\infty(X,\End(E))$, one associates the Toeplitz operator
\[
T_{f,p}=P_p\, f\, P_p,
\]
whose semiclassical behavior as $p\to+\infty$ encodes both the 
underlying geometry
and the deformation of the commutative algebra of observables into a star-product.
In the compact setting, this circle of ideas has a long history, from the
microlocal approach of Boutet de Monvel–Guillemin \cite{BdMG81} to the
algebro-geometric and analytic developments of Bordemann–Meinrenken–Schlichenmaier,
Schlichenmaier, and many others
(see e.g. \cite{BMS94,Guill:95,Schlich:00} and the references therein).

We have introduced in \cite{MM07,MM08b} a different approach based on
the existence of a full off-diagonal asymptotic expansion,
of the Bergman kernel and its refinements,
which, in turn, yield the Toeplitz calculus and the associated star-product.
This method found several applications, 
cf.\ \cite{BMMP14,Fine_Quant_Duke_2012,Fin22b,Ioo18b,IKPS20,MM12,MS24a}.

The goal of the present paper is to exhibit large classes of \emph{non-compact}
complex manifolds for which the Berezin–Toeplitz quantization package continues
to hold, in a form suitable for geometric applications.
Going beyond compactness raises two intertwined issues.
First, the quantum spaces of $L^2$-holomorphic sections may be infinite-dimensional, and the Bergman projection need not enjoy global smoothing properties.
Second, even when $L$ is positive, the Bergman kernel asymptotics may fail
without additional control at infinity.
Our guiding principle is that, on non-compact manifolds, the Toeplitz calculus is
available as soon as one can ensure a \emph{spectral gap} for the Kodaira Laplacian
in the high tensor power limit.
This point of view was used in the analytic localization method
developed in \cite{MM07,MM08b} and underlies many 
vanishing theorems.

\smallskip
\noindent\textbf{A Toeplitz package under a spectral gap.}
Let $(X,\Theta)$ be a Hermitian manifold, $(L,h^L)$ a positive holomorphic line bundle,
$(E,h^E)$ a holomorphic Hermitian vector bundle, and set
$\omega=\frac{\sqrt{-1}}{2\pi}R^L$.
Assuming that the Kodaira Laplacian on $(0,1)$-forms with values in $L^p\!\otimes E$
has a \emph{spectral gap} growing linearly in $p$ (cf.\ \eqref{bk1.4}),
we prove that the Bergman kernels of the $L^2$-holomorphic spaces
$H^0_{(2)}(X,L^p\otimes E)$ admit a full asymptotic expansion on compact subsets,
together with the usual off-diagonal exponential decay.
Moreover, the full Berezin–Toeplitz calculus holds for the natural algebra
$\cC^\infty_{\rm const}(X,\End(E))$ of smooth endomorphism-valued functions
that are constant outside a compact set.
This is summarized in Theorem~\ref{t2.1} (and its variants, including the compact case
Theorem~\ref{t2.1comp}).
In particular, the product and commutator expansions
\[
T_{f,p}T_{g,p}\sim \sum_{r\ge0} p^{-r}T_{C_r(f,g),p},
\qquad
[T_{f,p},T_{g,p}]\sim \frac{\sqrt{-1}}{p}\,T_{\{f,g\},p}+\cdots
\]
hold in the appropriate sense, yielding a star-product on
$\cC^\infty_{\rm const}(X,\End(E))$; see Theorem~\ref{T:BTstar}.
An essential result is an intrinsic criterion characterizing Toeplitz families with compactly
supported kernels (Theorem~\ref{toet3.1}), which also
provides an algorithmic way to compute the coefficients $C_r(f,g)$.
%
%\smallskip
%\noindent\textbf{Geometric conditions implying the spectral gap.}
%A substantial part of the paper is devoted to verifying the spectral gap hypothesis
%in a wide range of non-compact situations.
%We first recall the \emph{fundamental estimate} approach of \cite{MM07,MM08a},
%which gives spectral gaps (and $L^2$-cohomology vanishing) under explicit curvature
%and torsion bounds on complete Hermitian manifolds.
%This is formulated here as Conditions~\ref{C:specgapgeom1} and \ref{C:specgapgeom2},
%leading to Theorems~\ref{noncompact0}, \ref{noncompact1} and the resulting Toeplitz
%package Theorem~\ref{t2.12}.
%We then discuss several geometric sources of these estimates, including:
%complete K\"ahler--Einstein manifolds, weakly $1$-complete (in particular Stein)
%manifolds, quasi-projective manifolds equipped with big line bundles and singular
%metrics, manifolds of bounded geometry, and pseudoconvex domains.
%These examples are meant to demonstrate that Berezin--Toeplitz quantization is not
%confined to the compact manifolds, but applies naturally in analytic and arithmetic
%settings where $L^2$-holomorphic sections are the correct quantum spaces.
%

\smallskip
\noindent\textbf{Geometric conditions implying the spectral gap.}
%The abstract spectral gap assumption is verified in practice by $L^2$-techniques
%for the $\bar\partial$-Laplacian, in the spirit of the fundamental estimate
%of \cite{MM07,MM08a}.  In Section~\ref{S:noncomp}, we formulate several concrete geometric
%hypotheses (completeness and uniform curvature positivity) 
%under which the Kodaira Laplacian on
%$L^p\otimes E$ has a spectral gap of order $p$.  
A substantial part of the paper is devoted to verifying the spectral gap hypothesis
in a wide range of non-compact situations.
We first recall the \emph{fundamental estimate} approach of \cite{MM07,MM08a},
which gives spectral gaps (and $L^2$-cohomology vanishing) under explicit curvature
and torsion bounds on complete Hermitian manifolds.
This is formulated here as Conditions~\ref{C:specgapgeom1} and \ref{C:specgapgeom2},
leading to Theorems~\ref{noncompact0}, \ref{noncompact1}, and the resulting Toeplitz
package Theorems~\ref{t2.12}, \ref{t2.11} for compactly
supported (or constant outside a compact set) symbols.
We then illustrate these criteria through a
collection of standard non-compact geometries, showing that the Toeplitz package
is applicable far beyond the compact case.  Typical examples include:

\emph{Complete K\"ahler manifolds of bounded geometry.}
If $(X,\Theta)$ is complete with bounded geometry (uniform lower bound on the
injectivity radius and uniform bounds on the curvature and its derivatives),
and if $(L,h^L)$ and $(E,h^E)$ have bounded geometry, then
the Berezin-Toeplitz quantization holds for observables
whose derivatives are all bounded; cf. Theorem \ref{thm:3.2new23b}.
%is uniformly positive outside a compact set (or globally),
%then the fundamental estimate yields a spectral gap on $(0,1)$-forms for all
%large $p$.  This covers, for instance, complete K\"ahler manifolds with
%uniformly controlled curvature at infinity, and produces Toeplitz quantization
%on the corresponding $L^2$-holomorphic spaces.
In this setting, we mention the recent work \cite{ILMM26a}, which investigates the exponential decay of the Bergman kernel and applies these results to the existence of Poincar\'e series.

\emph{K\"ahler--Einstein and negatively curved geometries.}
Many canonical complete metrics arising in complex geometry fall into the
previous framework, such as complete K\"ahler--Einstein manifolds of negative
Ricci curvature.  In these settings, the positivity of the canonical line bundle combines
with the global control provided by the metric to yield the desired
spectral gap and hence the Bergman kernel expansion on compact subsets.

\emph{Pseudoconvex domains and weighted Bergman spaces.}
For pseudoconvex domains (in particular, strongly pseudoconvex domains) endowed
with a positive line bundle (for example, the trivial line bundle endowed
with a strictly plurisubharmonic weight), 
we obtain the spectral gap for the Kodaira Laplacian
with $\db$-Neumann boundary conditions.
Our results then give a Toeplitz calculus, recovering and extending classical
Toeplitz operators on Bergman-type spaces to a geometric line bundle setting.

\emph{Stein, $1$-convex and weakly $1$-complete manifolds.}
By using a smooth plurisubharmonic exhaustion that is strictly
plurisubharmonic outside a compact set, one can build Hermitian metrics
and weights with coercivity at infinity.  This yields a spectral gap for the
high-power Kodaira Laplacians and, consequently, Toeplitz quantization.  
This provides a flexible class of examples where
the ``quantum spaces'' are naturally infinite-dimensional.

\emph{Quasi-projective manifolds and Poincar\'e-type ends.}
For $X=\overline X\setminus D$ with $D$ being a normal crossings divisor in 
a compact manifold $\overline X$, one can work
with complete Poincar\'e-type (or cusp-type) metrics near $D$ and with line
bundles whose curvature dominates the metric.
We again derive a spectral gap and obtain Toeplitz asymptotics on compact subsets of
$X$.  This is particularly suited to arithmetic and locally symmetric
situations, where finite-volume quotients naturally carry cusp geometries.

Let us mention the related work \cite{HM17a}, where
the authors develop Berezin–Toeplitz quantization with quantum spaces 
being the spectral spaces of the Kodaira Laplacian on the set where 
the curvature is positive; see also \cite{ILMM20} for the symplectic case.

% \smallskip
% \noindent\textbf{Coverings, Poincar\'e series, and isotropic states.}
% Another theme of the paper concerns the interaction between Toeplitz quantization
% and covering geometry.
% Given a covering $\pi:\widetilde X\to X=\widetilde X/\Gamma$ of finite volume and a
% $\Gamma$-equivariant positive line bundle $\widetilde L$, we prove that the
% high-power Kodaira Laplacians on $\widetilde X$ and on $X$ enjoy uniform spectral gaps,
% under bounded geometry and curvature positivity assumptions
% (Theorem~\ref{th-cov}).
% This allows us to compare the corresponding Bergman kernels and to build holomorphic
% sections on quotients via Poincar\'e series.
% In particular, we revisit the construction of isotropic states associated with
% Bohr--Sommerfeld submanifolds and show how to periodize them on coverings, leading to
% $\Gamma$-automorphic holomorphic sections with controlled $L^2$-behavior
% (Theorem~\ref{t5.6}).
% We also present concrete families of examples on Hermitian symmetric spaces and
% their finite-volume quotients, where these constructions produce explicit and
% geometrically meaningful holomorphic sections.

\smallskip
\noindent\textbf{A Szeg\H{o}-type theorem for compactly supported symbols.}
Finally, we prove a spectral asymptotics result for compact Toeplitz operators
associated with bounded symbols of compact support.
Under the hypotheses of Theorem~\ref{t2.1}, the operators $T_{f,p}$ are compact when
$f$ has compact support, and their eigenvalue distribution satisfies a
Szeg\H{o}-type law (Theorems~\ref{t6.1} %and \ref{t6.2}
), extending our non-compact
framework to classical phenomena for Toeplitz quantization and Bergman spaces.

\smallskip
\noindent\textbf{Organization of the paper.}
In Section~\ref{BTsec}, we set up the analytic framework and recall the Bergman kernel
expansion and Toeplitz calculus under a spectral gap, proving, in particular, Theorem~\ref{t2.1}, as well as the Toeplitz criteria and product expansions.
Section~\ref{S:noncomp} is devoted to geometric conditions ensuring the spectral gap on non-compact
manifolds and to a collection of examples.
Section~\ref{S:Szego} establishes the Szeg\H{o}-type eigenvalue asymptotics for compactly supported symbols.

%--------
%%%
\section{The Berezin-Toeplitz package and the spectral gap}
\label{BTsec}

\subsection{An abstract setting}
Let $(X,J)$ be a complex manifold of complex dimension $n$, and let $g^{TX}$ be a
$J$-invariant Riemannian metric. Let $\Theta$ be the associated real $(1,1)$-form
$\Theta(X, Y)=g^{TX}(JX, Y)$.
The Riemannian volume form is $dv_X=\Theta^n/n!$.
Let $(F,h^F)$ be a holomorphic Hermitian vector bundle on $X$. On
$\cC^{\infty}_0(X,F)$, we consider the $L^2$ inner product
\begin{equation}\label{lm2.0}
\big\langle  s_1,s_2 \big\rangle	
	:=\int_X\big\langle  s_1(x),s_2(x)\big\rangle_{F}\,dv_X(x)\,,
\end{equation}
where $\langle\cdot,\cdot\rangle_{F}$ is induced by $h^F$.
The completion of $\cC^{\infty}_0(X,F)$
with respect to \eqref{lm2.0} is denoted by $L^{2}(X,F)=L^{2}(X,F,dv_X,h^F).$
We consider the space of holomorphic $L^{2}$ sections:
\begin{equation}\label{lm2.02a}		
H^{0}_{(2)}(X,F)\coloneqq H^{0}_{(2)}(X,F,dv_X,h^F)=		
\big\{s\in L^{2}(X,F,dv_X,h^F) : \text{$s$ is holomorphic}\big\}\,.		
\end{equation}		
We deduce from the Cauchy estimates for holomorphic functions that
for every compact set
$K\subset X$ there exists $C_K>0$ such that for all 
$s\in H^{0}_{(2)}(X,F)$,
\begin{equation}\label{b1.1}
\sup_{x\in K}|s(x)|\leqslant C_K\|s\|_{L^2}\,.
\quad\text{for all $s\in H^{0}_{(2)}(X,L^p\otimes E)$\,.}
\end{equation}
This implies that $H^{0}_{(2)}(X,F)$ is a closed subspace of
$L^2(X,F)$. Moreover,
$H^{0}_{(2)}(X,F)$ is separable (cf.\ \cite[p.\,60]{Weil:58}).
\begin{definition} \label{almt2.1b}
The \emph{Bergman projection} is the orthogonal projection
\(
P:L^{2}(X,F)\to H^{0}_{(2)}(X,F)\,.
\)
\end{definition}
%--------

\noindent
By \eqref{b1.1}, for a fixed $x\in X$, the evaluation functional $s\mapsto s(x)$ on
$H^{0}_{(2)}(X,F)$
is continuous.
By the Riesz representation theorem, there exists
$P(x,\cdot)\in L^2(X,F_x\otimes F^{*})$
such that
\begin{equation}\label{lm2.01a}
s(x)=\int_{X}P(x,x') s(x') dv_{X}(x')\,,
\quad\text{for all $s \in H^{0}_{(2)}(X,F)$\,.}
\end{equation}
%-----------
\begin{definition} \label{almt2.1c}
The section $P(\cdot,\cdot)$ of
$F\boxtimes F^{*}$ over
$X\times X$ is called the \emph{Bergman kernel} of
$H^{0}_{(2)}(X,F)$.
\end{definition}
%-----------

\noindent
Set $d:= \dim H^{0}_{(2)}(X,F)\in\N\cup\{\infty\}$.
Let $\{s_i\}_{i=1}^{d}$ be any orthonormal basis of
$H^{0}_{(2)}(X,F)$ with respect to the inner
product \eqref{lm2.0}. Using the estimate \eqref{b1.1} we can show that
\begin{equation} \label{bk2.4}
P(x,x')= \sum_{i=1}^{d} s_i (x) \otimes (s_i(x'))^*
\in F_x\otimes F_{x'}^*\,,
\end{equation}
where the right-hand side converges on every compact subset of $X$, together with all its derivatives (see e.g.\ \cite[p.\,62]{Weil:58}).
Thus $P(\cdot,\cdot)\in \cC^{\infty}(X\times X,F\boxtimes F^{*})$.
It follows from \eqref{bk2.4} that
%\begin{equation}\label{lm2.01}
$(Ps)(x)=\int_{X}P(x,x') s(x') dv_{X}(x')$,
for all $s\in L^2(X,F)$\,,
%\end{equation}
that is, $P(\cdot,\cdot)$ is the Schwartz kernel
of the Bergman projection $P$.

Let $L$ and $E$ be two holomorphic vector bundles on $X$.
We assume that $L$ is a line bundle.
The bundle $E$ is an auxiliary twisting bundle.
It is interesting to work with a twisting vector bundle $E$ for several reasons.
For example, when one has to deal with $(n,0)$-forms with values
in $L^{p}:=L^{\otimes p}$ for $p\in\N^*$,
one sets $E=\Lambda^{n} (T^{*(1,0)}X)$.
From a physical point of view, the presence of $E$ means a quantization 
of a system with several degrees of internal freedom.
We fix Hermitian metrics $h^L$, $h^E$ on $L$, $E$. 
We denote by $\cC^{\infty}_b(X,\End(E))$ the space of smooth 
sections of $\End(E)$ whose all derivatives are bounded, cf.\ \eqref{eq:C_b}.
 %We denote the space of smooth bounded sections of $\End(E)$ by 
%\[\cC^{\infty}_b(X,\End(E))=\Big\{s\in\cC^{\infty}(X,\End(E)): 
%\sup_{x\in X}|s(x)|_{\End(E)}<\infty\Big\},\]
The following definition introduces one of the main notions of this paper.
%===
\begin{definition}\label{Toepdef}
For a bounded section $f\in\cC^{\infty}_b(X,\End(E))$, set %we denote
%--------------------------------------------------------------------
\begin{equation}\label{toe2.4}
T_{f,\,p}:L^2(X,L^p\otimes E)\longrightarrow L^2(X,L^p\otimes E)\,,
\quad T_{f,\,p}=P_p\,f\,P_p\,,
\end{equation}
%----------------------------------------------------------------------------
where the action of $f$ is the fiberwise action of $f$.
The map that associates $f\in\cC^{\infty}_b(X,\End(E))$
with the family of bounded operators $\{T_{f,\,p}\}_p$ on 
$L^2(X,L^p\otimes E)$ is called
the  {\em Berezin-Toeplitz quantization}.
\end{definition}
%===
Note that $T_{f,\,p}$ is a Carleman operator with a smooth integral kernel given by
\begin{equation}\label{toe2.5}
T_{f,\,p}(x,x')=\int_X P_p(x,x'')f(x'')P_p(x'',x')\,dv_X(x'')\,.
\end{equation}

For two arbitrary bounded sections
$f,g\in \cC^{\infty}_b(X,\End(E))$, it is easy to see that
the composition $T_{f,\,p}T_{g,\,p}$ is not of the form
$T_{fg,\,p}$ in general. But we shall show that we have
$T_{f,\,p}T_{g,\,p}\sim T_{fg,\,p}$ asymptotically for $p\to\infty$.
In order to explain this, we introduce the following
more general notion of a Toeplitz operator.
%---------------------------------------------------------------------------
\begin{definition}\label{toe-def}
A {\em Toeplitz operator}\index{Toeplitz operator}
is a sequence $\{T_p\}_{p\in\N}$ of linear operators
%---------------------------------------------------------------------------
\begin{equation}\label{toe2.1}
T_{p}:L^2(X,L^p\otimes E)\longrightarrow L^2(X,L^p\otimes E)
\end{equation}
Verifying $T_{p}=P_p\,T_p\,P_p$, 
such that there exists a 
sequence $g_\ell\in\cC^\infty_b(X,\End(E))$ 
such that for any $k\geqslant0$, there exists $C_k>0$
%for any $p\in \N$, we have $T_{p}=P_p\,T_p\,P_p$, and
 with
%----------------------------------------------------------------------------
\begin{equation}\label{toe2.3}
\Big\|T_p-\sum_{\ell=0}^kT_{g_\ell,\,p}\, p^{-\ell}\Big\|
\leqslant C_k\,\, p^{-k-1}\quad \text{ for any } p\in \N^*,
\end{equation}
where $\norm{\,\cdot\,}$ denotes the operator norm on the space of
bounded operators.
The section $g_0$ is called the {\em principal symbol} of $\{T_p\}$.
%The full symbol of $\{T_p\}$ is the formal series
%$\sum_{l=0}^\infty \hbar^{l}g_l\in\cC^\infty(X,\End(E))[[\hbar]]$
%and the {\em principal symbol}\index{principal symbol}
%of $\{T_p\}$ is $g_0$.
%If each $T_p$ is self-adjoint, $\{T_p\}$ is called self-adjoint.
\end{definition}
%---------------------------------------------------------------------
\noindent
We express \eqref{toe2.3} symbolically by
\begin{equation}\label{atoe2.1}
T_p= \sum_{\ell=0}^k T_{g_\ell,p}\, p^{-\ell} 
+ \mO(p^{-k-1})\,,\:\:p\to\infty.
\end{equation}
If \eqref{toe2.3} holds for any $k\in \N$, then we write
\eqref{atoe2.1} with $k=+\infty$.
%\begin{equation}\label{atoe2.2}
%T_p= \sum_{l=0}^\infty T_{g_l,p}\, p^{-l} + \mO(p^{-\infty}).
%\end{equation}
The Poisson bracket $\{f, g\}$ on $(X, 2\pi \omega)$ is defined by: 
for $f, g\in\cC^{\infty}(X)$,
if $\xi_{f}$ denotes the Hamiltonian vector field generated by 
$f$, which is defined by
$2\pi i_{\xi_{f}}\omega=df$, then
\begin{align}\label{2.17}
\{f, g\}=\xi_{f}(dg).
\end{align}
Endowed with the Poisson bracket, the algebra $\cC^{\infty}(X)$
becomes a Lie algebra.

One of our goals is to show that $T_{f,\,p}T_{g,\,p}$ is 
a Toeplitz operator in the sense of Definition
\ref{toe-def}. This will be achieved by using the asymptotic expansions 
of the Bergman kernel and the kernels of the Toeplitz operators.
%-----
\begin{definition}\label{D:btp}
Let $(X,\Theta)$ be a Hermitian manifold,
and let $(L,h^L)$ and $(E,h^E)$ be holomorphic Hermitian
vector bundles on $X$ of rank one and $r$, respectively.
%-----
We assume that $(L,h^L)$ is positive and denote
by $\omega=\frac{\sqrt{-1}}{2\pi}R^L$ the K\"ahler metric induced by
the curvature of $(L,h^L)$.
%-----
Let $\cA$ be a $\C$-subalgebra of $\cC^\infty(X, \End(E))$ such that
the subalgebra 
$$\cA_\C:=\{f\in\cA:\text{there exists $\widetilde{f}\in\cC^\infty(X)$
with $f=\widetilde{f}\Id_E$}\}$$ 
is a Lie subalgebra of $(\cC^\infty(X),\{\cdot,\cdot\})$,
where $\{\cdot, \cdot\}$ is the Poisson bracket on $(X, 2\pi \omega)$.

We say that the Berezin-Toeplitz package holds 
for the K\"ahler manifold $(X,\omega)$ and algebra $\cA$ with
quantum spaces $H^{0}_{(2)}(X,L^p\otimes E)$
%-----
if the following
statements hold:
\\[2pt]
(i) For any $f,g\in\cA$, the composition $T_{f,\,p}T_{g,\,p}$ 
%of the Toeplitz operators $T_{f,p}$ and $T_{g,p}$ 
admits
the asymptotic expansion
\begin{equation}\label{toe4.2}
T_{f,\,p}T_{g,\,p}=\sum^\infty_{r=0}p^{-r}T_{C_r(f,\,g),\,p}
+\mO(p^{-\infty})\,,\:\:p\to\infty
\end{equation}
in the sense of \eqref{atoe2.1}, where $C_r$ are bi-differential operators,
$C_0(f,g)=fg$ and $C_r(f,g)\in\cC^\infty(X,\End(E))$,
 especially, $\supp (C_r(f,g))\subset \supp(f)$ $\cap \supp(g)$.
% and $C_0(f,g)=fg$.
\\[2pt]
(ii) If $f,g\in\cA_\C$, then %$f,g\in\cC^{\infty}_0(X,\C)$, then
we have
\begin{equation}\label{toe4.4}
\big[T_{f,\,p}\,,T_{g,\,p}\big]=
\frac{\sqrt{-1}}{\, p}T_{\{f,\,g\},\,p}+\mO(p^{-2})\,,\:\:p\to\infty.
\end{equation}
%-----
\\[2pt]
(iii) For every $f\in\cA$, let us denote by 
${\norm f}_\infty\coloneqq\sup\{|f(x)(u)|_{h^{E}}/ |u|_{h^E}:
x\in X, 0\neq u\in E_x\}$
%\begin{equation}\label{toe4.17sup}
%{\norm f}_\infty
%:=\sup_{\stackrel{ x\in X}{0\neq u\in E_x}} |f(x)(u)|_{h^{E}}/ |u|_{h^{E}}.
%\end{equation}
Then for any $f\in\cA$, there exists $C>0$ such that
the norm of $T_{f,\,p}$ satisfies
\begin{equation}\label{toe4.17a}
{\norm f}_\infty-\frac{C}{\sqrt{p}}\leq\norm{T_{f,\,p}}\leq{\norm f}_\infty\,,
\qquad\lim_{p\to\infty}\norm{T_{f,\,p}}={\norm f}_\infty.
\end{equation}
%thus
%\begin{equation}\label{toe4.17a}
%\lim_{p\to\infty}\norm{T_{f,\,p}}={\norm f}_\infty.
%%:=\sup_{\stackrel{ x\in X}{0\neq u\in E_x}} |f(x)(u)|_{h^{E}}/ |u|_{h^{E}}.
%\end{equation}
\\[2pt]
(iv) The coefficients $C_{r}(f,g)$ can be algorithmically computed
in terms of the geometric data $\Theta, h^L, h^E$.
%(iv) The coefficients $C_{r}(f,g)$ are given by
%$C_{r}(f,g)=C_{r,\,\om}(f,g)$, where $\om=\frac{\imat}{2\pi}R^{L}$
%{\rm(}compare \cite[(0.30)]{MM12} %\eqref{bk2.951}
%{\rm)}.
\end{definition}

Let $\cC_{\rm const}^{\infty}(X, {\rm End}(E))$ denote the algebra 
of smooth sections of $X$ that are constant outside a compact set;
that is, $f\in\cC_{\rm const}^{\infty}(X, {\rm End}(E))$ if there exists
a compact set $K\subset X$ such that $f=C\Id_{E_x}$ for all
$x\in X\setminus K$.
For any $f\in \cC_{\rm const}^{\infty}(X, {\rm End}(E))$, we
consider the Toeplitz operator $T_{f, p}$ as in (\ref{toe2.4}):
\begin{align}
T_{f,\,p}:L^2(X,L^p\otimes E)\longrightarrow L^2(X,L^p\otimes E)\,,
\quad T_{f,\,p}=P_p\,f\,P_p\,.
\end{align}
%===
Let $\Omega^{0,{\bullet}}(X,F)$ be the space of $(0,q)$-forms
over $X$ with values in $F$.
We denote by $\Omega_0^{0,{\bullet}}(X,F)$		
the subspace of $\Omega^{0,{\bullet}}(X,F)$ consisting of elements		
with compact support.		
%-------

The Dolbeault operator acting on sections of the holomorphic vector
bundle $F$ gives rise to the Dolbeault complex
\(\big(\Omega^{0,\bullet}(X,F), 
\overline{\partial}^{\smash{\scriptscriptstyle{F}}}\big)\,.\)
We denote by $\overline{\partial}^{\smash{\scriptscriptstyle{F}},*}$ the formal adjoint 
of $\overline{\partial}^{\smash{\scriptscriptstyle F}}$
with respect to the $L^2$-inner product \eqref{lm2.0}.
Set
\begin{align}\label{lm2.1}
%\begin{split}
D = \sqrt{2}\big(\, \overline{\partial}^{\smash{\scriptscriptstyle F}}
+ \,\overline{\partial}^{\smash{\scriptscriptstyle{F}},*}\,\big)
\,,\quad
\square^{F} = \tfrac{1}{2}D^2 =\overline{\partial}^{\smash{\scriptscriptstyle F}}\,
\overline{\partial}^{\smash{\scriptscriptstyle{F}},*}
+\,\overline{\partial}^{\smash{\scriptscriptstyle{F}},*}\,
\overline{\partial}^{\smash{\scriptscriptstyle F}}.
%\end{split}
\end{align}
The operator $\square^{F}$ is called the \emph{Kodaira Laplacian}.
It acts on
$\Omega ^{0,\bullet}(X,F)$ and preserves its $\Z$-grading.
%write
Let us denote by $\Omega^{0,q}_{(2)}(X,F)
:=L^2(X,\Lambda^{q}(T^{*(0,1)}X)\otimes F)$.

In the following, we consider the maximal extensions of the operator 
$\overline{\partial}^{\smash{\scriptscriptstyle F}}$
in the $L^2$-spaces and denote by 
$\overline{\partial}^{\smash{\scriptscriptstyle{F}},*}$ its Hilbert-space adjoint.
The operator defined by
\begin{equation}\label{ell-}
\begin{split}
\Dom(\square^{F})=\big\lbrace u\in\Dom(\overline{\partial}^{\smash{\scriptscriptstyle F}})
\cap\Dom(\overline{\partial}^{\smash{\scriptscriptstyle{F}},*})\!:
\overline{\partial}^{\smash{\scriptscriptstyle F}}u\in\Dom(\overline{\partial}^{\smash{\scriptscriptstyle{F}},*})\,,\;\overline{\partial}^{\smash{\scriptscriptstyle{F}},*} u\in 
\Dom(\overline{\partial}^{\smash{\scriptscriptstyle F}})\big\rbrace\,,\\
\square^{F} u=\overline{\partial}^{\smash{\scriptscriptstyle F}} 
\overline{\partial}^{\smash{\scriptscriptstyle{F}},*} u+
\overline{\partial}^{\smash{\scriptscriptstyle{F}},*}
\overline{\partial}^{\smash{\scriptscriptstyle F}} u\;\;\text{for}\;u\in\Dom(\square^{F})\,.
\qquad\qquad\qquad\qquad
\end{split}
\end{equation}
is a self-adjoint extension of the Kodaira Laplacian, 
called the Gaffney extension (see \cite[Proposition 3.1.2]{MM07}).
The quadratic form associated with $\square^{F}$ is the form $Q$ given by
\begin{equation}\label{ell2,1}
\begin{split}
&\Dom(Q):=\Dom(\overline{\partial}^{\smash{\scriptscriptstyle F}})\cap
\Dom(\overline{\partial}^{\smash{\scriptscriptstyle{F}},*}), \\
Q(s_1,s_2)=&\langle  \overline{\partial}^{\smash{\scriptscriptstyle F}} s_1,
\overline{\partial}^{\smash{\scriptscriptstyle F}} s_2\rangle
+\langle  \overline{\partial}^{\smash{\scriptscriptstyle{F}},*}s_1,
\overline{\partial}^{\smash{\scriptscriptstyle{F}},*}s_2\rangle\,,
\quad\text{for  } s_1,s_2\in \Dom(Q).
\end{split}
\end{equation}
In our situation, we consider $F=L^p\otimes E$
the operators from \eqref{lm2.1} and their extensions by
$D_p$ and $\Box_p$. 
%===
\begin{definition}[spectral gap]\label{specgapdef}
Let $(X,\Theta)$ be a Hermitian manifold,		
let $(L,h^L)$ and $(E,h^E)$ be holomorphic Hermitian		
vector bundles on $X$ of rank one and $r$, respectively.		
We say that the Kodaira-Laplacian has a \emph{spectral gap} if
there exist $C_0$, $C_L>0$ such that for any $p\in \N^*$,
\begin{equation}\label{bk1.4}
\begin{split}
&\norm{D_{p}s}^2_{L^2}\geqslant(2C_0 p -C_L)\norm{s}^2_{L^2}\,,\\[2pt]
&\qquad s\in \Dom\Big(\overline{\partial}^{\smash{L^p\otimes E}}\Big)\cap\,
\Dom\Big(\overline{\partial}^{\smash{L^p\otimes E,*}}\Big)\cap\, 
\Omega^{0,1}_{(2)}(X,L^p\otimes E).
\end{split}
\end{equation}
\end{definition}
%-----------------------------------------------------------
The following result generalizes the expansion of the 
Bergman kernel \cite[Theorem 4.1.1]{MM07}
and the Berezin-Toeplitz package \cite[Theorem 7.4.1]{MM07}, 
\cite[Theorem 1.1]{MM08b} 
from the case of compact
manifolds 
in the situation where a spectral gap exists.
%-----------------------------------------------------------
\begin{theorem}\label{t2.1}
Let $(X,\Theta)$ be a Hermitian manifold,		
let $(L,h^L)$ and $(E,h^E)$ be holomorphic Hermitian		
vector bundles on $X$ of rank one and $r$, respectively.		
We assume that $(L,h^L)$ is positive and denote		
by $\omega=\frac{\sqrt{-1}}{2\pi}R^L$.		
%-----		
Assume that the Kodaira-Laplacian possesses a spectral gap, 
as stated in \eqref{bk1.4}.
% i.\,e., that
%there exist $C_0$, $C_L>0$ such that for any $p\in \N^*$,
%\begin{equation}\label{bk1.4}
%\norm{D_{p}s}^2_{L^2}\geqslant(2C_0 p -C_L)\norm{s}^2_{L^2}\,,
%\quad s\in \Dom(\overline{\partial}^{L^p\otimes E})\cap\,
%\Dom(\overline{\partial}^{L^p\otimes E,*})\cap\, 
%\Omega^{0,1}_{(2)}(X,L^p\otimes E).
%\end{equation}
Then we have the following two statements:

\noindent
(1) The Bergman kernel asymptotics for 
$H^0_{(2)}(X,L^p\otimes E)$
holds on compact sets of $X$. More precisely,
there exist coefficients ${\bb}_{r}\in \cC^{\infty}(X, {\rm End}(E))$, 
$r\in \N$, such that for any
compact set $K\subset X$ and any $k, l\in \N$, there exists 
$C_{k, l, K}>0$ such that
\begin{align}\label{2.21}
\Big|\frac{1}{p^{n}}P_{p}(x, x)-\sum^{k}_{r=0}{\bb}_{r}(x)p^{-r}
\Big|_{\cC^{l}(K)}
\leqslant
C_{k, l, K} p^{-k-1},
\end{align}
where ${\bb}_{0}={\rm det}(\frac{\dot{R}^{L}}{2\pi})\Id_E$ and 
$\dot{R}^{L}\in {\rm End}(T^{(1, 0)}X)$
is defined by for $W, Y\in T^{(1, 0)}X$,
\begin{align}
R^{L}(W, \overline{Y})=\langle \dot{R}^{L}W, \overline{Y}\rangle.
\end{align}

\noindent
(2) The Berezin-Toeplitz quantization package holds for the K\"ahler manifold
$(X,\omega)$, the algebra
$\cC^\infty_{\rm const}(X,\End(E))$, and quantum spaces
$H^0_{(2)}(X,L^p\otimes E)$ in the sense of Definition \ref{D:btp} with $C_0(f,g)=fg$.
\end{theorem}
%-----------------------------------------------------------

\noindent
The abstract setting above holds for compact Hermitian
manifolds; see \cite[Section 2.7]{MM11}.
\begin{theorem}\label{t2.1comp}
Let $(X,\Theta)$ be a compact Hermitian manifold of dimension $n$, and $(L,h^L)$ 
be a positive line bundle. We set 
$\omega=\frac{\sqrt{-1}}{2\pi}R^L$.
Let $(E,h^E)$ be a holomorphic Hermitian vector bundle on $X$.
Then we have the following two statements:

(1) The Bergman kernel asymptotics for
$H^0(X,L^p\otimes E)$
hold on $X$. 

(2) The Berezin-Toeplitz quantization package holds for the 
K\"ahler manifold $(X,\omega)$, the algebra 
$\cC^\infty(X,\End(E))$, and quantum spaces
$H^0(X,L^p\otimes E)$.
\end{theorem}
%-----------------------------------------------------------

%-----------------------------------------------------------------------------
\begin{remark}\label{toer1}
%\textbf{(i)} 
Relations \eqref{toe4.4} and \eqref{toe4.17a} were first proved in some special cases:
in \cite{KliLe:92} for Riemann surfaces, 
in \cite{Cob:92} for $\C^n$, and in \cite{BLU:93} for bounded symmetric domains in $\C^n$,
by using explicit calculations.
Then Bordemann, Meinrenken, and Schlichenmaier \cite{BMS94}
%(cf. also Guillemin \cite{Guill:95})
treated the case of a compact K{\"a}hler manifold (with $E=\C$)
using the theory of Toeplitz structures (generalized Szeg\H{o} operators) by
Boutet de Monvel and Guillemin \cite{BdMG81}.
%Guillemin \cite{Guill:95} notices that the method is implicit in \cite{BoG}.
Moreover, Schlichenmaier \cite{Schlich:00}
(cf. also  \cite{Cha03,KS01})
continued this train of thought and showed that for any $f,g\in \cC^\infty(X)$,
 the product $T_{f,\,p}\,T_{g,\,p}$ has an asymptotic expansion
 \eqref{toe4.2} and constructed an associative star product geometrically.
\end{remark}
%-------------------------

%---------------------------------------------------------------------------
\subsection{Model situation: Bergman kernel on \texorpdfstring{$\C^n$}{Cn}}\label{toes1}
In this section, we introduce the model operator $\cL$, a Kodaira-Laplace operator acting on 
$\C^n$, and describe its spectrum. We formulate the expansion of Bergman and Toeplitz kernels 
in terms of the Schwartz kernel associated with the projection onto the kernel of $\cL$. 
Our analysis is based on the Fourier expansion with respect to the eigenfunctions of 
this model operator.

Write $\C^n\simeq\R^{2n}$ with real coordinates $Z=(Z_1,\dots,Z_{2n})$ and
complex coordinates $z_j=Z_{2j-1}+\imat Z_{2j}$.
Equip $\C^n$ with the Euclidean metric and the K\"ahler form
\[
\om=\frac{\imat}{2}\sum_{j=1}^n dz_j\wedge d\ov z_j\,.
\]
Its volume form is the Euclidean volume form $dZ=dZ_1\cdots dZ_{2n}$.
Let $(L^2(\R^{2n}),\norm{\,\cdot\,}_{L^2})$ be the space 
of square integrable functions on $\R^{2n}$ with respect to the Lebesgue measure.

Let $0<a_1\leq a_2\leq\ldots\leq a_n$.  
Let $L=\C$ be the trivial holomorphic line bundle on $\C^n$ with canonical section
$\mathbf 1$, endowed with the Hermitian metric
\begin{equation}\label{abk2.65}
|\mathbf 1|_{h^L}(z)=\exp\!\Big(-\frac{1}{4}\sum_{j=1}^na_j|z_j|^2\Big)=:\rho(Z).
%\exp\!\Big(-\frac{\pi}{2}\sum_{j=1}^n|z_j|^2\Big)=:\rho(Z).
\end{equation}
Thus $H^0_{(2)}(\C^n,L)$ identifies with the Segal--Bargmann space
of holomorphic functions square-integrable with respect to $\rho\,dZ$.
It is well-known that $\{z^\beta: \beta\in \N^n\}$
forms an orthogonal basis of this space.

 We introduce the operators of creation 
\[b_i=-2\frac{\partial}{\partial z_i}+\frac{a_i}{2}\,\ov z_i,\]
and annihilation
%\begin{equation}\label{bk2.67a}
%\qquad
\[b_i^{+}=2\frac{\partial}{\partial\ov z_i}+\frac{a_i}{2}\,z_i,\]
%\end{equation}
and the model operator (complex harmonic oscillator)
\begin{equation}\label{bk2.67}
\cL=\sum_{i=1}^n b_i\,b_i^{+}.
\end{equation}
The operator
$\cL$ is the $(0,0)$-part of the Kodaira Laplacian on $(\C^n,L)$
after conjugation by $\rho$, cf.\ \cite{MM07,MM08a}.
It acts as a densely defined self-adjoint operator on 
$(L^2(\R^{2n}),\norm{\,\cdot\,}_{L^2})$.
%===
\begin{theorem}[{\cite[Theorem 4.1.20]{MM07}, \cite[Theorem 1.15]{MM08a}}]\label{bkt2.17}
The spectrum of $\cL$ on $L^2(\R^{2n})$ is given by
\begin{equation}\label{bk2.68}
{\spec}(\cL)=\Big\{ 2\sum_{i=1}^n a_i\alpha_i \,:\, \alpha\in\N ^n\Big\}\,.
%\Big\{ 4\pi |\alpha| \,:\, \alpha\in\N ^n\Big\}\,.
\end{equation}
Each $\lambda\in\spec(\cL)$ is an eigenvalue of infinite multiplicity and
an orthogonal basis of the eigenspace  
of $\lambda=2\sum_{i=1}^n\alpha_i a_i$
is given by
\begin{equation}\label{bk2.69}
B_\lambda=\Big\{b^{\alpha}\big(z^{\beta}\exp\big({-\frac{1}{4}\sum_{i=1}^n
a_i|z_i|^2}\big)\big)\,,\quad\text{with $\beta\in\N^n$}\Big\}\,.
\end{equation}
%\begin{equation}\label{bk2.69}
%b^{\alpha}\big(z^{\beta}\exp\big({-\frac{1}{4}\sum_{i=1}^n
%a_i|z_i|^2}\big)\big)\,,\quad\text{with $\beta\in\N^n$}\,.
%\end{equation}
%an orthogonal basis of the corresponding eigenspace is given by
%\begin{equation}\label{bk2.69}
%B_\lambda=\Big\{b^{\alpha}\big(z^{\beta} e^{-\pi\sum_i |z_i|^2/2}\big):
%\text{$\alpha\in\N^n$ with $4\pi|\alpha|=\lambda$, $\beta\in\N^n$}\Big\}
%\end{equation}
where $b^{\alpha}:=b^{\alpha_1}_1\cdots b^{\alpha_n}_n$. Moreover,
$\bigcup\{B_\lambda:\lambda\in\spec(\cL)\}$
forms a complete orthogonal basis of $L^2(\R^{2n})$.
In particular, an orthonormal basis of
\begin{equation}\label{bk2.70}
\varphi_\beta(z)=\left(\frac{a ^\beta}{(2\pi)^n 2 ^{|\beta|} \beta!}
\prod_{i=1}^n a_i\right)^{1/2}z^\beta
\exp\Big (-\frac{1}{4} \sum_{j=1}^n a_j |z_j|^2\Big )\,,\quad \beta\in\N^n\,.
\end{equation}
%$\Ker (\cL)$ is
%\begin{equation}\label{bk2.70}
%\Big\{\varphi_\beta(z)=\big(\tfrac{\pi ^{|\beta|}}{\beta!}\big)^{1/2}z^\beta
%e^{-\pi\sum_i |z_i|^2/2}\,:\beta\in\N^n\Big\}\,.
%\end{equation}
\end{theorem}
%===
%
%\begin{thm}\label{bkt2.17}
%The spectrum of $\cL$ on $L^2(\R^{2n})$ is given by
%\begin{equation}\label{bk2.68}
%\spec(\cL)=
%\Big\lbrace2\sum_{i=1}^n\alpha_i a_i\,:\, 
% \alpha =(\alpha_1,\cdots,\alpha_n)\in\N ^n\Big\rbrace
%\end{equation}
%and an orthogonal basis of the eigenspace of $2\sum_{i=1}^n\alpha_i a_i$
%is given by
%\begin{equation}\label{bk2.69}
%b^{\alpha}\big(z^{\beta}\exp\big({-\frac{1}{4}\sum_{i=1}^n
%a_i|z_i|^2}\big)\big)\,,\quad\text{with $\beta\in\N^n$}\,.
%\end{equation}
%In particular, an orthonormal basis of 
%$\ke (\cL)$ is
%\begin{equation}\label{bk2.70}
%\varphi_\beta(z)=\left(\frac{a ^\beta}{(2\pi)^n 2 ^{|\beta|} \beta!}
%\prod_{i=1}^n a_i\right)^{1/2}z^\beta
%\exp\Big (-\frac{1}{4} \sum_{j=1}^n a_j |z_j|^2\Big )\,,\quad \beta\in\N^n\,.
%\end{equation}
%
%\end{thm}
%---------------------------------------------------------------------------

Let $\cP:L^2(\R^{2n})\to\Ker(\cL)$ be the orthogonal projection and $\cP(Z,Z')$
its Schwartz kernel (with respect to $dZ'$). Summing \eqref{bk2.70} yields
\begin{equation}\label{toe1.3}
\cP(Z,Z') =\prod_{i=1}^n
\frac{a_i}{2\pi}\:\:\exp\Big(-\frac{1}{4}\sum_{i=1}^n
a_i\big(|z_i|^2+|z^{\prime}_i|^2 -2z_i\overline{z}_i'\big)\Big)\,.
\end{equation}
%\begin{equation}\label{toe1.3}
%\cP(Z,Z')=\exp\!\Big(-\frac{\pi}{2}\sum_{i=1}^n\big(|z_i|^2+|z_i'|^2-2z_i\ov z_i'\big)\Big).
%\end{equation}

%---------------------------------------------------------------------------
%{\color{blue}
\subsection{Asymptotic expansion of Bergman kernel}\label{s3.2}
Let $(X,\Theta)$ be a Hermitian manifold, and let $(L,h^L)$ and $(E,h^E)$ be
holomorphic Hermitian vector bundles on $X$ of rank one and $r$, respectively.
We use the identifications and notations to state the asymptotics.

\noindent
\textbf{\emph{Normal coordinates.\/}}
For $x\in X$, let $a^X_x$ be the injectivity radius of $(X,g^{TX})$ at $x$. 
Denote by $B^{X}(x,\var)$ and $B^{T_xX}(0,\var)$ the open balls in $X$
and $T_xX$, respectively. We identify them via the exponential map 
$Z\mapsto\exp^X_x(Z)$ for $\var\le a^X_x$. For any subset $Y\subset X$,
we set $a^Y=\inf_{x\in X}a^X_x$. 
Throughout the paper, $\varepsilon\in\,]0,a^X_x/4[$.
Let $d(\cdot,\cdot)$ denote the Riemannian distance function associated with the Riemannian manifold $(X, g^{TX})$.

\noindent
\textbf{\emph{Basic trivialization.\/}}
Fix $x_0\in X$. For $Z\in B^{T_{x_0}X}(0,\var)$, we identify
$(L_Z,h^L_Z)$ and $(E_Z,h^E_Z)$ with $(L_{x_0},h^L_{x_0})$ and $(E_{x_0},h^E_{x_0})$
by parallel transport along
\(
\gamma_Z:[0,1]\ni u\longmapsto \exp^X_{x_0}(uZ),
\)
and similarly for $L^p\otimes E$ using $\nabla^{L^p\otimes E}$.
With this identification, a function $f\in \cC^\infty(X,\End(E))$ corresponds to
\(
f_{x_0}:B^{T_{x_0}X}(0,\var)\to \End(E_{x_0})\), %\qquad
\(f_{x_0}(Z)=f\circ\exp^X_{x_0}(Z)\),
and the Bergman kernel $P_p(x,x')$ induces a family of smooth sections
\(
(Z,Z')\longmapsto P_{p,x_0}(Z,Z')\in \End(E_{x_0})\),
\(|Z|,|Z'|<\var,\)
depending smoothly on $x_0$.

\noindent
\textbf{\emph{Coordinates on $T_{x_0}X$.\/}}
Let $\{w_i\}_{i=1}^n$ be an orthonormal basis of $T^{(1,0)}_{x_0}X$ and set
$e_{2j-1}=\tfrac{1}{\sqrt2}(w_j+\ov w_j)$,
$e_{2j}=\tfrac{\imat}{\sqrt2}(w_j-\ov w_j)$.
We use real coordinates $Z=(Z_1,\dots,Z_{2n})$ on $T_{x_0}X\simeq\R^{2n}$ via
%\begin{equation}\label{n11}
\((Z_1,\ldots,Z_{2n}) \longmapsto \sum_i Z_i e_i\in T_{x_0}X\),
%\end{equation}
and complex coordinates $z=(z_1,\ldots,z_n)$ on $\C^n\simeq\R^{2n}$.

\noindent
\textbf{\emph{Volume form on $T_{x_0}X$.\/}}
Let $dv_{TX}$ be the Riemannian volume form on $(T_{x_0}X,g^{T_{x_0}X})$.
Then there exists a smooth positive function $\kappa_{x_0}$ such that
\begin{equation}\label{atoe2.7}
dv_X(Z)=\kappa_{x_0}(Z)\,dv_{TX}(Z),\qquad \kappa_{x_0}(0)=1.
\end{equation}

\noindent
\textbf{\emph{Sequences of operators.\/}}
Let $\Theta_p:L^2(X,L^p\otimes E)\to L^2(X,L^p\otimes E)$ be a sequence of
continuous operators with smooth kernel $\Theta_p(\cdot,\cdot)$ with respect to
$dv_X$ (e.g.\ $\Theta_p=T_{f,p}$). In the basic trivialization we write the
corresponding kernels as $\Theta_{p,x_0}(Z,Z')$. We say that
$\Theta_{p,x_0}(Z,Z')=\mO(p^{-\infty})$ if for any $l,m\in\N$ there exists
$C_{l,m}>0$ such that
$|\Theta_{p,x_0}(Z,Z')|_{\cC^m(X)}\le C_{l,m}p^{-l}$.
The asymptotics will be expressed in terms of the model Bergman kernel
$\cP_{x_0}$ of the operator $\cL$ on $T_{x_0}X\simeq\R^{2n}$ (cf.\ \eqref{toe1.3}).

%===
\begin{notation}\label{noe2.7}
Fix $k\in\N$ and let $\{Q_{r,x_0}\in \End(E)_{x_0}[Z,Z']:\ 0\le r\le k,\ x_0\in X\}$
be a family of polynomials in $Z,Z'$, smooth in $x_0$.
Let $K\subset X$ be compact and $\var'\in\,]0,a^K[$.
We write
\begin{equation}\label{toe2.7}
p^{-n} \Theta_{p,x_0}(Z,Z')\cong \sum_{r=0}^k
(Q_{r,x_0}\cP_{x_0})(\sqrt{p}Z,\sqrt{p}Z')\,p^{-r/2}
+\mO(p^{-(k+1)/2})
\end{equation}
on $\{(Z,Z')\in TX\times_K TX:\ |Z|,|Z'|<\var'\}$ if there exist $C_0>0$ and a
decomposition
\begin{equation}\label{toe2.71}
\begin{split}
p^{-n}\Theta_{p,x_0}(Z,Z')&-\sum_{r=0}^k
(Q_{r,x_0}\cP_{x_0})(\sqrt{p}Z,\sqrt{p}Z')\,\kappa_{x_0}^{-1/2}(Z)
\kappa_{x_0}^{-1/2}(Z')\,p^{-r/2}\\
&=\Psi_{p,k,x_0}(Z,Z')+\mO(p^{-\infty}),
\end{split}
\end{equation}
where for every $l\in\N$ there exist $C_{k,l}>0$, $M>0$ such that for all $p\in\N^*$,
\[
|\Psi_{p,k,x_0}(Z,Z')|_{\cC^l(X)}
\le C_{k,l}\,p^{-(k+1)/2}\,(1+\sqrt{p}|Z|+\sqrt{p}|Z'|)^M\,
e^{-C_0\sqrt{p}\,|Z-Z'|},
\]
on $\{(Z,Z')\in TX\times_K TX:\ |Z|,|Z'|<\var'\}$.
\end{notation}

\noindent
\textbf{\emph{The sequence $P_p$.\/}}
Let $K\subset X$ be compact and $\varepsilon\in]0,a^K[$.
Choose $\bbf_{\!\varepsilon}:\R\to[0,1]$ smooth, even, with
$\bbf_{\!\varepsilon}(v)=1$ for $|v|\le \varepsilon/2$ and
$\bbf_{\!\varepsilon}(v)=0$ for $|v|\ge \varepsilon$, and set
\begin{equation}\label{0c3}
F_\varepsilon(a)=\Big(\int_{-\infty}^{+\infty}\bbf_{\!\varepsilon}(v)\,dv\Big)^{-1}
\int_{-\infty}^{+\infty} e^{iva}\,\bbf_{\!\varepsilon}(v)\,dv.
\end{equation}
Then $F_\varepsilon\in\mathcal S(\R)$ is even and $F_\varepsilon(0)=1$.
We first record the far off-diagonal decay.
%------------------------------------------------------------------------------
\begin{theorem}[Off-diagonal expansion]\label{tue16}
Assume that the spectral gap condition \eqref{bk1.4} holds.
Then for any compact set $K\subset X$,
for any $\ell,m\in\N$ and $\var\in\,]0,a^K[$\,, there exists a positive constant
$C_{K,\ell,m,\var}>0$ such that for any $p\geqslant 1$ and $x,x'\in K$, the following
estimate holds:
\begin{equation}\label{0c7}
\left|F_\varepsilon(D_p)(x,x') - P_{p}(x,x')\right|_{\cC ^m(K\times K)}
\leqslant C_{K,\ell,m,\var}\, p^{-\ell}.
\end{equation}
Especially, by setting $D_\varepsilon=\{(x,x')\in X\times X:d(x,x')>\varepsilon\}$, where $d(\cdot,\cdot)$ is the Riemannian distance on $(X,g^{TX})$,
we have
\begin{equation}\label{toe2.6a}
|P_p(x,x')|_{\cC^m(K\times K\setminus D_\varepsilon)}
\leqslant C_{K,l,m,\var}\, p^{-\ell}\,.
\end{equation}
The $\cC ^m$ norm in \eqref{0c7} and \eqref{toe2.6a} is induced by
$\nabla^L$, $\nabla^E$, $h^L$, $h^E$, and $g^{TX}$.
\end{theorem}
%-------------------------------------
\begin{proof}
For $a\in \R$, set
%\begin{equation}\label{bk2.10}
\(\phi_p(a)=\mathds{1}_{[\sqrt{p\mu_0},+\infty[}(|a|)\,F_\varepsilon(a)\).
%\end{equation}
By \eqref{bk1.4}, for $p$ large enough,
\begin{equation}\label{bk2.11}
F_\varepsilon(D_p)-P_p=\phi_p(D_p).
\end{equation}
By \eqref{0c3}, for any $m\in\N$ there exists $C_m>0$ such that
%\begin{equation}\label{bk2.12}
\(\sup_{a\in\R} |a|^m|F_\varepsilon(a)|\le C_m\).
%\end{equation}
Using elliptic estimates for $D_p^2$ as in \cite{MM07} (cf.\ \cite[Lemma 1.6.2]{MM07})
and Sobolev norms defined on a finite cover of $K$ by geodesic balls, one obtains:
for $l,m'\in\N$ there exists $C_{l,m'}>0$ such that for $p\ge1$,
%\begin{equation}\label{bk2.13}
\(\|D_p^{m'}\phi_p(D_p)Qs\|_{L^2}\le C_{l,m'}\,p^{-l+2m}\,\|s\|_{L^2}\),
%\end{equation}
for any differential operator $Q$ of order $m$ with compact support in a chart and scalar principal symbol.
Combining this with the elliptic estimates (as in \cite{MM07}) yields, for
differential operators $P,Q$ of orders $m',m$ with compact supports in charts,
%\begin{equation}\label{bk2.14}
\(\|P\phi_p(D_p)Qs\|_{L^2}\le C_l\,p^{-l}\,\|s\|_{L^2}\).
%\end{equation}
By the Sobolev inequality and \eqref{bk2.11} we obtain \eqref{0c7}.
Finally, by finite propagation speed \cite[Theorem D.2.1]{MM07},
$F_\varepsilon(D_p)(x,x')$ depends only on the restriction of $D_p$ to
$B^X(x,\varepsilon)$ and vanishes if $d(x,x')\ge\varepsilon$, hence \eqref{toe2.6a}.
\end{proof}
%-------------------------------------

%Next we formulate the near off-diagonal expansion.

%-------------------------------------
\begin{theorem}[Near off-diagonal expansion]\label{tue17}
There exists a family of polynomials
in $Z,Z'$ with the same parity as $r$,
\[\{J_{r,\,x_{0}}\in \End(E)_{x_0}[Z,Z']:r\in\N,x_0\in X\},\]
with the following property.
Assuming that the spectral gap condition \eqref{bk1.4} holds,
then for any compact set $K\subset X$, for any $k\in \N$, and any
$\varepsilon\in\,]0, a^K/4[$\,,
we have
\begin{equation} \label{toe2.9}
p^{-n} P_{p,\,x_0}(Z,Z^\prime)\cong \sum_{r=0}^k
(J_{r,\,x_0} \cP_{x_0})(\sqrt{p}Z,\sqrt{p}Z^{\prime})p^{-\frac{r}{2}}
+\mO(p^{-\frac{k+1}{2}})\,,
\end{equation}
on the set $\{(Z,Z^\prime)\in TX\times_K TX:\abs{Z},\abs{Z^{\prime}}<2\var\}$,
in the sense of Notation \ref{noe2.7}.
\end{theorem}
%-------------------------------------
\begin{proof}
We only sketch the proof; details are in \cite[Proposition 4.1, Theorem 4.18$^\prime$]{DLM06}
and \cite[Proposition 4.1.1, Theorem 4.1.24]{MM07}.
Although $X$ is assumed compact in \cite{DLM06}, the arguments use only the spectral gap,
so the localization applies verbatim here. One pulls back the geometric data by the
exponential map to $T_{x_0}X\simeq\R^{2n}$, extends them suitably, and compares with the
model kernel of $\cL$ from Section~\ref{toes1}. 
The conclusion follows from the spectral gap, the rescaling procedure, and functional analytic techniques inspired 
by Bismut–Lebeau \cite[\S 11]{BL91}.
\end{proof}

%Until the end of Section~\ref{s3.2}, we assume $\Theta=\om$.
Setting $\bb_r(x_0)=(J_{2r,x_0}\cP_{x_0})(0,0)$, \eqref{toe2.9} implies the diagonal
expansion:

%-------------------------------------
\begin{theorem}[On-diagonal expansion]\label{bkt2.18}
Assume that the spectral gap condition \eqref{bk1.4} holds.
Then, for any compact set $K\subset X$,
and for any $k,\,\ell\in\N$, there exists a positive constant
$C_{K,k,\ell}>0$ such that for any $p\geqslant 1$, 
the following
estimate holds:
\begin{equation}\label{bk2.6}
\Big |P_p(x,x)- \sum_{r=0}^{k} \bb_r(x) p^{n-r}
\Big |_{\cC^\ell(K)} \leqslant C_{K,k,\ell}\, p^{n-k-1}\,,
\end{equation}
where ${\bb}_{0}=\det(\frac{\dot{R}^{L}}{2\pi})\Id_E$
and
\begin{equation} \label{toe4.131}%{abk2.7}
\bb_1 = \frac{1}{8\pi}\det\Big(\frac{\dot{R}^L}{2\pi}\Big)
\Big[r^X_\om -2 \Delta_\om \Big(\log({\det} (\dot{R}^L))\Big)
+ 4  \sqrt{-1} \Lambda_\om(R^E) \Big],
\end{equation}
where  $r^X_\om$, $\Delta_\om$ are the scalar curvature and 
the Bochner Laplacian associated with $g^{TX}_{\om}$.
\end{theorem}
%-------------------------------------
% The existence of \eqref{bk2.6} and the leading term were proved in \cite{Tia:90,Catlin,Zel98}. 
The coefficients $\bb_r$ can be computed from
$J_{r,x_0}$ and equivalently from the operators $\cF_{r,x_0}$ with kernels
\begin{equation}\label{bk2.24}
\cF_{r,\,x_{0}}(Z,Z')= J_{r,\,x_{0}}(Z,Z')\,\cP(Z,Z')
\end{equation}
with respect to $dZ'$. Following \cite{MM07,MM08a}, one rescales the Kodaira-Laplacian and performs resolvent analysis. For $s\in\cC^\infty(\R^{2n},E_{x_0})$,
$|Z|\le 2\var$, and $t=1/\sqrt p$, set
\begin{equation}\label{bk2.21}
(S_t s)(Z):=s(Z/t),\qquad
\cL_t:=S_t^{-1}\kappa^{1/2}\,t^2(2\Box_p)\kappa^{-1/2}S_t.
\end{equation}
Then \cite[Theorem 4.1.7]{MM07} gives second order operators $\mO_r$ such that,
as $t\to0$,
\begin{equation}\label{bk2.22}
\cL_t=\cL_0+\sum_{r=1}^m t^r\mO_r+\cO(t^{m+1}),
\end{equation}
and by \cite[Theorems 4.1.21, 4.1.25]{MM07},
%\begin{equation}\label{bk2.30}
$\cL_0=\sum_j b_j b_j^+=\cL$. 
%\qquad $\mO_1=0$.
%\end{equation}

\noindent
\textbf{\emph{Resolvent analysis.\/}}
Define recursively $f_r(\lambda)\in\End(L^2(\R^{2n},E_{x_0}))$ by
\begin{equation}\label{bk2.23}
f_0(\lambda)=(\lambda-\cL_0)^{-1},\qquad
f_r(\lambda)=(\lambda-\cL_0)^{-1}\sum_{j=1}^r \mO_j\,f_{r-j}(\lambda).
\end{equation}
Let $\delta$ denote the positively oriented circle centered at 
the origin with a sufficiently small radius.
Then \cite[(1.110)]{MM08a} (cf.\ also \cite[(4.1.91)]{MM07}) gives
\begin{equation}\label{bk2.77}
\cF_{r,\,x_{0}}=\frac{1}{2\pi\sqrt{-1}}\int_\delta f_r(\lambda)\,d\lambda.
\end{equation}
Assume now that $\omega=\Theta$. Since $\spec(\cL)$ is explicit, one obtains (with $\cP^\bot=\Id-\cP$)
\begin{equation}\label{bk2.31}
\begin{split}
&\cF_{0,\,x_{0}}=\cP,\quad \cF_{1,\,x_{0}}=0, \qquad
\cF_{2,\,x_{0}}=-\cL^{-1}\cP^\bot\mO_2\cP-\cP\mO_2\cL^{-1}\cP^\bot,
%\\&\cF_{3,\,x_{0}}=-\cL^{-1}\cP^\bot\mO_3\cP-\cP\mO_3\cL^{-1}\cP^\bot.
\end{split}
\end{equation}
In particular,
\begin{equation}\label{bk2.33}
J_{0,\,x_{0}}=1,\qquad 
J_{1,\,x_{0}}=0.
\end{equation}
% To state $\bb_1$ we recall that $\nabla^{TX}$ denotes the Levi--Civita connection,
% with curvature $R^{TX}=(\nabla^{TX})^2$, and $r^X$ its scalar curvature.
%--------------------------------
%\begin{theorem}\label{toet4.5}
% We have
%\begin{equation}\label{toe4.131}
% \bb_{1} = \frac{1}{8\pi}r^X + \frac{\imat}{2\pi}R^E_{\Lambda}\,.
% \end{equation}
%\end{theorem}
%--------------------------------

The coefficients $\bb_1$ and $\bb_2$ were computed by Lu \cite{Lu:00} (for $E=\C$),
X.~Wang \cite{Wang05}, and L.~Wang \cite{Wangl03}, in various generalities, using peak
sections and H\"ormander's $L^2$ $\ov\partial$-method as in \cite{Tia:90}.
Dai--Liu--Ma computed $\bb_1$ via the heat kernel \cite[\S 5.1]{DLM06}; see also
\cite[\S 2]{MM08a}, \cite[\S 2]{MM06} and \cite[\S 4.1.8, \S 8.3.4]{MM07} for the
symplectic case. A new method for $\bb_2$ was given in \cite{MM12}.

\subsection{Calculus and expansion of Toeplitz kernels} \label{s3.3}
We derive the calculus of Toeplitz kernels from the Bergman kernel expansion
\eqref{toe2.9} together with the Taylor expansion of the symbol. This reduces the
problem to a kernel calculus on $\C^n$ for kernels of the form $F\cP$, where $F$ is
a polynomial and $\cP$ is the model Bergman kernel \eqref{toe1.3}.
For $F\in\C[Z,Z']$ we denote by $F\cP$ the operator on $L^2(\R^{2n})$ with kernel
$F(Z,Z')\cP(Z,Z')$ with respect to the volume form $dZ$ (cf.\ \eqref{toe2.5}).
The following lemma \cite[Lemma 7.1.1]{MM07} summarizes the corresponding kernel
calculus.
%---------------------------------------------------------------------
\begin{lemma}\label{toet1.1}
For any
$F,G\in\C[Z, Z^{\prime}]$ there exists a polynomial
$\cK[F,G]\in\C[Z, Z^{\prime}]$
with degree $\deg\cK[F,G]$ of the same parity as
$\deg F+\deg G$, such that
\begin{equation}\label{toe1.6}
((F\cP) \circ (G\cP))(Z, Z^{\prime}) =
\cK[F,G](Z, Z^{\prime}) \cP( Z, Z^{\prime}).
\end{equation}
\end{lemma}
%---------------------------------------------------------------------
\noindent
For $F\in\C[Z,Z']$ we denote by $(F\cP)_p$ the operator with kernel
$p^{n}(F\cP)(\sqrt{p}Z,\sqrt{p}Z^{\prime})$, i.e.
\[
((F\cP)_p \varphi)(Z)=\int_{\R^{2n}}p^{n}
(F\cP)(\sqrt{p}Z, \sqrt{p}Z^{\prime})\varphi(Z^{\prime})\,dZ^{\prime},\qquad
\varphi\in L^2(\R^{2n}).
\]
Then a change of variables gives, for $F,G\in\C[Z,Z']$,
\begin{align}\label{toe1.15}
((F\cP)_p\circ (G\cP)_p)(Z,Z^\prime)
= p^{n}((F\cP)\circ (G\cP))(\sqrt{p}Z, \sqrt{p}Z^{\prime}).
\end{align}

\medskip
\noindent
We now turn to Toeplitz operators $T_{f,p}=P_p f P_p$ with
$f\in\cC^\infty_{\rm const}(X,\End(E))$. As a first step, we show that
their kernels decay
rapidly away from the diagonal.
%---------------------------------------------------------------------
\begin{lemma}[{\cite[Lemma\,4.2]{MM08b}}] \label{toet2.1}
Assume that the spectral gap condition \eqref{bk1.4} holds
and let $f\in\cC^\infty_{\rm const}(X,\End(E))$.
Then for any compact set $K\subset X$,
for any $l,m\in\N$ and $\var\in\,]0,a^K[$, there exists
$C_{K,l,m,\var}>0$ such that for all $p\geqslant 1$,
%-------------------------------------------------------------------------
\begin{equation} \label{toe2.6b}
|T_{f,\,p}(x,x')|_{\cC^m(K\times K\setminus D_\varepsilon)}\leqslant C_{K,l,m,\varepsilon}p^{-l}
\end{equation}
%---------------------------------------------------------------------
where $D_\varepsilon=\{(x,x')\in X\times X:d(x,x')<\varepsilon\}$ and the $\cC^m$-norm is induced by $\nabla^L,\nabla^E$ and $h^L,h^E,g^{TX}$.
\end{lemma}
%---
\begin{proof}
By \eqref{toe2.5} it suffices to combine the far off-diagonal estimate
\eqref{toe2.6a} for $P_p$ with the fact that on $K\times K$ the kernel
$P_p(x,x')$ has at most polynomial growth in $p$ in $\cC^m$ (which follows from
\eqref{toe2.9}). This yields \eqref{toe2.6b}.
\end{proof}

The near off-diagonal expansion \eqref{toe2.9} and Lemma~\ref{toet1.1} imply the
corresponding expansion for Toeplitz kernels (cf.\ \cite[Lemma\,4.6]{MM08b},
\cite[Lemma\,7.2.4]{MM07}).
%===
\begin{theorem} \label{toet2.3}
Assume that the spectral gap condition \eqref{bk1.4} holds.
Let $f\in\cC^\infty_{\rm const}(X,\End(E))$.
There exists a family \[\{Q_{r,\,x_0}(f)\in\End(E)_{x_0}[Z,Z^{\prime}]:r\in\N,\,x_0\in X\}\,,\]
depending smoothly on the parameter $x_0\in X$, where
$Q_{r,\,x_0}(f)$
are polynomials with the same parity as $r$ with the following property. For any compact set $K\subset X$, for any $k\in \N$, and any
$\varepsilon\in\,]0, a^K/4[$\,,
we have
%===
\begin{equation} \label{toe2.13}
p^{-n}T_{f,\,p,\,x_0}(Z,Z^{\prime})
\cong \sum^k_{r=0}(Q_{r,\,x_0}(f)\cP_{x_0})(\sqrt{p}Z,\sqrt{p}Z^{\prime})
p^{-r/2} + \mO(p^{-(k+1)/2})\,,
\end{equation}
%===
on the set $\{(Z,Z^\prime)\in TX\times_K TX:\abs{Z},\abs{Z^{\prime}}<2\var\}$,
in the sense of Notation \ref{noe2.7}. Moreover,
  $Q_{r,\,x_0}(f)$ are expressed by
  \begin{equation} \label{toe2.14}
  Q_{r,\,x_0}(f) = \sum_{r_1+r_2+|\alpha|=r}
    \cK\Big[J_{r_1,\,x_0}\;,\;
  \frac{\partial ^\alpha f_{\,x_0}}{\partial Z^\alpha}(0)
  \frac{Z^\alpha}{\alpha !} J_{r_2,\,x_0}\Big]\,.
  \end{equation}
%===
where $\cK[\cdot,\cdot]$ was introduced in \eqref{toe1.6}. We have for any $x_0\in X$,
  \begin{align} \label{toe2.15}
  Q_{0,\,x_0}(f)= f(x_0)\in\End(E_{x_0}) .
  \end{align}
  \end{theorem}
%===
\begin{proof}
By \eqref{toe2.5} and Lemma~\ref{toet2.1}, for $|Z|,|Z'|<\var/2$ the kernel
$T_{f,p,x_0}(Z,Z')$ is determined up to $\cO(p^{-\infty})$ by the restriction of
$P_p$ and $f$ to a fixed neighborhood of $x_0$.
Choose $\rho\in\cC^\infty(\R)$ even, with
%\begin{equation}\label{alm4.19}
$\rho (v)=1  \  \  {\rm if} \  \  |v|<2$, %\quad 
$\rho (v)=0 \   \   {\rm if} \  |v|>4$.
%\end{equation}
Then for $|Z|,|Z'|<\var/2$,
\begin{equation}\label{toe2.17}
\begin{split}
T_{f,\,p,\,x_0}(Z,Z^{\prime})=\!\!
\int\limits_{{ T_{x_0}X}}&
P_{p,x_0}(Z,Z^{\prime\prime})
\rho\left(\tfrac{2}{\var}|Z^{\prime\prime}|\right)
f_{x_0}(Z^{\prime\prime})P_{p,x_0}(Z^{\prime\prime},Z^{\prime})
 \kappa_{x_0}(Z^{\prime\prime})
\,dv_{TX}(Z^{\prime\prime})
+\cO(p^{-\infty})\,.
\end{split}
\end{equation}
Expand $f_{x_0}$ at $0$ and rescale:
\begin{equation}\label{toe2.18}
\begin{split}
f_{x_0}(Z)
&=\sum_{|\alpha|\leqslant k}
\frac{\partial^\alpha f_{x_0}}{\partial Z^\alpha}(0)
\frac{ Z^\alpha}{\alpha!}+\cO(|Z|^{k+1})\\
&=\sum_{|\alpha|\leqslant k} p^{-|\alpha|/2}
\frac{\partial^\alpha f_{x_0}}{\partial Z^\alpha}(0)
\frac{ (\sqrt{p}Z)^\alpha}{\alpha!}
+p^{-\frac{k+1}{2}}\cO(|\sqrt{p}Z|^{k+1}).
\end{split}
\end{equation}
Insert \eqref{toe2.18} and the Bergman kernel expansion \eqref{toe2.9} into
\eqref{toe2.17}, taking into account the $\kappa_{x_0}$-factors in
\eqref{toe2.71}. After the change of variables $\sqrt{p}\,Z''=W$, the resulting
composition reduces to the model calculus \eqref{toe1.6}, \eqref{toe1.15}, which
gives \eqref{toe2.13} and \eqref{toe2.14}. Finally, \eqref{bk2.33} and
\eqref{toe2.14} yield \eqref{toe2.15}:
$Q_{0,\,x_0}(f)= \cK[1, f_{x_0}(0)] = f_{x_0}(0) = f(x_0)$.
\end{proof}
\begin{corollary}\label{toec2.1}
For any $f\in\cC^\infty_{\rm{const}}(X,\End(E))$, we have
\begin{equation}\label{bk4.2}
T_{f,\,p}(x,x)= \sum_{r=0}^{\infty} \bb_{r,f}(x) p^{n-r}+\cO(p^{-\infty})\,,
\quad \bb_{r,f}\in\cC^\infty(X,\End(E))\,.
\end{equation}
uniformly on compact sets in the $\cC^\infty$ topology, analogous 
to the expansion \eqref{bk2.6}.
\end{corollary}
%--------------------------
\begin{proof}
Take $Z=Z'=0$ in \eqref{toe2.13}. Then \eqref{bk4.2} holds with
$\bb_{r,f}(x)=(Q_{2r,x}(f)\,\mathcal{P}_{x})(0,0)$.
\end{proof}
%--------------------------
Since \eqref{toe2.14} gives $Q_{2r,x}(f)$ explicitly in terms of the $J_{r,x}$,
one can compute the first coefficients $\bb_{r,f}$. In particular, \cite{MM12}
computes $\bb_{1,f}$ and $\bb_{2,f}$; see also applications in K\"ahler geometry
(e.g.\ \cite{Fin10,Fine_Quant_Duke_2012,LM06a}).

%%-------------------------------------
%\begin{theorem}[{\cite[Theorem 0.1]{MM12}}] \label{toet4.6}
%For any $f\in\cC^\infty_{\rm const}(X,\End(E))$, we have{\rm:}
%\begin{equation}\label{bk4.3}
%\bb_{0,\,f}=f, \quad \bb_{1,\,f} = \frac{r^X}{8\pi}  f
%+ \frac{\sqrt{-1}}{4\pi}
%\left(R^E_{\Lambda}  f + fR^E_{\Lambda} \right)
%-  \frac{1}{4\pi} \Delta^{E} f\,.
%\end{equation}
%If $f\in \cC^\infty_{\rm const}(X)$, then
%\begin{equation}\label{abk4.4}
%\begin{split}
%\pi^2 \bb_{2,\,f}  =& \,\pi^2  \bb_{2}  f
%+ \frac{1}{32}\Delta^2 f
%- \frac{1}{32}  r^X \Delta f
%- \frac{\sqrt{-1}}{8} \big\langle \Ric_\om,
%\partial\ov{\partial}f\big\rangle\\
%&+ \frac{\sqrt{-1}}{24} \big\langle df,
%\nabla^E R^E_{\Lambda}\big\rangle_{\om}
% + \frac{1}{24} \big\langle \partial f, \nabla^{1,0*} R^E\big\rangle_{\om}
% - \frac{1}{24} \big\langle \ov{\partial} f,
%\ov{\partial}^{E*} R^E\big\rangle_{\om}\\
%&-  \frac{\sqrt{-1}}{8} (\Delta f)R^E_{\Lambda}
%+ \frac{1}{4}\big\langle \partial \ov{\partial} f, R^E\big\rangle_{\om}\,.
%\end{split}
%\end{equation}
%\end{theorem}

Lemma~\ref{toet2.1} and Theorem~\ref{toet2.3} give the local expansion of the
kernel of $T_{f,p}$. The same method applies to compositions $T_{f,p}T_{g,p}$,
yielding an expansion of the form \eqref{toe2.13}. Conversely, the existence of
such an expansion (with compactly supported kernel) provides a convenient
criterion for a family to be Toeplitz.

\begin{theorem}[Criterion for Toeplitz operators,
compact support version]\label{toet3.1}
Let $$\mT=\{T_p:L^2(X,L^p\otimes E)\longrightarrow L^2(X,L^p\otimes E):p\geq1\}$$
be a family of bounded linear operators with smooth kernels $T_p(\cdot,\cdot)$
satisfying the following conditions:
\\[2pt]
(i) For any $p\in \N$,  $P_p\,T_p\,P_p=T_p$\,.
\\[2pt]
(ii) There exists a compact set $K\subset X$ and a family
$\{S_p:L^2(X,L^p\otimes E)\longrightarrow L^2(X,L^p\otimes E):p\geq1\}$
of operators with smooth
kernels such that
for all $p\in\N$ we have $T_p=P_pS_pP_p$ and
$\supp S_p(\cdot,\cdot)\subset K\times K$ with the following property.
%\\[2pt]
%(iii)
For any $\varepsilon_0>0$ and any $\ell\in\N$,
there exists $C_{\ell,\,\varepsilon_0}>0$ such that
for all $p\geqslant 1$ and all $(x,x')\in K\times K$
with $d(x,x')>\varepsilon_0$,
%---
\begin{equation} \label{toe3.1}
|S_{p}(x,x')|\leqslant C_{\ell,\varepsilon_0}p^{-\ell}.
\end{equation}
%---
(iii) There exists a family of polynomials
$\{\mQ_{r,\,x_0}(\mT)\in\End(E)_{x_0}
[Z,Z^{\prime}]\}_{x_0\in X}$
such that:
\begin{itemize}
\item[(a)]
As a polynomial, each $\mQ_{r,\,x_0}(\mT)$ possesses the same parity as $r$.
\item[(b)] The family is smooth in $x_0\in X$, the
sections $X\ni x_0\longmapsto\mQ_{r,\,x_0}(\mT)(0,0)$
is supported in $K$,
\item[(c)] There exists $0<\var^\prime<a^K/4$ such that for every  $x_0\in K$,
 every $Z,Z^\prime \in T_{x_0}X$ with  $\abs{Z},\abs{Z^{\prime}}<\var^\prime$, 
 and every $k\in\N$ we have
\begin{equation} \label{toe3.2}
p^{-n}T_{p,\,x_0}(Z,Z^{\prime})\cong
\sum^k_{r=0}(\mQ_{r,\,x_0}(\mT)\cP_{x_0})
(\sqrt{p}Z,\sqrt{p}Z^{\prime})p^{-\frac{r}{2}} + \mO(p^{-\frac{k+1}{2}}),
\end{equation}
in the sense of Notation \ref{noe2.7}.
\end{itemize}
Then $\mT=\{T_p:p\geq1\}$ is a Toeplitz operator in the sense of
Definition \ref{toe-def}.
\end{theorem}
\begin{proof}
We define inductively a sequence $(g_l)_{l\geqslant 0}$ with
$g_l\in \cC^\infty_0(X,\End(E))$ such that
\begin{equation}\label{toe3.4}
T_{p} = \sum_{l=0}^m P_{p}\, g_l \, P_{p}\,p^{-l} +\mO(p^{-m-1})\,,
\quad\text{for every $m\geqslant 0$}\,.
\end{equation}

\smallskip
\noindent\textbf{Step 1: construction of $g_0$ and the case $m=0$.}
Fix $x_0\in X$ and set
\begin{equation}\label{toe3.5}
g_0(x_0)=\mQ_{0,\,x_0}(\mT)(0,0)\in\End(E_{x_0})\,.
\end{equation}
By assumption (iii)(b), $g_0(x_0)=0$ for $x_0\notin K$.
We claim that
\begin{equation}\label{toe3.60}
p^{-n}\,(T_p-T_{g_0,p})_{x_0}(Z,Z')\cong \mO(p^{-1})\,,
\end{equation}
which implies
\begin{equation}\label{toe3.6}
T_p=P_p\,g_0\,P_p+\mO(p^{-1}).
\end{equation}

The key analytic input is the spectral gap \eqref{bk1.4}, which gives (as in the
proof of Theorem~\ref{tue16}) the identities/estimates
\begin{align}
&F_\varepsilon(D_p)s = P_p\,s,\qquad p\ge p_0,\ s\in H^0_{(2)}(X,L^p\otimes E), \label{toe5.3a}\\
&\|F_\varepsilon(D_p)-P_p\| = \cO(p^{-\infty}), \label{toe5.3b}\\
&|F_\varepsilon(D_p)(x,x')-P_p(x,x')|_{\cC^m(K\times K)}
\le C_{l,m,\var}\,p^{-l}, \qquad x,x'\in K, \label{toe5.4}
\end{align}
together with finite propagation speed
\cite[\S 2.8]{Tay1:96}, \cite[Appendix D.2]{MM07}
(cf.\ also \cite[Proposition 4.1]{DLM06}), implying that $F_\varepsilon(D_p)(x,\cdot)$
depends only on $D_p|_{B^X(x,\var)}$ and vanishes outside $B^X(x,\var)$.

A crucial point is the following (proved in \cite[p.\ 596-597]{MM08b} by working
with the compactly supported kernel of $F_\varepsilon(D_p)S_pF_\varepsilon(D_p)$
and then using \eqref{toe5.3b}):
%-----------------------------------------------------------------------------
\begin{proposition}[{\cite[Proposition 4.11]{MM08b}}]\label{toet3.2}
In the conditions of Theorem \ref{toet3.1} we have
$\mQ_{0,\,x_0}(\mT)(Z,Z^{\prime})=\mQ_{0,\,x_0}(\mT)(0,0)$
for all $x_0\in X$ and all $Z,Z^{\prime}\in T_{x_0}X$.
\end{proposition}
%-----------------------------------------------------------------------------

We now compare the expansions of $T_p$ and $T_{g_0,p}=P_p\,g_0\,P_p$ near the
diagonal. By \eqref{toe2.13} with $k=1$,
\begin{equation}\label{6.33a}
p^{-n}T_{g_0,\,p,\,x_0}(Z,Z')\cong
\big(g_{0}(x_0)\cP_{x_0}
+Q_{1,\,x_0}(g_0)\cP_{x_0}\,p^{-1/2}\big)(\sqrt{p}Z,\sqrt{p}Z')
+\mO(p^{-1})\,,
\end{equation}
since $Q_{0,x_0}(g_0)=g_0(x_0)$ by \eqref{toe2.15}. On the other hand, the assumed
expansion \eqref{toe3.2} with $k=1$ gives
\begin{equation}\label{6.33b}
p^{-n}T_{p,\,x_0}(Z,Z')\cong
\big(g_{0}(x_0)\cP_{x_0}
+\mQ_{1,\,x_0}(\mT)\cP_{x_0}\,p^{-1/2}\big)(\sqrt{p}Z,\sqrt{p}Z')
+\mO(p^{-1})\,,
\end{equation}
where we used Proposition~\ref{toet3.2} and \eqref{toe3.5}.
Subtracting \eqref{6.33a} from \eqref{6.33b} yields
\begin{equation}\label{6.33d}
p^{-n}(T_{p}-T_{g_0,\,p})_{x_0}(Z,Z')\cong
\big((\mQ_{1,\,x_0}(\mT)-Q_{1,\,x_0}(g_0))\cP_{x_0}\big)
(\sqrt{p}Z,\sqrt{p}Z')\,p^{-1/2}+\mO(p^{-1})\,.
\end{equation}
Thus it remains to show
\begin{equation}\label{6.33e}
F_{1,\,x}:=\mQ_{1,\,x}(\mT)-Q_{1,\,x}(g_0)\equiv0,
\end{equation}
which is proved in \cite[Lemma 4.18]{MM08b}. This establishes \eqref{toe3.60},
hence \eqref{toe3.6}, i.e.\ \eqref{toe3.4} for $m=0$.

\smallskip
\noindent\textbf{Step 2: induction.}
Assume \eqref{toe3.4} holds up to some $m\ge 0$. Set $\mT^{(0)}:=\mT$ and define
\begin{equation}\label{toe3.22}
\mT^{(1)}=\{T^{(1)}_{p}=p(T_p-T_{g_0,\,p}):p\geq1\}.
\end{equation}
Using \eqref{toe3.6} we can write
\(
T^{(1)}_{p}=P_p\big(pS_p-F_\varepsilon(D_p)\,g_0\,F_\varepsilon(D_p)\big)P_p,
\)
so (i)–(ii) are immediate. Condition (iii) follows by subtracting the asymptotic
expansions of $T_{p,x_0}$ (from \eqref{toe3.2}) and $T_{g_0,p,x_0}$ (from
\eqref{toe2.13}), and using Proposition~\ref{toet3.2} and \eqref{6.33e} to see
that the coefficients of $p^0$ and $p^{-1/2}$ vanish. Hence $\mT^{(1)}$ again
satisfies the hypotheses of Theorem~\ref{toet3.1}, and admits an expansion
\begin{equation}\label{toe3.2a}
p^{-n}T^{(1)}_{p,\,x_0}(Z,Z')\cong
\sum_{r=0}^k(\mQ_{r,\,x_0}(\mT^{(1)})\cP_{x_0})(\sqrt{p}Z,\sqrt{p}Z')p^{-r/2}
+\mO(p^{-(k+1)/2}).
\end{equation}
Define
\begin{equation}\label{toe3.5a}
g_1(x_0)=\mQ_{0,\,x_0}(\mT^{(1)})(0,0)\in\End(E_{x_0}).
\end{equation}
Then \eqref{toe3.4} holds for $m=1$.
In general, for $l\in\N$ define
\begin{equation} \label{toe3.22a}
\mT^{(l+1)}=\{T^{(l+1)}_{p}=p(T^{(l)}_p-T_{g_l,\,p}):p\geq1\},\quad
T^{(l+1)}_{p}=p^{l+1}T_p-\sum_{j=0}^l p^{l+1-j}T_{g_j,p}\,,
\end{equation}
and repeat the above argument to verify that $\mT^{(l+1)}$ satisfies (i)–(iii),
hence admits an expansion
\begin{equation}\label{toe3.2b}
p^{-n}T^{(l+1)}_{p,\,x_0}(Z,Z')\cong
\sum_{r=0}^k(\mQ_{r,\,x_0}(\mT^{(l+1)})\cP_{x_0})(\sqrt{p}Z,\sqrt{p}Z')p^{-r/2}
+\mO(p^{-(k+1)/2}),
\end{equation}
and define
\begin{equation}\label{toe3.5b}
g_{l+1}(x_0)=\mQ_{0,\,x_0}(\mT^{(l+1)})(0,0)\in\End(E_{x_0}).
\end{equation}
This inductive construction yields \eqref{toe3.4} for all $m$, hence $\mT$ is a
Toeplitz family in the sense of Definition~\ref{toe-def}.
\end{proof}

\subsection{Proof of Theorem \ref{t2.1}}
\begin{proof}
\noindent
(1) We follow \cite[Theorems 4.1.1 and 6.1.1]{MM07}. We first deduce a spectral
gap for the Kodaira Laplacian $\square_p=\frac12 D_p^2$ acting on sections of
$L^p\otimes E$. Let $f\in \Dom(\square_p)\cap L^2(X,L^p\otimes E)$ and set
$s=\ov\partial^E_p f$. Then \eqref{bk1.4} gives
\begin{equation}\label{2.22}
\|2\square_{p}f\|^{2}_{L^{2}}
=2\big\langle\ov\partial_{p}^{E,*}s,\ov\partial_{p}^{E,*}s\big\rangle
=\|D_ps\|_{L^2}^2
\ge (2C_0p-C_L)\|s\|_{L^2}^2
=(2C_0p-C_L)\langle\square_p f,f\rangle .
\end{equation}
Hence
\begin{equation}\label{2.23}
{\rm Spec}(2\square_p)\subset\{0\}\cup[2C_0p-C_L,\infty[\,.
\end{equation}
By \eqref{2.23} we may localize as in Theorem~\ref{tue16}
(cf.\ \cite[\S 4.1.2]{MM07}) and conclude from
\cite[Theorem 4.1.24]{MM07} exactly as in the compact case
\cite[Theorem 4.1.1]{MM07}.

\smallskip
\noindent
(2) Let $g\in\cC^\infty_0(X,\End(E))$. We denote by
$(F_\varepsilon(D_p)\,g\,F_\varepsilon(D_p))(x,x')$ the smooth kernel of
$F_\varepsilon(D_p)\,g\,F_\varepsilon(D_p)$ with respect to $dv_X(x')$.
For any relatively compact open $U\Subset X$ with $\supp(g)\subset U$ we have,
using \eqref{toe5.3a}, \eqref{toe5.3b}, \eqref{toe5.4},
\begin{align}\label{toe5.5}
\begin{split}
&T_{g,p}-F_\varepsilon(D_p)\,g\,F_\varepsilon(D_p)=\mO(p^{-\infty})
\quad\text{in operator norm},\\
&T_{g,p}(x,x')-\big(F_\varepsilon(D_p)\,g\,F_\varepsilon(D_p)\big)(x,x')
=\cO(p^{-\infty}) \quad \text{on } U\times U.
\end{split}
\end{align}

Fix $f,g\in\cC^\infty_0(X,\End(E))$ and choose $U\Subset X$ such that
$\supp(f)\cup\supp(g)\subset U$ and $d(x,y)>2\varepsilon$ for all
$x\in\supp(f)\cup\supp(g)$, $y\in X\setminus U$. Then \eqref{toe5.3a} implies
\begin{equation}\label{toe5.6}
T_{f,p}T_{g,p}=P_p\,F_\varepsilon(D_p)\,f\,P_p\,g\,F_\varepsilon(D_p)\,P_p .
\end{equation}
The kernel of $F_\varepsilon(D_p)fP_p\,gF_\varepsilon(D_p)$ is supported in
$U\times U$, and Lemmas~\ref{toet2.1}, \ref{toet2.3} and \eqref{toe5.4} show that it satisfies
\eqref{toe3.1}. Moreover, as in \eqref{toe2.17}, for $x_0\in U$ and
$|Z|,|Z'|<\var/4$,

\begin{equation}\label{toe4.5}
\begin{split}
(&F_\varepsilon(D_p)f P_p \,g\,F_\varepsilon(D_p))_{x_0}(Z,Z^\prime)=
(T_{f,\,p}\,T_{g,\,p})_{x_0}(Z,Z^\prime)+\cO(p^{-\infty})\\
&=\int_{{ T_{x_0}X}}
T_{f,\,p,\,x_0}(Z,Z^{\prime\prime})
\rho\Big(\frac{4|Z^{\prime\prime}|}{\var}\Big)
T_{g,\,p,\,x_0}(Z^{\prime\prime},Z^{\prime})
\kappa_{x_0}(Z^{\prime\prime})
\,dv_{TX}(Z^{\prime\prime})+\cO(p^{-\infty}).
\end{split}
\end{equation}  
Combining \eqref{toe4.5} with Theorem~\ref{toet2.3} gives, as in the proof of
Theorem~\ref{toet2.3},
\begin{multline}\label{toe5.7}
p^{-n}\big(F_\varepsilon(D_p)fP_p\,gF_\varepsilon(D_p)\big)_{x_0}(Z,Z')
\cong \sum_{r=0}^k (Q_{r,x_0}(f,g)\cP_{x_0})(\sqrt p Z,\sqrt p Z')\,p^{-r/2}
+\mO(p^{-(k+1)/2}),
\end{multline}
with
\begin{equation}\label{toe4.7}
Q_{r,x_0}(f,g)=\sum_{r_1+r_2=r}\cK\big[Q_{r_1,x_0}(f),\,Q_{r_2,x_0}(g)\big],
\end{equation}
where $Q_{r,x_0}(f)$ are given by \eqref{toe2.14}. 
Hence, by
Theorem~\ref{toet3.1} and \eqref{toe5.5}, there exist $C_l(f,g)\in\cC^\infty_0(X,\End(E))$
with $\supp(C_l(f,g))\subset\supp(f)\cap\supp(g)$ such that for any $k\ge1$,
\begin{equation} \label{toe5.8}
\Big\|F_\varepsilon(D_p)f P_p\,g\,F_\varepsilon(D_p) s
- \sum_{l=0}^k F_\varepsilon(D_p) P_{p}\,C_l(f,g) \, p^{-l}\, P_{p} F_\varepsilon(D_p)s\Big\|
\leqslant C_kp^{-k-1}\|s\|.
\end{equation}
The estimates \eqref{toe5.6} and \eqref{toe5.8} imply that 
\begin{equation}\label{toe5.9}
\Big\|T_{f,p}T_{g,p}-\sum_{l=0}^k P_p\,C_l(f,g)\,p^{-l}P_p\Big\|
\le C_k\,p^{-k-1}.
\end{equation}
This proves (i) in Definition~\ref{D:btp} for sections
$f,g$ %\in\cC^\infty_0(X,\End(E))$ 
with compact support. 
In general, for $f,g\in\cC^\infty_{\rm const}(X,\End(E))$
we write $f=f_0+c_f$, $g=g_0+c_g$, with $f_0,g_0$ 
with compact support and $c_f,c_g\in\C$. 
Then $T_{f,p}=T_{f_0,p}+c_fP_p$, $T_{g,p}=T_{g_0,p}+c_gP_p$, hence
\[T_{f,p}T_{g,p}=T_{f_0,p}T_{g_0,p}+c_gT_{f,p}+c_fT_{g,p}+c_fc_gP_p\,.\]
Using the expansion \eqref{toe5.9} for $T_{f_0,p}T_{g_0,p}$
we obtain the expansion for $T_{f,p}T_{g,p}$ taking into account
that $C_0(f_0,g_0)+c_g f+c_f g+c_f c_g=C_0(f,g)$
and $C_r(f_0,g_0)=C_r(f,g)$ f\"ur $r\geq1$.
This completes the proof of preperty (i) from Definition~\ref{D:btp}. 

We note that the coefficients $C_r(f,g)\in\cC^\infty(X,\End(E))$ 
in \eqref{toe4.2}, \eqref{toe5.9},
are constructed inductively in Theorem \ref{toet3.1}.
By setting $\mT_{f,g}=\{T_{f,p}T_{g,p}\}$ we have 
$C_0(f,g)(x)=\mQ_{0,\,x}(\mT_{f,g})(0,0)=Q_{0,x}(f,g)$ by \eqref{toe3.5};
for $l\geq0$, we define inductively 
\begin{equation}\label{toe5.91}
\begin{split}
&\mT_{f,g}^{(l+1)}=\Big\{
p^{l+1}T_{f,p}T_{g,p}-\sum_{j=0}^l p^{l+1-j}T_{C_j(f,g),p}\Big\},\\
&C_{l+1}(f,g)(x)=\mQ_{0,\,x}(\mT_{f,g}^{(l+1)})(0,0)
\end{split}
\end{equation}
as in \eqref{toe3.22a} and \eqref{toe3.5b}.
%
%$\mT_{f,g}^{(l+1)}=\{T^{(l+1)}_p\}$ with
%$T^{(l+1)}_{p}=p^{l+1}T_{f,p}T_{g,p}-\sum_{j=0}^l p^{l+1-j}T_{C_j(f,g),p}$
%as in \eqref{toe3.22a} and set $C_{l+1}(f,g)(x)=\mQ_{0,\,x}(\mT_{f,g}^{(l+1)})(0,0)$,
%see \eqref{toe3.5b}.
%

Property (ii) from Definition~\ref{D:btp} follows as in the compact
case \cite[Theorem 7.4.1]{MM07}, \cite[Theorem 1.1]{MM08b}; in particular
\begin{equation}\label{toe4.8}
C_0(f,g)(x)=Q_{0,x}(f,g)=\cK[Q_{0,x}(f),Q_{0,x}(g)]=f(x)g(x),
\end{equation}
and the commutator relation \eqref{toe4.4} follows from
\begin{equation}\label{toe4.16}
C_1(f,g)(x)-C_1(g,f)(x)=\imat\{f,g\}\,\Id_E.
\end{equation}

Finally, for (iii), we adapt \cite[Theorem 7.4.2]{MM07} and
\cite[Theorem 4.19]{MM08b}. Choose $x_0\in X$ and $u_0\in E_{x_0}$, $|u_0|=1$,
such that $|f(x_0)(u_0)|=\|f\|_\infty$. In normal coordinates at $x_0$, trivialize
$L$ and $E$, and let $e_L$ be the unit frame of $L$. Consider the peak sections
\begin{equation}\label{toe4.18a}
S^p_{x_0}(x)=p^{-n/2}P_p(x,x_0).(e_{L,x_0}^{\otimes p}\otimes u_0)\,.
\end{equation}
Using \eqref{0c7} and \eqref{toe5.5}, $F_\varepsilon(D_p)$ has the same local
expansion as $P_p$, and thus by \eqref{toe2.9},
\begin{equation}\label{toe4.18}
\big\|T_{f,p}S^p_{x_0}-f(x_0)S^p_{x_0}\big\|_{L^2}
\le \frac{C}{\sqrt p}\,\|S^p_{x_0}\|_{L^2}.
\end{equation}
If $f$ is real-valued, then $df(x_0)=0$ and the factor $C/\sqrt p$ improves to
$C/p$. This is precisely \eqref{toe4.17a}, completing the proof.
\end{proof}
%}
%end blue
%---------------------------------------------------------------------------

%-----------------------------------------------------------------------------
%\begin{theorem}[{\cite[Theorem 0.3]{MM12}}]\label{toec1}
Let $f,g\in\cC^\infty_{\rm{const}}(X,\End(E))$
and assume that $\omega=\Theta$.
By \cite[Theorem 0.3]{MM12} we have:
 \begin{equation}\label{toe4.3} \begin{split}
% C_0(f,g)=&fg, \quad
C_1(f,g)&=-  \frac{1}{2\pi}
\big\langle \nabla^{1,0} f, \ov{\partial}^{E} g\big\rangle_{\om}\in 
\cC^{\infty}_0(X,\End(E)),\\
C_2(f,g)&= \, \bb_{2,\,f,\,g} - \bb_{2,\,fg}- \bb_{1,\,C_1(f,\,g)}.
\end{split}\end{equation}
%Here $\langle\cdot\,,\cdot \rangle$ acts $\C$-bilinearly
%(and pointwisely) on $TX$ and not on $E$.
If $f,g\in\cC^\infty_{\rm{const}}(X)$, then
 \begin{equation}\label{toe4.3a}
C_2(f,g)= \,
\frac{1}{8\pi^2 } \left\langle  D^{1,0}\partial f, D^{0,1}\ov{\partial} g\right\rangle
+ \frac{\sqrt{-1}}{4\pi^2 } \left\langle  \Ric_\om, 
\partial f \wedge\ov{\partial} g\right\rangle
 -\frac{1}{4\pi^2} \left\langle\partial f\wedge \ov{\partial} g, R^E\right\rangle_{\om}\,.
\end{equation}
%\end{theorem}
%---
%---
\begin{remark}\label{toer2}
%\textbf{(i)} Relations \eqref{toe4.4} and \eqref{toe4.17a} were first proved in some special cases:
%in \cite{KliLe:92} for Riemann surfaces, in \cite{Cob:92} for $\C^n$, 
%and in \cite{BLU:93} for bounded symmetric domains in $\C^n$,
% using explicit calculations.
%Then Bordemann, Meinrenken, and Schlichenmaier \cite{BMS94}
%%(cf. also Guillemin \cite{Guill:95})
%treated the case of a compact K{\"a}hler manifold (with $E=\C$)
%using the theory of Toeplitz structures (generalized Szeg{\"o} operators) by
%Boutet de Monvel and Guillemin \cite{BdMG81}.
%%Guillemin \cite{Guill:95} notices that the method is implicit in \cite{BoG}.
%Moreover, Schlichenmaier \cite{Schlich:00}
%(cf. also  \cite{KS01}, \cite{Cha03})
%continued this train of thought and showed that for any $f,g\in \cC^\infty(X)$,
% the product $T_{f,\,p}\,T_{g,\,p}$ has an asymptotic expansion
% \eqref{toe4.2} and constructed geometrically an associative star product.
% \\[2pt]
%\textbf{(ii)}
There are two ways to prove \eqref{toe4.16}, which also holds
in the more general symplectic case. 
One is to compute directly the difference, % and to use some of the identities \eqref{toe1.12}. 
as is done \cite[p.\,593-594]{MM08b}, \cite[p.\,311]{MM07},
or one can explicitly compute 
each coefficient
$C_1(f,g)$, which is more involved and was done in
\cite[Theorem 1.1]{Ioo18} in the symplectic case,
and then take the difference. 
%The formula for $C_1(f,g)$ is given in Theorem \ref{toec1}.
\end{remark}
%===

%-------------------------
%%%
\subsection{The Berezin-Toeplitz star product}

Let $(X,\om)$ be a K\"ahler manifold. The following notion was introduced by
Bayen, Flato, Fronsdal, Lichnerowicz, and Sternheimer in \cite{BFFLS78a}
as a perspective on quantum mechanics that abstracts
from Hilbert spaces to focus on observables,
where quantum observables are represented by the space
$\cC^{\infty}(X,\C)[[\hbar]]$ of formal series with coefficients in $\cC^{\infty}(X,\C)$,
for which the quantum of action $\hbar$ plays the role
of a formal variable.
%===
\begin{definition}\label{DQdef}
Let $(X,\omega)$ be a K\"ahler manifold, and let $\{\cdot,\cdot\}$
be the Poisson bracket \eqref{2.17} on $(X,\omega)$.
A \emph{formal deformation} of a Poisson
subalgebra $\cA\subset(\cC^\infty(X,\C),\{\cdot,\cdot\})$
is a linear associative product $\star$ on
$\cA[[\hbar]]$, called star product,
admitting $1\in\cC^{\infty}(X,\C)$ as unity and
whose product $\star$ is given for all
$f,\,g \in\cC^{\infty}(X,\C)$ by
%\begin{equation}\label{DQintro}
\(f\star g= \sum_{r=0}^{\infty}\hbar^r\,C_r(f,g)\),
%\end{equation}
where $\{C_r\}_{r\in\N}$ is a sequence of bidifferential operators acting 
on $\cC^{\infty}(X,\C)$,
and satisfying $C_0(f,g)=fg$ and %\eqref{toe4.16},
%\begin{equation}\label{DQfla}
$C_1(f,g) - C_1(g,f) =\sqrt{-1}\{f,g\}$, where $\{f,g\}$
is defined in \eqref{2.17}.
%
%A \emph{deformation quantization} of a symplectic manifold 
%$(X,\omega)$
%%of the Poisson bracket \eqref{2.17} 
%is a linear associative $\hbar$-algebra
%on the space
%$\cC^{\infty}(X,\C)[[\hbar]]$, admitting $1\in\cC^{\infty}(X,\C)$ as unity and
%whose product $\star$ is given for all
%$f,\,g \in\cC^{\infty}(X,\C)$ by
%\begin{equation}\label{DQintro}
%f\star g= \sum_{r=0}^{\infty}\hbar^r\,C_r(f,g)\;,
%\end{equation}
%where $\{C_r\}_{r\in\N}$ is a sequence of bidifferential operators acting 
%on $\cC^{\infty}(X,\C)$,
%and satisfying $C_0(f,g)=fg$ and \eqref{toe4.16},
%%\begin{equation}\label{DQfla}
%$C_1(f,g) - C_1(g,f) =\{f,g\}$, where $\{f,g\}$ is the Poisson bracket 
%\eqref{2.17} on $(X,\omega)$\,.
%%\end{equation}
\end{definition}
%====
A formal deformation of the Poisson algebra $(\cC^\infty(X,\C),\{\cdot,\cdot\})$
is also called a deformation quantization of the K\"ahler manifold $(X,\omega)$.
Deformation quantization on a compact K\"ahler, or more generally on a 
compact symplectic manifold $(X,\omega)$, is subtle 
since associativity imposes infinitely many constraints on the bidifferential operators 
$\{C_r\}_{r\in\N}$. Existence was proved by De Wilde–Lecomte \cite{DL83}, and
a geometric construction was given by Fedosov \cite{Fed96}. 
Kontsevich extended existence to general Poisson manifolds \cite{Kon03}, 
though explicit computation of the operators $C_r$ remains difficult.
%
%Deformation quantization over a general compact symplectic manifold 
%$(X,\om)$ is challenging due to the infinite constraints of associativity on the 
%sequence of bidifferential operators $\{C_r\}_{r\in\N}$. 
%De Wilde and Lecomte solved this problem in \cite{DL83}, and 
%Fedosov provided a systematic geometric construction in \cite{Fed96} 
%to produce and study deformation quantizations of symplectic manifolds. 
%Kontsevich established the existence of deformation quantization 
%in general Poisson manifolds in \cite{Kon03}, but computing the 
%bidifferential operators remains challenging. 
In this context, we obtain a corollary of Theorem \ref{t2.1} 
regarding the existence and computability of the so-called 
\emph{Berezin-Toeplitz star product} 
over non-compact manifolds (cf.\ \cite{KS01,Schlich:00} for the compact 
K\"ahler case).
%===
\begin{theorem}
Let $(X,\omega)$ be a K\"ahler manifold,		
let $(L,h^L)$ be a holomorphic Hermitian 
line bundle on $X$ such that $\omega=\frac{\sqrt{-1}}{2\pi}R^L$.	
Assume that the Kodaira-Laplacian possesses a spectral gap, 
as stated in \eqref{bk1.4}.
Then the sequence of bidifferential operators $\{C_r\}_{r\in\N}$
given by the asymptotic expansion \eqref{toe4.2} defines a %deformation quantization
$\star$-product on $\cC^\infty_{\rm const}(X,\C)$
on the K\"ahler manifold $(X,\om)$ 
in the sense of Definition \ref{DQdef}, and the calculus of Toeplitz kernels from 
Section \ref{s3.3} provides an algorithm
to compute the sequence $\{C_r\}_{r\in\N}$ recursively in terms 
of local geometric data over $X$.
\end{theorem}
%===
\begin{proof}
Let $f,g\in\cC^\infty_{\rm const}(X,\C)$. By Theorem~\ref{t2.1} there
exists a unique sequence of bidifferential operators
\(
C_r:\cC^\infty_{\rm const}(X,\C)\times
\cC^\infty_{\rm const}(X,\C)\longrightarrow
\cC^\infty_{\rm const}(X,\C)\), \(r\in\N\),
such that, as $p\to+\infty$, \eqref{toe4.2} holds,
$T_{f,\,p}T_{g,\,p}=\sum^\infty_{r=0}p^{-r}T_{C_r(f,\,g),\,p}
+\mO(p^{-\infty})
$\,.
%\begin{equation}\label{toe-star1}
%T_{f,p}\,T_{g,p}\sim \sum_{r=0}^{\infty} p^{-r}\,T_{C_r(f,g),p}.
%\end{equation}
Moreover, $C_r$ depends only on a finite
jet of $f,g$ and the geometric data at the point (locality).
Set $\hbar=1/p$ and define, for $f,g\in \cC^\infty_{\rm const}(X,\C)$,
\begin{equation}\label{toe-star2}
f\star g:=\sum_{r=0}^{\infty}\hbar^{r}(-\imat)^r\,C_r(f,g)\in
\cC^\infty_{\rm const}(X,\C)[[\hbar]].
\end{equation}
By \eqref{toe4.8}, we have $C_0(f,g)=fg$, thus
$\star$ deforms the pointwise product.
Associativity of operator composition implies the associativity of $\star$.
Indeed, for $f,g,h\in\cC^\infty_{\rm const}(X,\End(E))$, we have
\(
T_{f,p}\,(T_{g,p}T_{h,p})=(T_{f,p}T_{g,p})\,T_{h,p},
\)
and expanding both sides using \eqref{toe4.2} %{toe-star1} 
yields, for each $k\ge0$,
\begin{equation}\label{toe-star3}
\sum_{r+s=k} C_r(f,C_s(g,h))=\sum_{r+s=k} C_r(C_s(f,g),h).
\end{equation}
Thus $\star$ defines an associative product on
$\cC^\infty_{\rm const}(X,\C)[[\hbar]]$.

By \eqref{toe4.8}, we have $C_0(f,1)=f$.
Since $T_{1,p}=P_p$, we have $T_{f,p}T_{1,p}=T_{f,p}$; thus,
 $T_{f,p}T_{1,p}-T_{C_0(f,1),p}=T_{f,p}-T_{f,p}=0$; so
$\mT_{f,1}^{(1)}=\{0\}$. Hence, $C_1(f,1)=0$ by \eqref{toe5.91}.
By induction, it follows using \eqref{toe5.91} that $\mT_{f,1}^{(r)}=\{0\}$; thus,
 $C_r(f,1)=0$ for $r\ge1$. In the same way, $C_0(1,f)=f$
and $C_r(1,f)=0$ for $r\ge1$. Therefore, $1$ is the unit for
$\star$. 

Finally, because the $C_r$ are local bidifferential operators and our symbols
are constant outside a compact set, each $C_r(f,g)$ again lies in
$\cC^\infty_{\rm const}(X,\C)$.
\end{proof}
%===
\begin{theorem}\label{T:BTstar}
Let $(X,\omega)$ be a compact K\"ahler manifold; let $(L,h^L)$ be a holomorphic Hermitian		
line bundle on $X$ such that	
$\omega=\frac{\sqrt{-1}}{2\pi}R^L$.
Then the sequence of bidifferential operators 
$\{C_r\}_{r\in\N}$
given by the asymptotic expansion \eqref{toe4.2} defines a deformation quantization 
of the K\"ahler
manifold $(X,\om)$ in the sense of
Definition \ref{DQdef}, and the calculus of Toeplitz kernels developed 
in Section \ref{s3.3} provides an algorithm
to compute the sequence $\{C_r\}_{r\in\N}$ recursively in terms 
of local geometric data over $X$.
\end{theorem}
%===
In the same way can define a formal deformation of the algebra $\cC^\infty(X,\End(E))$,
by setting for $f,g\in\cC^\infty(X,\End(E))$, 
%\begin{equation}\label{toe4.4c-b}
$f\star g\coloneqq\sum_{k=0}^\infty 
(-\imat)^kC_k(f,g) \hbar^{k}\in\cC^\infty_{\rm{const}}(X,\End(E))[[\hbar]]$\,,
%\end{equation}
where $C_{k}(f,g)$ is determined by \eqref{toe4.2}.
This is the \emph{Berezin-Toeplitz star product} (cf.\ \cite{MM07,MM08b} 
for the symplectic case and arbitrary twisting bundle $E$).

%The construction of the star-product can be carried out even in the presence of a twisting vector bundle $E$.
%Let $f,g\in\cC^\infty(X,\End(E))$. Set
%\begin{equation}\label{toe4.4c-b}
%f*_{\hbar}g:=\sum_{k=0}^\infty C_k(f,g) \hbar^{k}\in\cC^\infty_{\rm{const}}(X,\End(E))[[\hbar]]\,,
%\end{equation}
%where $C_{r}(f,g)$ is determined by \eqref{toe4.2}. 
%Then \eqref{toe4.4c} defines an associative star product on $\cC^\infty_{\rm{const}}(X,\End(E))$ called \emph{the Berezin-Toeplitz star product} (cf.\ \cite{KS01,Schlich:00} for the K\"ahler case with $E=\C$ 
%and \cite{MM07,MM08b} for the symplectic case and arbitrary twisting bundle $E$).
%The associativity of the star product \eqref{toe4.4c} follows immediately 
%from the associativity rule for the composition of Toeplitz operators, 
%$(T_{f,\,p}\circ T_{g,\,p})\circ T_{h,\,p}=
%T_{f,\,p}\circ (T_{g,\,p}\circ T_{h,\,p})$ for any  
%$f,g,h\in\cC^\infty_{\rm{const}}(X,\End(E))$, and from the asymptotic expansion 
%\eqref{toe4.2} applied to both sides of the latter equality.

The coefficients $C_r(f,g)$, $r=0,1,2$, are given by \eqref{toe4.3}. Set 
\begin{equation}\label{toe4.5b}
\{\!\{f,g\}\!\}:=%\frac{1}{\imat}(C_1(f,g)-C_1(g,f))
\frac{1}{2\pi\imat}
\big(\langle \nabla^{1,0} g, \ov{\partial}^{E} f\rangle_{\om}-\langle \nabla^{1,0} f, \ov{\partial}^{E} g\rangle_{\om}\big)\,.
\end{equation}
If $fg=gf$ on $X$ we have 
\begin{equation}\label{toe4.4b-2}
\big[T_{f,\,p}\,,T_{g,\,p}\big]=\tfrac{\sqrt{-1}}{\, p}\,T_{\{\!\{f,g\}\!\},\,p}
+\mO\big(p^{-2}\big)\,,\quad p\to\infty.
\end{equation}
Due to the fact that $\{\!\{f,g\}\!\}=\{f,g\}$ if $E$ is trivial and 
comparing \eqref{toe4.4} to \eqref{toe4.4b-2},  one can regard 
$\{\!\{f,g\}\!\}$ defined in \eqref{toe4.5b} 
as a non-commutative Poisson bracket.% by  \eqref{toe4.3}.

%%%
\subsection{Coherent states}
Let $(X,\om)$ be a Hermitian manifold, and let $(L,h^L)$ be a positive
holomorphic Hermitian line bundle on $X$ satisfying
$\omega=\frac{\sqrt{-1}}{2\pi}R^L$. In this section,
we assume that $E=\C$ is the trivial holomorphic line bundle,
equipped with a possibly non-trivial Hermitian metric.

In quantum measurement theory, physical states are represented by positive rank-$1$ operators acting on the Hilbert spaces that correspond to the quantum physical system. In semiclassical analysis, one 
considers quantum states that best approximate the classical states 
associated with points on the symplectic manifold, the associated classical phase space. 
 In Berezin-Toeplitz quantization, \emph{coherent state} represent this notion.

%
%In the theory of quantum measurements, physical states are represented by positive rank-$1$ operators acting
%on the Hilbert spaces that correspond to the quantum physical system. In the context of semiclassical
%analysis, one must then consider quantum states that best approximate the classical states represented
%by points on the symplectic manifold, which correspond to the associated classical phase space.
%In the context of Berezin-Toeplitz quantization, this is given by the following notion of
%a \emph{coherent state}.
%
\begin{proposition}\label{cohstateprojprop}
For $p\in\N$ and $x\in X$, the coherent state projector $\Pi_p(x)$
is the orthogonal projector acting on $H^0_{(2)}(X,L^p)$ satisfying
%\begin{equation}\label{cohstateprojfla}
$\Ker\Pi_p(x)=\{\,s\in H^0_{(2)}(X,L^p)~|~s(x)=0\,\}$\,.
%\end{equation}
\end{proposition}

The link to Berezin-Toeplitz quantization is provided by the following result.
%The basic link between these coherent states
%and the Berezin-Toeplitz quantization presented in Section \ref{BTsec},
%which justifies their expected semiclassical properties, is given by the following result.

\begin{proposition}\label{sxprop}
For any $x\in X$, any $p\in\N$ large enough, 
and any unit vector $e_x\in L^p_x$, the coherent state projector $\Pi_p(x)$
is the rank-$1$ orthogonal projector on the line spanned by the section $s_x\in H^0_{(2)}(X,L^p)$
defined for all $y\in X$ by
\begin{eqnarray}\label{sx}
s_x(y)=P_p(y,x).e_x\,,
\end{eqnarray}
where $e_x\in L^p_x$ is any given unit vector and satisfies
\begin{eqnarray}\label{sxnorm}
\|s_x\|^2_{L^2}=P_p(x,x)\,.    
\end{eqnarray}
\end{proposition}
\begin{proof}
Equation \eqref{sxnorm} is a straightforward consequence of the definition \eqref{sx}
of the section $s_x\in H^0_{(2)}(X,L^p)$, together with the standard properties of
Schwartz kernels under composition and the fact that $P_pP_p=P_p$.
Now, for any $x\in X$,
Theorem \ref{t2.1} shows that there exists $p_0\in\N$
such that, for any $p\geq p_0$, we have $P_p(x,x)\neq 0$, so that
$s_x\in H^0_{(2)}(X,L^p)$ does not vanish identically.
This implies that $s_x(x)\neq 0$, so that
in particular $\Pi_p(x)$ is, in fact, a rank one orthogonal projector.
We are thus reduced to showing that for any $s\in H^0_{(2)}(X,L^p)$ satisfying $s(x)=0$,
we have $\langle s,s_x\rangle=0$. But, using basic properties of the Bergman kernel,
for any $s\in H^0_{(2)}(X,L^p)$ satisfying $s(x)=0$, we have
\begin{equation}
\begin{split}
\langle s,s_x\rangle =\int_X\,\langle s(y),P_p(y,x).e_x\rangle_{L^p}
=\int_X\,\langle P_p(x,y)\,s(y),e_x\rangle_{L^p}=s(x).e_x=0\,.\\
\end{split}
\end{equation}
This establishes the result.
\end{proof}

The following result shows that Berezin-Toeplitz quantization coincides with the
\emph{coherent state quantization} associated with Definition \ref{cohstateprojprop}.

\begin{proposition}\label{BTquantdef}
For any $f\in\cC^{\infty}_0(X,\R)$ and any $p\in\N$ large enough,
the Berezin-Toeplitz quantization of $f$ satisfies
\begin{equation}\label{BTmapfla}
T_{f,p}=\int_X f(x)\,\Pi_p(x)\,P_p(x,x)\,dv_X(x)\,.
\end{equation}
\end{proposition}
\begin{proof}
%First note that for any $x\in X$ and $p\in\N$ large enough,
%the fact that $P_pP_p=P_p$ shows that the section $s_x\in H^0_{(2)}(X,L^p)$
%satisfies $\|s_x\|^2=P_p(x,x)\neq 0$.
By Definition \ref{Toepdef} of the Berezin-Toeplitz quantization
of $f\in\cC^{\infty}_0(X,\R)$, the reproductive property of the Bergman kernel
and using Proposition \ref{sxprop},
for any $s\in H^0_{(2)}(X,L^p)$ and any $y\in X$, we have
\begin{equation}
\begin{split}
T_{f,p}s(y) &=\int_X f(x)\,P_p(y,x).\left(\int_X\,P_p(x,w).s(w)\,dv_X(w)\right)\,dv_X(x)\\
&=\int_X f(x)\,\frac{\langle s,s_x\rangle\,s_x(y)}{\|s_x\|^2_{L^2}}\,P_p(x,x)\,dv_X(x)\\
&=\int_X f(x)\,\left(\Pi_p(x)s\right)(y)\,P_p(x,x)\,dv_X(x)\,.
\end{split}
\end{equation}
This establishes the result.
\end{proof}

Let us introduce the \emph{Berezin symbol}, 
which associates a classical observable with a quantum observable
by its expectation value at coherent states. Semiclassically, 
it represents the classical observable that best approximates the quantum observable.
%
%Let us now introduce the following notion of 
%\emph{the Berezin symbol},
%which associates a classical observable with 
%a quantum observable,
%given by its expectation value at coherent states. 
%This should be interpreted semiclassically
%as the classical observable that best approximates 
%the given quantum observable.

\begin{definition}\label{Bersymb}
%For any $p\in\N$ and 
The \emph{Berezin symbol} of a linear bounded operator $A$ on $H^0_{(2)}(X,L^p)$
is the function $\sigma(A)\in\cC^{\infty}(X,\R)$
defined for any $x\in X$ by
%\begin{equation}
$\sigma(A)(x)=\textup{Tr}[\Pi_p(x)\,A]$\,.
%\end{equation}
\end{definition}
%%%
%%%
\subsection{The Berezin transform}
Berezin introduced his transform in the context of quantization on 
bounded symmetric domains \cite{Berezin75}. 
Already in this setting, the construction 
was tied to the K\"ahler geometry induced by the Bergman metric. 
Beginning with a function $f$ defined on the base manifold, 
assigning to it its Toeplitz operator $T_{f,p}$, and subsequently computing the 
covariant symbol of the Toeplitz operator will result in a function known as the 
Berezin transform $B_pf$ of $f$.
In \cite{KS01}, it is shown that its asymptotic expansion gives
a formal Berezin transform in the sense of Karabegov, associated with a star
product equivalent to the Berezin-Toeplitz star product.

Let $(X,\om)$ be a Hermitian manifold, and let $(L,h^L)$ be a positive
holomorphic Hermitian line bundle on $X$ satisfying
$\omega=\frac{\sqrt{-1}}{2\pi}R^L$. In this section,
we assume that $E=\C$ is the trivial holomorphic line bundle 
equipped with a possibly non-trivial Hermitian metric.
Following \cite{Englis00}, we introduce the following basic tool in
Berezin-Toeplitz quantization.

\begin{definition}\label{Bpdef}
The \emph{Berezin transform} of $f\in\cC^{\infty}_{0}(X,\R)$, is defined by
\begin{equation}\label{Bpfla1}
B_pf=\sigma(T_{f,p}), \quad \text{for $p\in\N$}\;.
\end{equation}
%if $P_p(x,x)\neq 0$, and by $B_pf(x)=f(x)$ otherwise.
\end{definition}
%===
\noindent
We summarize the properties of the 
Berezin transform seen as a \emph{Markov operator} \cite{IKPS20}.
%
%
%The following result summarizes the basic properties of the 
%Berezin transform as a \emph{Markov operator}, as described 
%in \cite{IKPS20}.
%===
\begin{proposition}\label{Bpprop}
For any $x\in X$ and any $p\in\N$
such that $P_p(x,x)\neq 0$, the Berezin transform of
$f\in\cC^{\infty}_{b}(X,\R)$ satisfies
\begin{equation}\label{Bpfla}
B_pf(x)=\int_X\,\frac{|P_p(x,y)|_{p}^2}{P_p(x,x)} \,f(y)\,
dv_X(y)\;.
\end{equation}
where $|\cdot|_{p}$ is the norm induced by $h^L$ on
$L^p_x\otimes(L^p_y)^*$ for all $x,\,y\in X$.

%Furthemore, for any $p\in\N$, the Berezin transform defines a \emph{Markov operator},
%so that it sends measurable bounded positive functions to measurable bounded positive functions and
%satisfies $B_p1=1$.
Furthermore, for any $p\in\N$, the Berezin transform sends
measurable bounded positive functions to measurable bounded positive functions and
extends to a self-adjoint bounded positive operator
acting on $L^2(X,\R)$.
 \end{proposition}
 \begin{proof}
 %Recall that, since the orthogonal projection $P_p$ is a positive self-adjoint operator with smooth kernel,
 %we have $P_p(x,x)\geq 0$, for all $x\in X$, so that the Berezin transforms sends positive functions
 %to positive functions by Definition \ref{Bpdef}.
 By Definition \ref{Toepdef} of the Berezin-Toeplitz quantization of
$f\in\cC^{\infty}_{b}(X,\R)$, Definition \ref{Bersymb} of the Berezin symbol and using Proposition \ref{sxprop},
 for any $x\in X$,  we have
\begin{equation}
\begin{split}
    B_pf(x)&=\frac{\langle T_{f,p}s_x,s_x\rangle}{\|s_x\|^2_{L^2}}=\int_X\,\frac{\langle s_x(y),s_x(y)\rangle_{L^p}}{P_p(x,x)}\,f(y)\,dv_X(y)\\
    &=\int_X\,\frac{|P_p(x,y)|_{p}^2}{P_p(x,x)} \,f(y)\,dv_X(y)\,.
\end{split}
\end{equation}
This shows \eqref{Bpfla}.
Now for any positive $f\in\cC^{\infty}_{b}(X,\R)$ and any $x\in X$ 
such that $P_p(x,x)\neq 0$,
using the elementary fact that $\|T_{f,p}\|_{op}\leq\|f\|_\infty$ 
%following from Definition \ref{Toepdef} of the Berezin-Toeplitz quantization of
for $f\in\cC^{\infty}_{b}(X,\R)$, and since $\Pi_p(x)$ 
is a rank-$1$ projection by Proposition \ref{sxprop},
we get
\begin{equation}\label{Markoveq}
%\begin{split}
 B_pf(x)=\textup{Tr}[\Pi_p(x)T_{f,p}]\leq \textup{Tr}[\Pi_p(x)]\,\|T_{f,p}\|_{op}
    \leq\|f\|_\infty\,. %\sup_{y\in X}f(y)\,.
 %   \end{split}
\end{equation}
Since the Toeplitz operator $T_{f,p}$ associated to a positive symbol $f$ 
is a positive operator, the Berezin transform maps the set of bounded measurable 
positive functions to itself. %bounded measurable positive functions.
%and the same computation shows that $B_p1=1$.

To show that it extends as a bounded self-adjoint positive operator
acting on $L^2(X,\R)$, let us set $K:=\{x\in X~|~P_p(x,x)\neq 0\}\subset X$, and note
from \eqref{bk2.4} that for any $x\in X$ such that $P_p(x,x)=0$, Definition \ref{cohstateprojprop}
shows that $\Pi_p(x)=0$.
%Thus for any $f\in\cC^{\infty}_0(X,\R)$,
 %    by Definition \ref{Bpdef} we have
  %   \begin{equation}\label{1K}
   %      \|B_pf\|_{L^2}^2=\|\mathbf{1}_K B_pf\|_{L^2}^2\,.
    % \end{equation}
Then for all $x\in K$ and $y\in X$, writing $B_p(x,y)\geq 0$ for the Schwartz kernel of the Berezin transform
as given by formula \eqref{Bpfla}, we get via Cauchy-Schwarz inequality for all $f\in\cC^{\infty}_0(X,\R)$,
\begin{equation}\label{Schurtest}
\begin{split}
\|B_p f\|_{L_2}^2&\leq\int_K\left(\int_X B_p(x,y)\,dv_X(y)\right)
\left(\int_X B_p(x,y)\,|f(y)|^2\,dv_X(y)\right)dv_X(x)\\
&\leq\sup_{x\in X}\left(\int_X B_p(x,y)\,dv_X(y)\right)
\sup_{y\in K}\left(\int_X B_p(x,y)\,dv_X(x)\right)\|f\|_{L_2}^2\,.
\end{split}
\end{equation}
Together with \eqref{Markoveq} applied to $f\equiv 1$,
this implies that $\|B_p f\|_{L_2}^2\leq\|f\|_{L_2}^2$,
%and with $\eqref{1K}$ that $\|B_pf\|_{L^2}^2\leq 2\|f\|_{L_2}^2$,
so that $B_p$ defines in fact
a bounded self-adjoint operator on $L^2(X,\R)$ by density. The fact that it is positive and self-adjoint
follows directly from Definition \ref{Bpdef}.
 \end{proof}

The following result gives the asymptotic expansion of the Berezin transform, extending a result
of Englis in \cite{Englis00} in the case of pseudoconvex domains considered in Section \ref{pseudocvx},
and of Karabegov-Schlichenmaier in \cite{KS01} and Ioos-Kaminker-Polterovich-Shmoish in
\cite[Proposition 3.8]{IKPS20}
in the case of compact manifolds considered in Theorem \ref{t2.1comp}. As explained in
\cite[Remark 3.12]{IKPS20}, the weighted case considered in \cite[(1)]{Englis00} corresponds to
the case $E=\C$ equipped with a non-trivial Hermitian metric.

\begin{theorem}\label{KSexpprop}
Let $(X,\om)$ be a Hermitian manifold and let $(L,h^L)$ be a positive holomorphic Hermitian		
line bundle on $X$ satisfying $\omega=\frac{\sqrt{-1}}{2\pi}R^L$.
Assume $E=\C$ is the trivial holomorphic line bundle,
equipped with a possibly non-trivial Hermitian metric, and
that the Kodaira-Laplacian has a spectral gap in the sense of
Definition \ref{specgapdef}.
Then there exists a sequence of differential operators $\{D_j\}_{j\in\N^*}$
acting on $\cC^{\infty}(X,\R)$ such that
for any compact set $K\subset X$ and any $m,\,r\in\N$, there exist $l\in\N$ and a constant
$C_m>0$, uniform in the
$\cC^m$-norm of the derivatives of $h^L$ and $h^E$ up to order $l$,
such that for any
$f\in\cC_{0}^{\infty}(X,\R)$ and all $p\in\N$ big enough, we have
\begin{equation}\label{KSexp}
\left|\,B_{p}f-f-\sum_{j=1}^{r-1} p^{-j}\,D_j\,f\,\right|_{\cC^m(K)}\leq
C_m p^{-r}|f|_{\cC^{m+2r}(K)}\;.
\end{equation}
Furthermore, we have $D_1=-\frac{\Delta}{4\pi}$,
where $\Delta$ is the Laplace-Beltrami operator of $(X,g^{TX})$.
\end{theorem}
\begin{proof}
First, recall that for any compact set $K\subset X$,
Theorem \ref{t2.1} implies that there exists $p_0\in\N$ such that for all $x\in K$, we have
$P_p(x,x)\neq 0$ for all $p\geq p_0$, so that for any $f\in\cC_{0}^{\infty}(X,\R)$, formula
\eqref{Bpfla} holds for its Berezin transform $B_pf$.
Now, using \eqref{toe2.5}, we readily get from formula \eqref{Bpfla} that
\begin{equation}\label{BpflaE}
B_pf(x)=\frac{T_{f,p}(x,x)}{P_p(x,x)}\,.
\end{equation}
Thus, following the analogous computations in the proof of \cite[Proposition 3.4]{IP23},
this is a straightforward consequence of the explicit formulas for the coefficients
of the asymptotic expansion of the Bergman kernel and of the Toeplitz operators
along the diagonal given by \eqref{toe4.131} and 
%Theorems \ref{toet4.5} and %\ref{toet4.6} 
\cite[Theorem 0.1]{MM12} in the case $E=\C$.
% , together
% with the following Weitzenb\"ock formula for the Kodaira Laplacian $\Box^E$ acting
% on $f\in\cC_{0}^{\infty}(X,\End(E))$
% \begin{equation}
% \Box^E\,f=\Delta^E f
% -\sqrt{-1}\left[\langle\om,R^E\rangle_{g^{TX}},f\,\right]\,,
% \end{equation}
% where $\Delta^E$ is the B\"ochner Laplacian acting on $\cC^{\infty}(X,\End(E))$ defined in \eqref{Bochner}.
\end{proof}

\section{Non-compact manifolds}
\label{S:noncomp}

\subsection{General framework}
In this section, we give geometric conditions ensuring the spectral gap \eqref{bk1.4}
for the Kodaira Laplacian on $(0,1)$-forms with values in $L^p\otimes E$.
Combined with Theorem~\ref{t2.1}, these hypotheses yield the
Berezin--Toeplitz package for $\cC^\infty_{\rm const}(X,\End(E))$.
We first collect the analytic input on complete Hermitian manifolds used below; see
\cite[Lemma 3.3.1, Corollaries 3.3.3--3.3.4]{MM07}.
%
%We recall a basic fact due to Andreotti-Vesentini about analysis on complete Hermitian manifolds.
%For the proofs, we refer to \cite[Lemma 3.3.1]{MM07}, 
%\cite[Corollary 3.3.3, Corollary 3.3.4]{MM07}.
%--------
\begin{theorem}\label{esa}
Let $(X,\Theta)$ be a complete Hermitian manifold 
and let $(F,h^F)$ be a holomorphic Hermitian vector bundle. 

\smallskip
\noindent
(i) 
Let $\overline{\partial}^{\smash{\scriptscriptstyle F}}_{max}$ and $\overline{\partial}^{\smash{\scriptscriptstyle{F}},*}_{max}$ be the maximal extensions
of\, $\overline{\partial}^{\smash{\scriptscriptstyle F}}$ and $\overline{\partial}^{\smash{\scriptscriptstyle{F}},*}$, respectively.
Then $\Omega_0^{0,\bullet}(X,F)$ is dense in %$\Dom(\db^E_{max})$, 
%$\Dom(\db^{E,*}_{max})$, 
%$\Dom(\db^E_{max})\cap\Dom(\db^{E,*}_{max})$ 
\[\Dom(\overline{\partial}^{\smash{\scriptscriptstyle F}}_{max}),\:\: 
\Dom(\overline{\partial}^{\smash{\scriptscriptstyle{F}},*}_{max}),\:\:
\Dom(\overline{\partial}^{\smash{\scriptscriptstyle F}}_{max})\cap\Dom(\overline{\partial}^{\smash{\scriptscriptstyle{F}},*}_{max}),\]
in the 
graph norms of $\overline{\partial}^{\smash{\scriptscriptstyle F}}_{max}$, $\overline{\partial}^{\smash{\scriptscriptstyle{F}},*}_{max}$, 
and $\overline{\partial}^{\smash{\scriptscriptstyle F}}_{max}+\overline{\partial}^{\smash{\scriptscriptstyle{F}},*}_{max}$, respectively.

\smallskip
\noindent
(ii) The Hilbert space adjoint of the maximal extension and the maximal extension of the 
formal adjoint of $\overline{\partial}^{\smash{\scriptscriptstyle F}}$ coincide.
%The Hilbert space adjoint of the maximal extension
%$\db^E_{max}$ coincides with the maximal extension
%of $\db^{E,*}_{max}$\,, 
that is, $\overline{\partial}^{\smash{\scriptscriptstyle{F}},*}_{H}=\overline{\partial}^{\smash{\scriptscriptstyle{F}},*}_{max}$\,.

\smallskip
\noindent
(iii) The Kodaira-Laplacian 
$\square^F:\Omega^{0,\bullet}_{0}(X,F)\to\Omega^{0,\bullet}_{(2)}(X,F)$
is essentially self-adjoint. 
In particular, its Gaffney and Friedrichs extensions coincide, and their
associated quadratic form is the form $Q$ given by \eqref{ell2,1}.
\end{theorem}
%--------
We denote by $R^{\det}$ the curvature of the holomorphic Hermitian
connection $\nabla^{\det}$ on $K_X^*=\det (T^{(1,0)}X)$.
We have the following spectral gap result.
%We have the following generalization of Theorems \ref{bkt1.1} and \ref{bkt1.2}.
%-------------------------
%===
\begin{condition}\label{C:specgapgeom1}
Let $(X,J,\Theta)$ be a Hermitian manifold and 
let $(L,h^L)$ and $(E,h^E)$ be 
holomorphic Hermitian vector bundles 
of rank one and $r$, respectively.
We assume that the Riemannian metric $g^{TX}$
induced from $\Theta$ 
is complete and we suppose that there exist 
$C,\varepsilon>0$ such that
\begin{equation}\label{i}
\sqrt{-1}R^L >\varepsilon\Theta\,,
\quad\,\sqrt{-1}(R^{\det}+R^E)> -C\Theta \Id_E\,,\quad\,
|\partial \Theta|_{g^{TX}}< C\,.
\end{equation}
If $L=K_X:=\det(T^{*(1,0)}X)$ is the canonical line bundle on $X$, 
the first two conditions in \eqref{i} are to be replaced by    
\begin{equation*}    
\text{$h^L$ is induced by $\Theta$ and $\sqrt{-1}R^{\det}<-\varepsilon\Theta$,   
$\sqrt{-1}R^E> -C\Theta \Id_E$\,.}   
\end{equation*}
\end{condition}
%===
\begin{theorem}[{\cite[Theorem 6.1.1]{MM07}, \cite[Theorem 3.11]{MM08a}}]
\label{noncompact0}
Assume that Condition \ref{C:specgapgeom1} holds.
%Let $(X,J,\Theta)$ be a Hermitian manifold and 
%let $(L,h^L)$ and $(E,h^E)$ holomorphic Hermitian vector bundles 
%of rank one and $r$, respectively. We assume that Condition 
%\ref{C:specgapgeom1} holds.
%%Assume that $(X,\Theta)$ is a complete Hermitian manifold. 
%%Let $(L,h^L)$ and $(E,h^E)$ holomorphic Hermitian vector bundles 
%%of rank one and $r$, respectively.
%%Suppose that there exist $\varepsilon>0$, $C>0$ such that\,{\rm:}
%%\begin{equation}\label{i}
%%\sqrt{-1}R^L >\varepsilon\Theta\,,
%%\quad\,\sqrt{-1}(R^{\det}+R^E)> -C\Theta \Id_E\,,\quad\,
%%|\partial \Theta|_{g^{TX}}< C,
%%\end{equation}
Then there exist $C_1>0$ and $p_0\in\N$ such that for $p\geqslant p_0$ 
the quadratic form $Q_p$ associated to the Kodaira Laplacian $\square_p$
acting on $\Omega^{0,q}_{(2)}(X,L^p\otimes{E})$ satisfies
\begin{equation}\label{ell4,1}
Q_p(s,s)\geqslant  C_1 p \,\norm{s}^2_{L^2}\,,
\quad \text{for $s\in\Dom(Q_p)\cap\Omega^{0,q}_{(2)}(X,L^p\otimes{E})$, $q>0$}\,.
\end{equation}
Especially, there exists $p_0\in\N$ such that for $p\geqslant p_0$,
\begin{equation}\label{ell4,2}
{H}^{0,\,q}_{(2)}(X,L^p\otimes{E})=0\,,\quad\text{for  $q>0$}
\end{equation}
and the spectrum $\spec(\square_p)$ of $\square_p$
acting on $L^2(X,L^p\otimes E)$
is contained in $\{0\}\cup[\,p\,C_1 ,\infty[$\,.
%\begin{equation}\label{ell4,3}
%\spec(\square_p\vert_{L^2(X,L^p\otimes E)})\subset\{0\}\cup[p\,C_1 ,\infty[\,,
%\end{equation}
\end{theorem}
%--------------------------------
\begin{proof}
The proof is based on the Bochner-Kodaira-Nakano formula 
\cite[Theorem 1.4.12]{MM07}
and its consequence Nakano's inequality \cite[Corollary 1.4.17]{MM07}.
Let $(F,h^F)$ be a holomorphic Hermitian bundle on $X$.
Set $\widetilde{F}=F\otimes K^*_X$ where
$K^*_X=\Lambda^n(T^{(1,0)}X)=\det (T^{(1,0)}X)$.
Since $K_X\otimes K_X^*\cong\C$, there exists a natural isometry
\begin{equation}\label{lm2.70}
\begin{split}
&\Psi=\thicksim\,:\Lambda^{0,q}(T^*X)\otimes F
\longrightarrow\Lambda^{n,q}(T^*X)\otimes \widetilde F,\\
&\Psi \, s=\widetilde s=(w^1\wedge\ldots\wedge w^n\wedge s)\otimes
(w_1\wedge\ldots\wedge w_n),
\end{split}
\end{equation}
where $\{w_j\}^n_{j=1}$ is a local orthonormal frame of $T^{(1,0)}X$
and $\{w^j\}^n_{j=1}$ is its dual frame.
Let us denote by $\mathcal{T}=[i(\Theta),\partial\Theta]$ the 
Hermitian torsion of the metric $\Theta$.
The Bochner-Kodaira-Nakano formula \cite[Corollary 1.4.17]{MM07} 
shows that 
for any $s\in\Omega^{0,q}_0(X,F)$,
%-----------------------------------------------------------------------
%\begin{equation} \label{herm20,11}
%\begin{split}
%\frac{3}{2}\big(\norm{\overline{\partial}^{\smash{\scriptscriptstyle F}}s}^2_{L^2}+\norm{\overline{\partial}^{\smash{\scriptscriptstyle{F}},*}s}^2_{L^2}\big)
%&\geqslant\big\langle R^{F\otimes K^*_X}(w_j,\ov{w}_k)
%\ov{w}^k\wedge i_{\ov{w}_j}s,s\big\rangle\\
%&-\frac{1}{2}\big(\norm{\mathcal{T}^*\wi{s}}^2_{L^2}
%+\norm{\ov{\mathcal{T}}\wi{s}}^2_{L^2}
%+\norm{\ov{\mathcal{T}}^*\wi{s}}^2_{L^2}\big).
%\end{split}
%\end{equation}
%
\begin{equation} \label{herm20,11}
\frac{3}{2}\big(\norm{\overline{\partial}^{\smash{\scriptscriptstyle F}}s}^2+
\norm{\overline{\partial}^{\smash{\scriptscriptstyle{F}},*}s}^2\big)
\geqslant\big\langle R^{F\otimes K^*_X}(w_j,\ov{w}_k)
\ov{w}^k\wedge i_{\ov{w}_j}s,s\big\rangle
-\frac{1}{2}\big(\norm{\mathcal{T}^*\wi{s}}^2
+\norm{\ov{\mathcal{T}}\wi{s}}^2
+\norm{\ov{\mathcal{T}}^*\wi{s}}^2\big).
\end{equation}

By applying \eqref{herm20,11} for $F=L^p\otimes E$ and taking into account
that $R^{L^p}=pR^L$ and \eqref{i} we immediately obtain that for 
$Q_p(s,s)\geqslant  C_1 p \,\norm{s}^2_{L^2}$ for any 
$s\in\Omega^{0,q}_0(X,L^p\otimes E)$ and $q>0$.
By the density result of Theorem \ref{esa} we obtain \eqref{ell4,1}.
Then as in the proof of Theorem \ref{t2.1} we obtain the 
spectral gap of $\square_p$ on $L_{0,0}^2(X,L^p\otimes E)$.
\end{proof}
%--------------------------------

Theorems \ref{t2.1} and \ref{noncompact0} imply the following result.
%-----------------------------------------------------------
\begin{theorem}\label{t2.12}
Let 
$(X,\Theta)$ be a complete Hermitian manifold of dimension $n$ and $L$ and 
$E$ be two holomorphic vector bundles on $X$, where $\rank L=1$, 
such that condition \eqref{i} is fulfilled. Set $\omega=\frac{\sqrt{-1}}{2\pi}R^L$.
Then:

\noindent
(1) The Bergman kernel asymptotics for $H^0_{(2)}(X,L^p\otimes E,\Theta^n/n!)$
holds on compact sets of $X$. 

\noindent
(2) The Berezin-Toeplitz quantization package holds for the K\"ahler manifold
$(X,\omega)$, the algebra
$\cC^\infty_{\rm const}(X,\End(E))$ and quantum spaces
$H^0_{(2)}(X,L^p\otimes E,\Theta^n/n!)$.
\end{theorem}

%--------------------------------
Next, we consider an interesting analogue of Theorem \ref{noncompact0}
for $(n,0)$-forms with values in $L^p\otimes E$, where $n=\dim X$.
Let us first note that for such forms there exists a canonical $L^2$ condition. 
Indeed, let $(F,h^F)$ be a holomorphic Hermitian vector bundle over the
manifold $X$ and let $\Theta$ be any Hermitian metric on $X$.
Let $\Omega^{0,0}_{(2)}(X,F\otimes K_X)$ be the space of $L^2$
sections in $F\otimes K_X$, where the $L^2$ norm is calculated
with respect to the metrics $h^L$, $h^{K_X}$ induced by $\Theta$
and the volume form of $\Theta$. Of course,
$\Omega^{0,0}_{(2)}(X,F\otimes K_X)$ equals the space 
$\Omega^{n,0}_{(2)}(X,F)$ of square integrable $(n,0)$-forms with values in $F$.
For any $(n,0)$-form $s$ with values in $F$, and any metrics 
$g^{TX}$, $g^{TX}_1$ on $X$, 
with Riemannian volume forms $dv_X$, $dv_{X,1}$, respectively, 
we have $\abs{s}_{g^{TX}}^2dv_X=\abs{s}_{g^{TX}_1}^2dv_{X,1}$ pointwise.
%
%
%We have 
%\[
%\Omega^{n,0}_{(2)}(X,F)=%\{\text{measurable sections $s$ of $F\otimes K_X$}:\}
%\left\{\text{measurable $(n,0)$-forms $s$ of with values in $F$}:
%\int_X s\wedge\overline{s}<\infty\right\}
%\]
%and
%\[
%\|s\|^2_{L^2}=\int_X s\wedge\overline{s},\quad
%\text{for $s\in \Omega^{n,0}_{(2)}(X,F)$}.
%\]
This shows that the $L^2$ condition for $(n,0)$-forms is independent of
the choice of Hermitian metric $\Theta$ on $X$.
Secondly, if we work with $(n,0)$-forms, it is not necessary to impose any condition
regarding $R^{\det}$.
%===
\begin{condition}\label{C:specgapgeom2}
Let $(X,J,\Theta)$ be a Hermitian manifold and 
let $(L,h^L)$ and $(E,h^E)$ be holomorphic Hermitian vector bundles 
of rank one and $r$, respectively.
We assume that the Riemannian metric $g^{TX}$ induced from $\Theta$ 
is complete and we suppose that there exist $C,\varepsilon>0$ such that
\begin{equation}\label{ii}
\sqrt{-1}R^L >\varepsilon\Theta\,,
\quad\,\sqrt{-1}R^E> -C\Theta \Id_E\,,\quad\,
|\partial \Theta|_{g^{TX}}< C\,.
\end{equation}
%where $R^{\det}$ be the curvature of the holomorphic connection 
%$\nabla^{\det}$ on $K_X^*=\det(T^{(1,0)}(X))$.
\end{condition}
%===
\begin{theorem} \label{noncompact1}
Assume that Condition \ref{C:specgapgeom2} holds.
%Assume that $(X,\Theta)$ is a complete Hermitian manifold of dimension $n$. 
%Let $(L,h^L)$ and $(E,h^E)$ holomorphic Hermitian vector bundles 
%of rank one and $r$, respectively.
%Suppose that there exist $\varepsilon>0$, $C>0$ such that\,{\rm:}
%\begin{equation}\label{i1}
%\sqrt{-1}R^L >\varepsilon\Theta\,,
%\quad\,\sqrt{-1}R^E> -C\Theta \Id_E\,,\quad\,
%|\partial \Theta|_{g^{TX}}< C,
%\end{equation}
Then there exist $C_1>0$ and  $p_0\in\N$ such that for $p\geqslant p_0$ 
the quadratic form $Q_p$ associated with the Kodaira Laplacian 
$\square_p$ acting on $\Omega^{n,q}_{(2)}(X,F)$ satisfies
\begin{equation}\label{ell4,12}
Q_p(s,s)\geqslant  C_1 p \,\norm{s}^2%_{L^2}
\,,\quad \text{for $s\in\Dom(Q_p)\cap\Omega^{n,q}_{(2)}(X,L^p\otimes{E})$, $q>0$}\,.
\end{equation}
Especially, ${H}^{n,\,q}_{(2)}(X,L^p\otimes{E})=0$
for all $p\geqslant p_0$, $q>0$,
%}
%\end{equation}
and the spectrum $\spec(\square_p)$ of $\square_p$
acting on $L^2(X,L^p\otimes E)$
is contained in $\{0\}\cup[\,p\,C_1 ,\infty[$\,.
\end{theorem}
%--------------------------------
\begin{proof}
In this case, we can use the following form of Nakano's inequality 
(cf.\ \cite[Theorem 1.4.14]{MM07}), we have for any 
$s\in\Omega^{\bullet,\bullet}_0(X,F)$,
%\begin{equation} \label{herm20,1}
%\frac{3}{2}\langle\square^Fs,s\rangle\geqslant\big\langle[\sqrt{-1}R^F,
%\Lambda]s,s\big\rangle
%-\frac{1}{2}(\norm{\mathcal{T}s}^2_{L^2} +\norm{\mathcal{T}^*s}^2_{L^2}
%+\norm{\overline{\mathcal{T}}s}^2_{L^2}
%+\norm{\overline{\mathcal{T}}^*s}^2_{L^2})\,.
%\end{equation}
\begin{equation} \label{herm20,1}
\frac{3}{2}\langle\square^Fs,s\rangle\geqslant\big\langle[\sqrt{-1}R^F,
\Lambda]s,s\big\rangle
-\frac{1}{2}(\norm{\mathcal{T}s}^2 +\norm{\mathcal{T}^*s}^2
+\norm{\overline{\mathcal{T}}s}^2
+\norm{\overline{\mathcal{T}}^*s}^2)\,.
\end{equation}
For an $(n,q)$-form $s\in\Omega^{n,q}_0(X,F)$ we have 
\(\big\langle[\sqrt{-1}R^F,\Lambda]s,s\big\rangle\geqslant
qa_1(x)|s|^2\)
at every point $x\in X$, where $a_1(x)$
is the smallest eigenvalue of $R^F_x$ with respect to $\Theta$.
Setting $F=L^p\otimes E$, \eqref{ell4,12} follows immediately from \eqref{ii}
and Theorem \ref{esa}.
\end{proof}

%-----------------------------------------------------------
By Theorems \ref{t2.1} and \ref{noncompact1} we have:
%-----------------------------------------------------------
\begin{theorem}\label{t2.11}
Let $(X,\omega)$ be a complete K\"ahler manifold and $(L,h^L)$ 
be a prequantum line bundle, that is, $\omega=\frac{\sqrt{-1}}{2\pi}R^L$.
Let $(E,h^E)$ be a holomorphic Hermitian vector bundle.
Assume that there exists $C>0$ such that $\sqrt{-1}R^E>-C\omega\Id_E$.
Then:

\noindent
(1) The Bergman kernel asymptotics for $H^{n,0}_{(2)}(X,L^p\otimes E)$
holds on compact sets of $X$. 

\noindent
(2) The Berezin-Toeplitz quantization package holds 
for the K\"ahler manifold $(X,\omega)$, the algebra
$\cC^\infty_{\rm const}(X,\End(E))$, and quantum spaces
$H^{n,0}_{(2)}(X,L^p\otimes E)$.
\end{theorem}
%-----------------------------------------------------------
The advantage of this result is that we do not need 
any condition on $R^{\det}$, and the Hilbert quantum spaces
do not depend on the chosen Hermitian metric on $X$. 

%-----------------------------------------------------------
%\begin{theorem}\label{t2.11}
%Let 
%$(X,\Theta)$ be a complete Hermitian manifold and $L$ and 
%$E$ be two holomorphic vector bundles on $X$, where $\rank L=1$,
%such that condition \eqref{i1} is fulfilled.
%Then:
%
%(1) The Bergman kernel asymptotics for $H^{n,0}_{(2)}(X,L^p\otimes E)$
%holds on compact sets of $X$. 
%
%(2) The Berezin-Toeplitz quantization package holds for the algebra
%$\cC^\infty_{\rm const}(X,\End(E))$ and quantum spaces
%$H^{n,0}_{(2)}(X,L^p\otimes E)$.
%\end{theorem}
%-----------------------------------------------------------

%---------------------------------------------------------
%-----------------------------------------------------------
\begin{corollary}\label{c2.11}
Let $(X,\omega)$ be a complete K\"ahler-Einstein manifold 
with $\Ric\omega=-\omega$. Then:

\noindent
(1) The Bergman kernel asymptotics for $H^{0}_{(2)}(X,K_X^p)$
is valid on compact subsets of $X$. 

\noindent
(2) The Berezin-Toeplitz quantization package holds 
for the K\"ahler manifold $(X,\omega)$, the algebra
$\cC^\infty_{\rm const}(X)$, and quantum spaces
$H^{0}_{(2)}(X,K_X^p)$.
\end{corollary}
%-----------------------------------------------------------

\begin{example}
Let us recall some classes of complete
K\"ahler-Einstein manifolds.

\noindent
(1) The Bergman metric $\omega_B$ on an 
irreducible bounded symmetric domain in $D\subset\C^n$
is a complete K\"ahler-Einstein metric with negative Ricci curvature; see e.g., 
\cite[Proposition 3, p.\,59]{Mok89b}.
%The invariance of the Bergman metric under the transitive action of 
%the automorphism group implies the invariance of its Ricci tensor. 
%In the case of irreducibility, there is only one invariant $(1,1)$-tensor 
%up to a scale factor, thereby forcing the Ricci tensor to be a multiple 
%of the metric. In the context of bounded domains, this multiple is negative.
The canonical line bundle $K_D$ endowed with the metric
induced by $\omega_D$ is positive.

\noindent
(2) Cheng-Yau \cite{CY80} showed that every bounded strictly pseudoconvex 
domain in $\C^n$ admits a unique complete K\"ahler-Einstein metric 
with negative Ricci curvature, which is a biholomorphic invariant.
Mok-Yau \cite{MY:83} extended this result to bounded 
pseudoconvex domains in $\C^n$.
For general strictly pseudoconvex domains, the Bergman metric is not equal to 
the Cheng-Yau metric.

\noindent
(3)  if $X$ is a compact projective manifold, $\Sigma$ is an effective divisor of $X$,
such that $L=K_X\otimes\cO(\Sigma)$ is ample. 
Then by \cite{KobR84} there exists a unique complete K\"ahler-Einstein
metric $\omega$ with $\Ric\omega=-\omega$ on $X\setminus\Sigma$.

%Examples: bounded symmetric domains in $\C^n$ with 
%the Bergman metric, strictly pseudoconvex domains in $\C^n$ with 
%Cheng-Mok-Yau canonical metric, quotients of such domains
\end{example}
%-------------------------------------------------------\subsection{Stein, \texorpdfstring{$1$}{1}-convex, and weakly 1-complete manifolds}		
In this section, we consider very important classes of complex manifolds 		
for the theory of several complex variables. These classes satisfy certain		
complex convexity conditions.

Stein manifolds are the natural framework for the complex function theory of several variables
\cite{CAR51/52,Stein51}. 
The classical results of complex analysis in one variable, 
such as the Mittag-Leffler and Weierstrass theorems, generalize to 
Stein manifolds.
Stein manifolds are characterized by the existence of enough 
holomorphic functions that provide an embedding in Euclidean space. 
They are the non-compact analogs of projective manifolds. 
On a Stein manifold, any holomorphic line bundle becomes a prequantum 
line bundle since we can endow it with a positively curved Hermitian metric.
%------------		
\begin{definition}
Let $X$ be a complex manifold and let $\cO_X(X)$ be the algebra 
of holomorphic functions on $X$. We say that:

\noindent
(a) $X$ is \emph{holomorphically separable}
if for any $x,y\in X$, $x\neq y$, there
exists $f\in\cO_X(X)$ with $f(x)\neq f(y)$.

\noindent
(b) $X$ is \emph{holomorphically convex}
if for any compact set $K\subset X$, the holomorphic hull $\widehat K
:=\{x\in X : |f(x)|\leqslant\sup_K|f|\,\,\text{for all}\,f\in\cO_X(X)\}$
is compact.

\noindent
(c) $X$ is a \emph{Stein manifold}
if $X$ is holomorphically separable and convex.
\end{definition}		
%------------
\begin{examples}
(i) Every non-compact Riemann surface is a Stein manifold, 
by a theorem of Behnke-Stein.
\\[2pt]
(ii) $\C^n$ is Stein. A domain in $\C^n$ is Stein if and only if it is 
a domain of holomorphy.
\\[2pt]
(iii) A product of Stein manifolds is Stein.
\\[2pt]		
(iv) Any closed complex submanifold of $\C^n$ is a Stein manifold.		
Conversely, by a theorem of Remmert-Bishop-Narasimhan \cite{Bi:61,Na:60,Rem:56a}, any		
Stein manifold of dimension $n$ can be properly embedded in $\C^{2n+1}$		
and is thus biholomorphic to a closed submanifold of $\C^{2n+1}$. 
\end{examples}
%-----------------------------------------------------------			
An important characterization of Stein manifolds is provided through 
the use of plurisubharmonic functions. 
Let us introduce some convexity concepts for complex manifolds. 		
%------------		
\begin{definition}		
Let $X$ be a complex manifold. A smooth function $\varphi:X\to\R$		
is called \emph{ (strictly) plurisubharmonic}, for short \emph{(strictly) psh}, if for any		
local holomorphic chart $(U,z)$ on $X$ the matrix		
$(\partial^2\varphi/\partial z_j\partial\overline{z}_k)_{j,k}$		
is positive semidefinite (definite), that is, if the $(1,1)$-form		
$\imat\partial\db\varphi$ is semipositive (positive).		
The manifold $X$ is called:		
\\[2pt]		
(a) \emph{$1$-complete} if it admits		
a strictly plurisubharmonic exhaustion function,		
\\[2pt]		
(b) \emph{weakly $1$-complete} if it admits		
a plurisubharmonic exhaustion function,		
\\[2pt]		
(c) \emph{$1$-convex} if it admits		
an exhaustion function that is strictly plurisubharmonic outside		
a compact set of $X$.		
\end{definition}		
%------------		
We note that a smooth function $\varphi:X\to\R$ is strictly plurisubharmonic		
if and only if $\imat\partial\db\varphi$ defines a K\"ahler metric on $X$.		
The notions of $1$-complete and $1$-convex manifolds		
were introduced in the seminal paper by Andreotti-Grauert \cite{AG:62}		
and the notion of weakly $1$-complete manifold was introduced by Nakano 		
\cite{Nak:73}.

A complex manifold is Stein if and only if it is $1$-complete, as shown by Grauert’s solution to the Levi problem \cite{Gr:58}; cf. also \cite[Theorem 5.2.10]{Hor:66}.
%This is proven using $L^2$-estimates for $\db$ in \cite[Theorem 5.2.10]{Hor:66}. 
Any Hermitian metric of a holomorphic line bundle $(L,h^L)$ can be 
modified to a positively curved metric $h^L_\chi=h^Le^{-\chi(\varphi)}$, with $\varphi$ a strictly psh exhaustion function and $\chi$ a rapidly increasing convex function. Therefore, on a Stein manifold, any line bundle is a prequantum line bundle.		
%Hence, on a Stein manifold, the trivial line bundle i
%If $\varphi$ is a strictly plurisubharmonic exhaustion function, then 		
%we define a Hermitian metric on the trivial line bundle $L=\C$ by 		
%defining the norm of the global holomorphic frame		
%$1$ to be $|1|_h:=e^{-\varphi}$.		
%Hence $(L,h^L)=(\C,e^{-\varphi})$ is a prequantum line bundle		
%for the K\"ahler form $\omega=\frac{\imat}{2\pi}R^L=		
%\frac{\imat}{\pi}\partial\db\varphi$. 		

On a Stein manifold, we can construct complete metrics as follows.
Let $\lambda:\R\to\R$ be a smooth, convex, increasing function such that
\begin{equation}\label{eq:lambda''}
\int_0^\infty\!\!\!\sqrt{\lambda''(t)}\,dt=\infty\,.
\end{equation}
Then for any Hermitian metric $\Theta$ on the $X$, the metric
\begin{equation}\label{eq:omega-lambda}
\Theta_\lambda=\Theta+\imat\,\partial\db\lambda(\varphi)
=\Theta+\imat\,\lambda^{\,\prime}(\varphi)\,\partial\db\varphi		
+\imat\,\lambda^{\,\prime\prime}(\varphi)\,\partial\varphi\wedge\db\varphi\,.
\end{equation}
is a complete Hermitian metric.
We will denote in the following by $\cC$ the		
cone of smooth convex increasing functions on $\R$.		

%-----------------------------------------------------------		
\begin{theorem}\label{T:Stein1}		
Let $X$ be a Stein manifold and let $\varphi$ be a 		
strictly plurisubharmonic exhaustion function. 
Let $(L,h^L)$, $(E,h^E)$ be holomorphic Hermitian vector bundles,
with $L$ of rank one. 		
Set $\omega=\frac{\imat}{\pi}\partial\db\varphi$.	 %Then the following holds:

\noindent		
(a) Then there exists $\lambda\in\cC$ such that $\omega_\lambda$ is a complete
K\"ahler metric and 
\begin{equation}\label{eq:RE}
\imat R^E\geq-\omega_\lambda\otimes\Id_E.
\end{equation}
\noindent		
(b) There exists $\chi_0\in\cC$ such that for any $\chi\in\chi_0+\cC$, 
we have with $h^L_\chi\coloneqq h^L e^{-\chi(\varphi)}$:

\noindent		
(1) The Bergman kernel asymptotics for 
$H^{n,0}_{(2)}(X,L^p\otimes E,(h^L_\chi)^p\otimes h^E)$		
holds on compact sets of $X$. 		
	
\noindent	
(2) The Berezin-Toeplitz quantization package holds 		
for the K\"ahler manifold $(X,\omega_\chi)$, the algebra 
$\cC^\infty_{\rm const}(X,\End(E))$, and quantum spaces		
$H^{n,0}_{(2)}(X,L^p\otimes E,(h^L_\chi)^p\otimes h^E)$.		
\end{theorem}		
%-----------------------------------------------------------
\begin{proof}
(a) We first choose $\lambda_1\in\cC$ such that \eqref{eq:lambda''}
is satisfied for $\lambda=\lambda_1$. 
Given that $\imat\,\partial\db\varphi$ is a Kähler form and $\varphi$ is an 
exhaustion function, we use \eqref{eq:omega-lambda} to
determine $\lambda_2\in\cC$ with $\lambda_2^{\prime}$ so rapidly increasing that
\eqref{eq:RE} holds for $\lambda=\lambda_2$ 
(as done in e.g., \cite[Theorem 4.2.2]{Hor:66}).
Then $\lambda=\lambda_1+\lambda_2$ satisfies both conditions 
\eqref{eq:lambda''} and \eqref{eq:RE}.

\noindent 
(b) As in (a),
%Since $\imat\,\partial\db\varphi$ is a K\"ahler metric and $\varphi$ and exhaustion
%function 
we can find $\chi_0$ such that
\begin{equation}\label{eq:RLchi}
\imat R^{(L, h^L_{\chi_0})}=\imat R^{(L, h^L)}+
\imat\partial\db\chi_{0}(\varphi)\geq\omega_\lambda.
\end{equation}
 We check next that Condition \ref{C:specgapgeom2}
is satisfied in the present context for the K\"ahler manifold $(X,\omega_\lambda)$,
and the bundles $(L,h^L_\chi)$ and $(E,h^E)$.
Since $\chi\in\chi_0+\cC$,
we have $\imat\partial\db\chi(\varphi)\geq\imat\partial\db\chi_0(\varphi)$;
hence
\begin{equation}\label{eq:ii2}
\imat R^{(L, h^L_\chi)}=\imat R^{(L, h^L)}+
\imat\partial\db\chi(\varphi)
%\geq\imat R^{(L, h^L)}+\imat\partial\db\chi_0(\varphi)
\geq\omega_\lambda.
\end{equation}
By \eqref{eq:RE}, \eqref{eq:ii2} and the fact that $\omega_\lambda$
is K\"ahler, we deduce that Condition \ref{C:specgapgeom2}
is satisfied, thereby 
ensuring that (1) and (2) follow from Theorem \ref{noncompact1}.
\end{proof}
%===
Let us consider now the case of $1$-convex manifolds.
Since the exhaustion function of a $1$-convex manifold
is strictly psh only outside a compact set, one cannot use it to
construct a Hermitian metric of positive curvature on any line bundle.
We will thus assume the existence of a positive line bundle. 

On a $1$-convex manifold, we can construct a very natural
exhaustion function.
For this purpose, we recall the following
analytic characterization of $1$-convex manifold $X$ 
(see e.\,g.\ \cite{AG:62}): There exists a Stein space $Y$, 
a proper holomorphic surjective map $\rho:X\to Y$ 		
satisfying $\rho_*\mathcal{O}_X = \mathcal{O}_Y$, 		
and a finite set $A\subset Y$ such that the induced map 		
$X\setminus\rho^{-1}(A) \to Y\setminus A$ is biholomorphic. 		
The Stein space $Y$ is called the Remmert reduction of $X$ and 		
$\Sigma\coloneqq\rho^{-1}(A)$ the exceptional analytic set of $X$.
Consider a strictly psh smooth exhaustion function $\varphi_Y$ of $Y$, 
such that $\varphi_Y\geq0$ and $\{\varphi_Y=0\}=A$ (see \cite[p. 563]{Col98}). This is constructed by embedding $Y$ into a Euclidean space
$\C^N$ and constructing such a strictly psh exhaustion function
on $\C^N$.
Then $\varphi=\varphi_Y\circ\rho$ is a smooth psh exhaustion 
function of $X$, such that $\varphi\geq0$, $\{\varphi=0\}=\Sigma$ 
and $\varphi$ is strictly psh on $X\setminus\Sigma$.
%If $\psi:Y\to\R$ is a strictly psh exhaustion function, then
%$\varphi\coloneqq\psi\circ\rho:X\to\R$
%is an exhaustion function which is psh on $X$ and strictly psh outside the 
%compact analytic set $\Sigma$.	

%-----------------------------------------------------------		
\begin{theorem}\label{T:1convex}		
Let $X$ be a $1$-convex manifold and let $\varphi$ be an 
exhaustion function as above. 
Let $(L,h^L)$ be a positive line bundle on $X$, and $(E,h^E)$ 
be a holomorphic Hermitian vector bundle.
Let $\omega$ be a K\"ahler form on $X$.
Then the following holds:

\noindent		
(a) There exist $\lambda\in\cC$ and $C>0$ 
such that $\omega_\lambda$ is a complete
K\"ahler metric and 
\begin{equation}\label{eq:RE2}
\imat R^E\geq-C\omega_\lambda\otimes\Id_E.
\end{equation}
\noindent		
(b) There exists $\chi_0\in\cC$ such that for any $\chi\in\chi_0+\cC$, 
we have for
$h^L_\chi\coloneqq h^L e^{-\chi(\varphi)}${\rm{:}}

\noindent		
(1) The Bergman kernel asymptotics for 
$H^{n,0}_{(2)}(X,L^p\otimes E,(h^L_\chi)^p\otimes h^E)$		
holds on compact sets of $X$. 		
	
\noindent	
(2) The Berezin-Toeplitz quantization package holds 		
for the K\"ahler manifold $(X,\omega_\chi)$, the algebra 
$\cC^\infty_{\rm const}(X,\End(E))$, and quantum spaces		
$H^{n,0}_{(2)}(X,L^p\otimes E,(h^L_\chi)^p\otimes h^E)$.		
\end{theorem}		
%-----------------------------------------------------------
\begin{proof}
(a) We first observe that the $(1,1)$-form $\omega_\lambda$
in \eqref{eq:omega-lambda} is a K\"ahler form since $\varphi$
is psh on $X$. Moreover, the same argument as above shows
that  $\omega_\lambda$ is complete provided $\lambda$ satisfies
\eqref{eq:lambda''}.
For $a\in\R$, we denote by $X_a=\{x\in X:\varphi(x)<a\}\Subset X$.
Let us consider $c<d$ such that $\Sigma\subset X_c\subset X_d$.
There exists $C>0$ such that $\imat R^E\geq-C\omega\otimes\Id_E$
on $X_d$, and thus also $\imat R^E\geq-C\omega_\lambda\otimes\Id_E$.
Given that $\varphi$ is strictly psh outside $\Sigma$, we can select 
$\lambda$ to be rapidly increasing such that 
$\imat R^E \geq -\omega_\lambda \otimes \Id_E$ on $X \setminus X_c$, 
and $\omega_\lambda$ is complete.

\noindent
(b) We fix $\lambda\in\cC$ as in (a).
There exists $\varepsilon>0$ such that
$\imat R^{(L,h^L)}\geq\varepsilon\omega_\lambda$ on $X_d$.
Since$\imat\partial\db\chi(\varphi)$ is semi-positive on $X$ for any $\chi\in\cC$,
we have $\imat R^{(L,h^L_{\chi})}\geq\varepsilon\omega_\lambda$
on $X_d$.
Given that $\varphi$ is strictly psh outside $\Sigma$, there exists
$\chi_0\in\cC$ such that $\imat\partial\db\chi_0(\varphi)\geq\varepsilon\omega_\lambda$
on $X\setminus X_c$.
Since $\imat R^{(L,h^L)}$ is positive, we have 
$\imat R^{(L,h^L_{\chi_0})}\geq\varepsilon\omega_\lambda$
on $X\setminus X_c$. Hence, 
\begin{equation}\label{eq:ii3}
\imat R^{(L,h^L_{\chi_0})}\geq\varepsilon\omega_\lambda \quad
\text{on $X$}.
\end{equation}
By \eqref{eq:RE2}, \eqref{eq:ii3}, and the fact that $\omega_\lambda$
is K\"ahler, we deduce that Condition \ref{C:specgapgeom2}
is satisfied, thereby 
ensuring that (1) and (2) follow from Theorem \ref{noncompact1}.
\end{proof}
We consider now weakly $1$-complete manifolds. To give examples, note that: 		
%------------		
%\begin{examples}		
\\[2pt]		
(i) Any $1$-convex		
manifold is weakly $1$-complete.		
\\[2pt]		
(ii) If $X$ and $Y$ are complex manifolds		
and there exists a proper holomorphic map $\pi: X\to Y$ and $Y$ is weakly $1$-complete,		
then $X$ is weakly $1$-complete, too. Indeed, if $\varphi$ is a smooth 		
psh exhaustion function on $Y$,		
Then $\varphi\circ\pi$ is a smooth 		
psh exhaustion function on $X$.		
%\end{examples}		
%------------		
\begin{theorem}		
Let $X$ be a weakly $1$-complete manifold of dimension $n$		
and let $\varphi:X\to\R$ be a smooth psh exhaustion function.		
Let $(L,h^L)$ be a positive line bundle. We consider the K\"ahler metric
$\omega=\imat R^{(L,h^L)}$ on $X$.
Then the following holds:

\noindent		
(a) There exists $\lambda\in\cC$ 
such that $\omega_\lambda$ is a complete
K\"ahler.

\noindent		
(b) For any $\chi\in\lambda+\cC$, 
we have with 
$h^L_\chi\coloneqq h^L e^{-\chi(\varphi)}${\rm{:}}

\noindent		
(1) The Bergman kernel asymptotics for 
$H^{n,0}_{(2)}(X,L^p,(h^L_\chi)^p)$		
holds on compact sets of $X$. 		
	
\noindent	
(2) The Berezin-Toeplitz quantization package holds 		
for the K\"ahler manifold $(X,\omega_\chi)$, the algebra		
$\cC^\infty_{\rm const}(X)$ and quantum spaces		
$H^{n,0}_{(2)}(X,L^p,(h^L_\chi)^p)$.		
\end{theorem}		
%-----------------------------------------------------------
\begin{proof} (a) If $\lambda\in\cC$ satisfies \eqref{eq:lambda''}
then $\omega_\lambda$ is a complete K\"ahler metric.

\noindent
(b) For $\chi\in\lambda+\cC$ we have 
$\imat R^{(L,h^L_\chi)}\geq\omega_\lambda$
so Condition \ref{C:specgapgeom2}
is satisfied. 
\end{proof}
%
%%-----------------------------------------------------------
%\begin{theorem}
%Let $X$ be a Stein manifold and let $\varphi:X\to\R$ be a smooth psh 
%exhaustion function.
%let $L$ and $E$ be two holomorphic vector bundles on $X$, 
%where $\rank L=1$.
%Then there exists a convex increasing function $\chi_0:\R\to\R$ 
%such that for any convex increasing function $\chi:\R\to\R$ with 
%$\chi\geqslant\chi_0$ we have:
%
%\noindent
%(1) The Bergman kernel asymptotics for 
%$H^0_{(2)}(X,L^p\otimes E)\subset L^2(X,L^p\otimes E,
%h^Le^{-\chi(\varphi)})$ holds on compact sets of $X$.
%
%\noindent
%(2) The Berezin-Toeplitz quantization package holds for the algebra 
%$\cC^\infty_{\rm const}(X,E)$ and quantum spaces 
%$H^0_{(2)}(X,L^p\otimes E)$.
%\end{theorem}
%%---
\subsection{Big line bundles and quasiprojective manifolds}
Let $X$ be a compact complex manifold $X$ of dimension $n$.
A holomorphic line bundle $L$ on $X$ is called big
if its Kodaira-Iitaka dimension equals the dimension of $X$,
equivalently if 
\[\limsup_{p\to\infty} p^{-n}\dim H^0(X,L^p)>0.\]
If a compact manifold $X$ admits a big line bundle then $X$ 
is Moishezon and $L$ admits a singular metric $h^L$, 
smooth outside a proper analytic subset $\Sigma$ 
of $X$, and with strictly positive curvature
current $\imat R^{h^L}$ (see e.\,g.\ \cite[Lemma 2.3.6]{MM07}).
The main result of this section is the Berezin-Toeplitz
quantization of this Zariski open set 
endowed with the generalized Poincar\'e metric.

Let $X$ be a compact connected complex manifold of dimension $n$. 
Let $\Sigma$ be a closed analytic subset of $X$.
Let $\pi:\widetilde{X}\longrightarrow X$ be a resolution of 
singularities %(cf.\ Theorem \ref{hiro1}) 
such that $\pi: 
\widetilde{X}\setminus\pi^{-1}(\Sigma)\longrightarrow 
X\setminus \Sigma$ is biholomorphic and $\pi^{-1}(\Sigma)$ 
is a divisor with normal crossings. 
More precisely, there exists a finite sequence of blow-ups 
\begin{equation} \label{bnm2.1}
\widetilde{X}=X_m \stackrel{\tau_m}{\longrightarrow}{X_{m-1}} 
\stackrel{\tau_{m-1}}{\longrightarrow}\dotsm 
\stackrel{\tau_2}{\longrightarrow}{X_1} 
\stackrel{\tau_1}{\longrightarrow}{X_0}=X 
\end{equation} 
such that 
\begin{description} 
\item[(a)] $\tau_i$ is the blow-up along a non-singular center $Y_{i-1}$  
contained in the strict transform of $\Sigma$ in $X_{i-1}$, $i\geqslant{1}$,
\item[(b)] the strict transform of $\Sigma$ in $\wi{X}=X_{m}$ through  
$\pi=\tau_1\circ\tau_{2}\circ\dotsm\circ\tau_m$ is smooth 
and $\pi^{-1}(\Sigma)$ is a divisor with normal crossings.
\end{description}

Let $g^{T\wi{X}}_0$ be an arbitrary smooth $J$-invariant  metric  
on $\widetilde X$ and $\Theta '(\cdot,\cdot)=g^{T\wi{X}}_0(J\cdot,\cdot)$ 
the corresponding $(1,1)$-from. The \textit{generalized Poincar\'e metric}
on 
$X\setminus \Sigma=\widetilde X\setminus\pi^{-1}(\Sigma)$ 
is defined by the Hermitian form
\begin{equation}\label{poin} 
\Theta_{\var_0}=\Theta ' +\varepsilon_0\sqrt{-1}{\textstyle\sum}_i
\db \partial\log\left((-\log(\|\sigma_i\|^2_i))^2\right)\,, 
\quad \text{$0<\varepsilon_0\ll 1$ fixed}, 
\end{equation}   
where $\pi^{-1}(\Sigma)= \cup_i \Sigma_i$ is the decomposition into 
irreducible components $\Sigma_i$ of $\pi^{-1}(\Sigma)$
and each $\Sigma_i$ is non-singular; $\sigma_i$ are holomorphic sections 
of the associated  holomorphic line bundle $\cO_{\wi{X}}(\Sigma_i)$ 
which vanish to first order on $\Sigma_i$, and 
$\|\sigma_i\|_i$ is the norm for a smooth Hermitian metric $\|\cdot\|_i$ 
on $\cO_{\wi{X}}(\Sigma_i)$ such that $\|\sigma_i\|_i<1$.
 Let $R^{\cO_{\wi{X}}(\Sigma_i)}$ be the curvature of 
$(\cO_{\wi{X}}(\Sigma_i), \|\cdot\|_i)$.
%===
\begin{lemma}[{\cite[Lemma 6.2.1]{MM07}}]\label{lem-poin}
{\rm(i)} The generalized Poincar\'e metric \eqref{poin} 
is a complete Hermitian metric of finite volume. Its Hermitian torsion 
$\mathcal{T}_{\varepsilon_0}
=[i(\Theta_{\varepsilon_0}),\partial\Theta_{\varepsilon_0}]$  and 
the curvature $R^{\det}=R^{K^*_X}$ is also bounded.
%which dominates the Euclidean metric near $\pi^{-1}(\Sigma)$.

\noindent
{\rm(ii)} If $(E,h^E)$ is a holomorphic Hermitian vector bundle over $X$, set
\begin{equation}\label{bnm2.2} 
H_{(2)}^0(X\setminus\Sigma,E)
=\big\lbrace u\in L_{0,0}^2(X\setminus \Sigma, E\,,\,
\Theta_{\var_0}\,,h^E): \db^{E}u=0\big\rbrace ,
\end{equation}   
then 
\begin{equation}\label{bnm2.3} 
H_{(2)}^0(X\setminus\Sigma,E)=H^0(X,E).
\end{equation}   
\end{lemma}
%===
\begin{lemma}[{\cite[Lemma 6.2.2]{MM07}}]\label{bnmt2.2}  
There exists a singular Hermitian line bundle  
$(\widetilde{L},h^{\widetilde L})$ on $\widetilde{X}$ which is strictly  
positive and 
$\widetilde{L}|_{\widetilde{X}\setminus\pi^{-1}(\Sigma)}\cong 
\pi^*(L^{k_0})$, for some $k_0\in\N$.
\end{lemma}
%===
We introduce on $L|_{X\setminus\Sigma}$ the metric 
$(h^{\widetilde L})^{1/k_0}$ whose curvature extends to 
a strictly positive $(1,1)$--current on $\widetilde X$. Set
\begin{subequations}
\begin{align}\label{ell7}  
&h^L_{\varepsilon}:=(h^{\widetilde L})^{1/k_0}\,
\prod_i(-\log(\|\sigma_i\|^2_i))^\var\,, \quad 0<\varepsilon\ll 1\,,\\
&\label{ell8} 
H^0_{(2)}(X\setminus \Sigma,L^p)
:=\big\lbrace u\in L_{0,0}^2(X\setminus \Sigma, 
L^p\,,\,\Theta_{\var_0}\,,h^L_{\varepsilon}):\db^{L^p}u=0\big\rbrace .
\end{align} 
\end{subequations}
The space $H^0_{(2)}(X\setminus \Sigma,L^p)$ is the space of
 $L^2$-holomorphic sections relative to the metrics $\Theta_{\var_0}$ on 
$X\setminus \Sigma$ and 
$h^L_\varepsilon$ on $L|_{X\setminus \Sigma}$. 
Since $(h^{\widetilde L})^{1/k_0}$ is bounded away from zero 
(having plurisubharmonic weights), and its curvature extends 
a strictly positive $(1,1)$-current on $\wi{X}$, 
the elements of this space are $L^2$ integrable 
with respect to the Poincar\'e metric and a smooth metric $h^L_{*}$ 
of $L$ over the whole $X$. By the proof of Lemma \ref{lem-poin}
given in \cite[Proof of Lemma 6.2.1.(ii)]{MM07},
we have 
%\begin{align}\label{bnm2.14}  
$H^0_{(2)}(X\setminus\Sigma,L^p)\subset H^0(X,L^p)$.
%\end{align} 
The space  $H^0_{(2)}(X\setminus \Sigma,L^p)$ 
is our space of polarized sections. 
%===
\begin{theorem}\label{moi} 
Let $X$ be a compact complex manifold with an integral K\"ahler 
current $\omega$. Let $(L,h^L)$ be a singular polarization of $[\omega]$ 
with strictly positive curvature current having 
singular support contained in a proper analytic set $\Sigma$. 
Then the following statements hold:

\noindent
(1) 
The Bergman kernel of the space of polarized sections \eqref{ell8} 
has the asymptotic expansion on compact sets of $X\setminus\Sigma$.

\noindent
(2) The Berezin-Toeplitz quantization package holds for the K\"ahler manifold
$(X\setminus\Sigma,\omega_\varepsilon)$, with
$\omega_\varepsilon=\sqrt{-1}R^{h^L_{\varepsilon}}$.
\end{theorem}  
%===
\begin{proof}
It was shown in \cite[Theorem 6.2.3]{MM07}, by using
Lemmas \ref{lem-poin} and \ref{bnmt2.2}, that the spectral gap
\eqref{bk1.4} holds in the situation at hand.
Thus, both statements follow from Theorem \ref{t2.1}.
\end{proof}
%===
\subsection{Manifolds of bounded geometry}
The purpose of this section is to establish
the Berezin-Toeplitz quantization of manifolds with bounded geometry.
The result itself already appears in \cite[Lemma 4.6]{Fin22b} (derived there from our
method \cite{MM11,MM15});
see also \cite{Kordyukov_Mat_Zametki_2022} for a related statement.
We provide here a short, self-contained proof.
Let us first introduce the notion of bounded geometry in the form we will need it.
\begin{definition}\label{bndedgeomdef}
Let $(X,J,\Theta)$ be a Hermitian manifold and 
let $g^{TX}=\Theta(\cdot, J\cdot)$ be the associated Riemannian metric.
Let $(F,h^F)$ be a holomorphic Hermitian vector bundle.
We say that $(X,J,\Theta)$ and $(F,h^F)$ 
have bounded geometry if the derivatives of any 
order of $R^F$, $J$, $g^{TX}$ are uniformly bounded on $X$ 
in the norm induced by $g^{TX}$, $h^{F}$, and the injectivity radius of 
$(X,g^{TX})$ is positive.
\end{definition}
Let us denote by 
\begin{equation}\label{eq:C_b}
\cC^\infty_b(X,F):=\left\{f\in \cC^\infty(X,F):
\sup_{x\in X}|(\nabla^{F})^ks|_{g^{TX},h^F}<\infty
\text{ for any }k\in \N
\right\},
\end{equation}
where $\nabla^{F}$ is the connection induced by the Chern 
connection $\nabla^F$
and the Levi-Civita connection $\nabla^{TX}$ on the tensor algebra of $T^*X$,
and the norm $|\cdot|_{g^{TX},h^F}$ is induced by $g^{TX}$ and $h^F$.

\begin{assumption}\label{A:bdedgeom}
Let $(X,J,\Theta)$ be a Hermitian manifold of dimension $n$ 
with associated Riemannian metric $g^{TX}=\Theta(\cdot, J\cdot)$.
Let $(L,h^L)$, $(E,h^E)$ be holomorphic Hermitian vector bundles,
with $L$ of rank one.
Suppose that $(X, g^{TX})$ is complete and
$(X,J,\Theta)$, $(L,h^L)$, $(E,h^E)$
have bounded geometry.
\end{assumption}

We recall the following result about the 
exponential decay of the Bergman kernel.
%===
\begin{theorem}[{\cite[Theorem 1]{MM15}}]\label{thm:3.2new23}
In the situation of Assumption \ref{A:bdedgeom} assume 
that there exists $\varepsilon>0$ such that on $X$,
%\begin{align}\label{eq:0.6}
$\sqrt{-1}R^L> \varepsilon\Theta$\,.
%\end{align}
Then there exist $\boldsymbol{c} >0$, $\boldsymbol{p}_{0}>0$, 
which can be determined explicitly
from the geometric data %{\rm(}cf.\ \eqref{bk3.30}{\rm)} 
such that  for any $k\in \N$, there exists $C_k>0$ such that 
for any $p\geqslant \boldsymbol{p}_{0}$\,, $x,x'\in X$, we have 
\begin{equation}\label{eq:0.7}
\left| P_p(x,x')\right|_{\cC^k} \leqslant C_k \, p^{n+\frac{k}{2}}
\, \exp\!\left(- \boldsymbol{c} \,\sqrt{p}\, d(x,x')\right).
\end{equation}
\end{theorem}
%===
In the context of bounded geometry, we have
the following:
%===
\begin{theorem}[{\cite[Lemma 4.6]{Fin22b}}]\label{thm:3.2new23b}
Under the hypotheses of Theorem \ref{thm:3.2new23},
%In the situation of Assumption \ref{A:bdedgeom} assume 
%that there exists $\varepsilon>0$ such that on $X$,
%\begin{align}\label{eq:0.6b}
%\sqrt{-1}R^L> \varepsilon\Theta\,.
%\end{align}
the Berezin-Toeplitz package holds for
the algebra $\cC^\infty_b(X,\End(E))$.
\end{theorem}
%===
\begin{proof}
%Argument why this works for bounded geometry, based on
%Theorem \ref{thm:3.2new23}. 
Let $0<4\varepsilon<a^X$, where $a^X>0$ is the injectivity radius
of $X$. 
At first, for any $0<c_1<\boldsymbol{c}$, $k\in \N$, 
$f\in \cC^\infty_b(X,\End(E))$, there exists $C_k>0$ such that 
for any $p\geqslant \boldsymbol{p}_{0}$, $x,x'\in X$, we have
\begin{equation}\label{eq:3.24}
\left| T_{f,p}  (x,x')\right|_{\cC^k} \leqslant C_k \, p^{n+\frac{k}{2}}
\, \exp\!\left(- c_1 \,\sqrt{p}\, d(x,x')\right).
\end{equation}
In fact, by  \eqref{eq:0.7}, we have 
\begin{equation}\label{eq:3.25}
\begin{split}
 | P_p(x,y)f(y)P_p(y,x')|_{\cC^k \text{on } x,x'} 
&\leqslant  C p^{2n+\frac{k}{2}}
e^{-\boldsymbol{c}\sqrt{p} (d(x,y)+d(y,x') ) }\\
&\leqslant       C    p^{2n+\frac{k}{2}}
e^{-c_1 \sqrt{p} d(x,x')- (\boldsymbol{c}-c_1)\sqrt{p} d(x,y) }.
\end{split} 
\end{equation}
Now, under our assumption of bounded geometry, there exists $K>0$ such that 
the Ricci curvature of $(X, g^{TX})$ is bounded below by $-(2n-1) K^2 g^{TX}$. By Bishop's inequality, this implies that the volume of $B^{X}(x,r)\subset X$
is smaller than or equal to the volume
of a geodesic ball of radius $r>0$ in the space of constant
curvature $-K$ (see, for example, \cite[Lemma\,7.1.3]{Pet16}).
Then, by a classical estimate
for the volume of large balls in the space of constant curvature
$-K$, which can be found for example in \cite[p.\,3]{Mil68},
there exists a universal constant $C_{n,K}>0$, depending only on
$K$ and on the dimension $n$ of $ X$, such that for any $x\in X$ and $r>0$,
\begin{equation}\label{Bishop}
\text{Vol}\,B^{X}(x,r)\leq C_{n,K}\,e^{(2n-1)Kr}\;.
\end{equation}
 Then by \eqref{BTmapfla}, \eqref{eq:3.25}, \eqref{Bishop},  we get
\begin{multline}\label{eq:3.27}
\left| T_{f,p}  (x,x')\right|_{\cC^k} \leqslant 
\sum_{k=0}^\infty\, \int\limits_{B^X(x,(k+1)\varepsilon)\setminus B^X(x,k
\varepsilon)} | P_p(x,y)f(y)P_p(y,x')|_{\cC^k \text{on } x,x'} dv_X(y)\\
\leqslant  \sum_{k=0}^\infty C 
e^{(2n-1)K(k+1)\varepsilon} p^{2n+\frac{k}{2}}
e^{-(\boldsymbol{c}-c_1)\sqrt{p} k \varepsilon}
e^{-c_1\sqrt{p} k d(x,x')}.
\end{multline}
From \eqref{eq:3.27}, for any $p > \left(\frac{(2n-1)K}{\boldsymbol{c}-c_1}\right)^2$, 
the above series is bounded by
$$C p^{2n+\frac{k}{2}}e^{-c_1\sqrt{p} k d(x,x')},$$
and by slightly increasing $c_1$, we obtain \eqref{eq:3.24} if $d(x,x') > \varepsilon$.
If $d(x,x')\leqslant \varepsilon$, then the summation $\sum_{k=2}^\infty$ 
in \eqref{eq:3.27} can be estimated by 
\begin{equation}\label{eq:3.28}\sum_{k=2}^\infty C 
e^{(2n-1)K(k+1)\varepsilon} p^{2n+\frac{k}{2}}
e^{-\boldsymbol{c}\sqrt{p} (2k-1) \varepsilon}
\leqslant C e^{- \boldsymbol{c}\sqrt{p} \varepsilon}.
\end{equation}
The term $\sum_{k=0}^{1}$ can be estimated 
on the normal coordinate centered at $x$ using \eqref{eq:3.25}:
\begin{multline}\label{eq:3.29}
\sum_{k=0}^1 \cdots
\leqslant C\int_{B^X(x, 2\varepsilon)} 
p^{2n+\frac{k}{2}}
e^{-(\boldsymbol{c}-c_1)\sqrt{p} d(x,y) }
e^{-c_1\sqrt{p} d(x,x')} dv_X(y)
\leqslant C p^{n+\frac{k}{2}}
e^{-c_1\sqrt{p} d(x,x')} .
\end{multline}
From \eqref{eq:3.27}-\eqref{eq:3.29}, we get \eqref{eq:3.24}.
Subsequently, we establish the following criterion for Toeplitz operators, 
which serves as an analog of Theorem \ref{toet3.1}. 
\begin{lemma}[{\cite[Theorem 3.18]{Fin22b}}]\label{teot3.23} For a family of operators in Theorem \ref{toet3.1},
	we replace the conditions ii) and iii) with ii)' : 
For any $\varepsilon_0>0$, there exist  $p_0>0$, $C>0$, and $c_1>0$ such that 
for $p>p_0$ and $x,x'\in X$ with $d(x,x')\geq \varepsilon_0$.
\begin{align}\label{eq:3.31}
	\left| T_{p}  (x,x')\right| \leqslant C p^n
 \exp\!\left(- c_1 \,\sqrt{p}\, d(x,x')\right).
\end{align}
 We assume in iv) that for each $r$, the polynomial $\mathcal{Q}_{r,x_0}(\mT)$
as a section of $\End(E)$ twisted with tensor algebras of $T^*X$
 is uniformly bounded with derivatives and its degree on $x_0\in X$.
 Then $\{T_p\}_p$ is a Toeplitz operator.
\end{lemma}
\begin{proof} We follow step by step the proof of Theorem \ref{toet3.1}.
At first, from our assumption iv) on polynomial $\mathcal{Q}_{r,x_0}(\mT)$,
we know that $g_0$ in \eqref{toe3.5} is in $\cC^\infty_b(X, \End(E))$. 
To obtain Proposition \ref{toet3.2}, we need to modify the argument 
in the proof of \cite[Lemma 4.13]{MM08b} as follows by using assumption ii)': 
We use the notation in  \cite[Lemma 4.13]{MM08b}; by \eqref{eq:3.31},
note that $R_{r,p}(x,y)=0$ for $d(x,y)\geq \varepsilon'$,
 and for any $k>0$, there exist $c_1>0$ and $C>0$ such that 
\begin{align}\label{eq:3.32}
	\begin{array}{ll}
		\Big|\big(\cT_p-\sum_{r=1}^k p^{-r/2}R_{r,\,p}\big)(x,y)\Big|
	&\leqslant  C  \exp(-c_1\sqrt{p} d(x,y)) \text{ if } d(x,y)\geq \varepsilon',\\
& \hspace{-1.5cm} C p^{n -(k+1)/2}  \exp(- c_1 \sqrt{p} d(x,y)) + \mO(p^{-\infty})
\text{ if } d(x,y)< \varepsilon'.
\end{array}\end{align}
Now, by the argument in the proof of \eqref{eq:3.24}, i.e., we decompose the integral 
$\int_X$ by the sum of the integrals on $B^X(x,(k+1)\varepsilon)\setminus B^X(x,k
\varepsilon)$, and using \eqref{eq:3.32}, we obtain
\begin{align}\label{eq:3.33}
\begin{array}{ll}
&\displaystyle\int_X \Big|\big(\cT_p-\sum_{r=1}^k p^{-r/2}R_{r,\,p}\big)(y,x)\Big| dv_X (y)
=\mO(p^{-1}),\\
&\displaystyle\int_X \Big|\big(\cT_p-\sum_{r=1}^k p^{-r/2}R_{r,\,p}\big)(x,y)\Big| dv_X (y)
= \mO(p^{-1}).
\end{array}\end{align}	
From \eqref{eq:3.33}, we obtain \cite[(4.48)]{MM08b}
\begin{multline}\label{b6.381}
\big\|\big(\cT_p-\sum_{r=1}^k p^{-r/2}R_{r,\,p}\big)\,s \|_{L^2} ^2
\leqslant \int_{X} \Big(\int_{X}
\Big|\big(\cT_p-\sum_{r=1}^k p^{-r/2}R_{r,\,p}\big)(x,y)\Big| dv_{X}(y)\Big)\\
\times \Big(\int_{X}
\Big|\big(\cT_p-\sum_{r=1}^k p^{-r/2}R_{r,\,p}\big)(x,y)\Big| |s(y)|^2
dv_{X}(y)\Big)dv_{X}(x)
\leq C  p^{-2}\|s\|_{L^2}^2.
\end{multline}
Thus \cite[Lemma 4.13]{MM08b} holds in our situation. 
\end{proof}

Now we fix $f,g\in \cC^{\infty}_b(X,\End(E))$. From \eqref{eq:3.24}, 
and by using the same trick as in \eqref{eq:3.27}-\eqref{eq:3.29},
we get that  $(T_{f,p}T_{g,p})(x,x')$ satisfies conditions ii)' 
and iv) in Lemma \ref{teot3.23}. this implies that 
$T_{f,p}T_{g,p}$ is a Toeplitz operator by Lemma \ref{teot3.23}. 

We adapt the proof of \eqref{toe4.17a} presented in Theorem \ref{t2.1}. 
Suppose the supremum of $f$ is attained at a specific point $x_0$. 
In that case, this proof remains unchanged, as the expansion of the 
peak section \eqref{toe4.18a}, 
derived from the asymptotic expansion of the Bergman kernel, 
continues to hold on manifolds characterized by bounded geometry, 
as shown by Theorem \ref{thm:3.2new23}.
If the supremum is not attained, then
the same proof gives us that for any $\varepsilon > 0$, there exists 
$p_0>0$, such that for every $p\geq p_0$, we have
%\begin{equation}\label{toe4.18b}
$\|f\|_\infty-\varepsilon\leqslant\|T_{f,p}\|$,
%\end{equation}
and this entails \eqref{toe4.17a} in the present case.
\end{proof}
%===
\begin{example}
(1) Let $(X,J,\Theta)$ be a compact Hermitian manifold and 
with associated Riemannian metric $g=\Theta(\cdot, J\cdot)$.
Let $(L,h^L)$, $(E,h^E)$ be holomorphic Hermitian vector bundles,
with $L$ of rank one. Let $\rho: \widetilde{X} \to X$ be a Galois covering. 
Let us decorate by $\sim$ pullbacks of objects on $X$ by $\rho$. 
Then $(\widetilde{X}, \widetilde{g})$ is complete and
$(X,\widetilde{J},\widetilde{\Theta})$, $(\widetilde{L}, \widetilde{h^L})$, 
$(\widetilde{E},\widetilde{h^E})$ have bounded geometry.
Moreover there exists $\varepsilon>0$ such that 
$\imat R^{\widetilde{L}}>\varepsilon\,\widetilde{\Theta}$.

(2) Let $D$ be a smoothly bounded strictly pseudovonvex domain
in $\C^n$ or in a Stein manifold.
Then, for each fixed $\lambda<0$, there exists a unique complete
K\"ahler metric $\omega$ on $D$ 
satisfying $\Ric(\omega)=\lambda\omega$.
For the unique complete K\"ahler-Einstein metric $\omega_{\rm{CY}}$ 
with $\Ric(\omega_{\rm{CY}})=-\omega_{\rm{CY}}$ (the Cheng-Yau metric), 
the canonical bundle is positive and polarizes the metric.
It is known to have bounded geometry, 
as a result of its asymptotically complex hyperbolic nature, 
proven through the work of Cheng-Yau \cite{CY80}
and the boundary regularity analysis of Lee-Melrose \cite{LeMe82}.
Hence the results in this section apply to $(D,\omega_{\rm{CY}})$
and the canonical bundle $K_D$ endowed with the metric induced by $\omega_{\rm{CY}}$.

\end{example}
%===
\subsection{Pseudoconvex domains}
\label{pseudocvx}
%===
In this section, we consider relatively compact pseudoconvex domains
with smooth boundary in a complex manifold. They are endowed with an incomplete Hermitian
metric, and we will thus work with the Kodaira Laplacian with
$\db$-Neumann boundary conditions. 
This generalizes the results of Englis for strictly pseudoconvex domains
in $\C^n$.

Let $M$ be a complex manifold and 
let $X$ be a smooth, relatively compact domain in $M$.
We set $X=\{x\in M\,:\,\varrho(x)<0\}$
, where $\varrho\in\cC^\infty(M)$ is a defining function that satisfies
$|d\varrho|=1$ on $\partial X$. 
The \emph{Levi form} of $\partial X$ is the restriction
of $\partial\db\varrho$ to the holomorphic tangent bundle of $\partial X$.
The domain $X$ is called \emph{strictly pseudoconvex (pseudoconvex)} if the Levi form
is positive definite (semi-definite) at each point of $\partial X$.
Let us consider a holomorphic Hermitian vector bundle $(F,h^F)$ on $M$.
The complex manifold $M$ is endowed with a Hermitian metric
with $(1,1)$-form $\Theta$, and we consider its restriction
to $X$ with volume form $dv_X=\Theta^n/n!$.
We construct, as in \eqref{lm2.0}, \eqref{lm2.02a}, the spaces $L^{2}(X,F)$
and $H^{0}_{(2)}(X,F)$. 
Let $\overline{\partial}^{\smash{\scriptscriptstyle F}}:
\Omega^{0,\bullet}(M,F)\to\Omega^{0,\bullet\,+1}(M,F)$ be 
the Dolbeault operator; we denote by 
$\overline{\partial}^{\smash{\scriptscriptstyle{F}},*}$ its formal adjoint. Let 
$\overline{\partial}^{\smash{\scriptscriptstyle F}}:
\Dom(\overline{\partial}^{\smash{\scriptscriptstyle F}})\subset 
L^2_{0,\bullet}(X,F)\to L^2_{0,\bullet\,+1}(X,F)$ 
be its maximal extension on $L^2_{0,\bullet}(X,F)$. % (cf.\ \eqref{gm1.2}).
Let $\overline{\partial}^{\smash{\scriptscriptstyle{F}},*}_H$ be 
the Hilbert space adjoint of $\overline{\partial}^{\smash{\scriptscriptstyle F}}$ 
on $X$. In order to describe the domain of 
$\overline{\partial}^{\smash{\scriptscriptstyle{F}},*}_H$ 
we now present the following concepts.
Let $-e_\mathfrak{n}\in TM$ be the metric dual of $d\varrho$.
Then $e_\mathfrak{n}\in TM$
is the inward pointing unit normal at $\partial X$.
We decompose $e_\mathfrak{n}$ as
$e_\mathfrak{n}= e_\mathfrak{n}^{(1,0)}+e_\mathfrak{n}^{(0,1)}
\in T^{(1,0)}M\oplus T^{(0,1)}M$.
We introduce the space
%-----
\begin{equation}\label{lm2.85}
B^{0,q}(X,F)=\left\{s\in\Omega^{0,q}(\overline{X},F):
i_{e_\mathfrak{n}^{(0,1)}}s=0
%\sigma(\db^{E,*},d\varrho)s=0\,
\, \, \text{on}\,\,\partial X\right\}.
\end{equation}
%-----
It is then well-known that
$B^{0,q}(X,F)=\Dom(\overline{\partial}^{\smash{\scriptscriptstyle{F}},*}_H)
\cap\Omega^{0,q}(\overline X,F)$ and
$\overline{\partial}^{\smash{\scriptscriptstyle{F}},*}_H=\overline{\partial}^{\smash{\scriptscriptstyle{F}},*}$ on $B^{0,q}(X,F)$ (cf.\ \cite{FK:72,Hor:65}, 
\cite[Proposition 1.4.19]{MM07}).
%
%Let $\db^E:\Omega^{0,\bullet}(X,E)\to\Omega^{0,\bullet\,+1}(X,E)$ be 
%the Dolbeault operator; we denote by $\db^{E,*}$ its formal adjoint. Let 
%$\db^E:\Dom(\db^E)\subset L^2_{0,\bullet}(M,E)\to L^2_{0,\bullet\,+1}(M,E)$ 
%be its maximal extension on $L^2_{0,\bullet}(M,E)$ (cf. Lemma \ref{glol1}).
%Let $\db^{E,*}_H$ be the Hilbert space adjoint of $\db^{E}$ on $M$.
%We introduced the space $B^{0,q}(M,E)$ in \eqref{lm2.85}.
%By Proposition \ref{lmt2.18}, we have
%we know that $B^{0,q}(M,E)=\Dom(\db^{E,*}_H)\cap\Omega^{0,q}(\overline M,E)$ 
%and $\db^{E,*}_H=\db^{E,*}$ on $B^{0,q}(M,E)$. 
Thus
\begin{equation}\label{ell-sym}
\big\langle\db^{F}s_1,s_2\big\rangle
=\big\langle s_1,\overline{\partial}^{\smash{\scriptscriptstyle{F}},*}s_2\big\rangle\,,\quad 
\text{ for } s_1\in\Omega^{0,q}(\ov{X},F)\,,s_2\in B^{0,q+1}(X,F).
\end{equation}
We consider the operator
\begin{equation} \label{neu-rest}
\begin{split}
\Dom(\square^F):=\big\{s\in B^{0,q}(X,F)\,:\,\db^{\smash{\scriptscriptstyle F}} s\in B^{0,q+1}(X,F)\big\},\\
\square^Fs=\overline{\partial}^{\smash{\scriptscriptstyle F}}\overline{\partial}^{\smash{\scriptscriptstyle{F}},*} s+\overline{\partial}^{\smash{\scriptscriptstyle{F}},*}\overline{\partial}^{\smash{\scriptscriptstyle F}} s\,,
\quad\text{for $s\in\Dom(\square^F)$}\,,
\end{split}
\end{equation}
which by \eqref{ell-sym} is positive. 
%
%Let $e_\mathfrak{n}$ be the inward pointing unit normal at $\partial M$. 
%We decompose $e_\mathfrak{n}$ as 
%$e_\mathfrak{n}= e_\mathfrak{n}^{(1,0)}+e_\mathfrak{n}^{(0,1)}
%\in T^{(1,0)}X\oplus T^{(0,1)}X$. 
Then the boundary conditions of $\Dom(\square^E)$ in \eqref{neu-rest} 
are called 
\emph{$\db$-Neumann boundary conditions} \cite{FK:72,Hor:65} is given by :
\begin{equation}\label{gm4.2}
\Dom(\square^F)=\big\{s\in \Omega^{0,\bullet}(\ov{X},F);\,\,
 i_{e_\mathfrak{n}^{(0,1)}}s= i_{e_\mathfrak{n}^{(0,1)}}\overline{\partial}^{\smash{\scriptscriptstyle F}} s=0
\quad \text{on  } \partial X\big\}.
\end{equation}
An extension of the associated quadratic form $Q$ is
\begin{equation}\label{gm4.3}
\Dom(Q):= B^{0,q}(X,F), \, \, \,
Q(s_1,s_2):=\langle \db^{\smash{\scriptscriptstyle F}} s_1,
\db^{\smash{\scriptscriptstyle F}} s_2\rangle 
+\langle \db^{\smash{\scriptscriptstyle{F}},*} s_1,
\db^{\smash{\scriptscriptstyle{F}},*} s_2\rangle \,.
\end{equation}
It is easy to see that $Q$ is closable, so there 
exists a self-adjoint operator associated 
with the closure $\ov{Q}$
 (the Friedrichs extension of $\square^F$) called, in this context
, \emph{Kodaira Laplacian with $\db$-Neumann boundary conditions}. 
We still denote this operator by $\square^F$.
We have an analog of the Andreotti-Vesentini density result (Theorem \ref{esa}).
%-----------------------------------------------------------------------------
\begin{lemma}[{\cite{FK:72,Hor:65}}]\label{aprox}
$\Omega^{0,\bullet}(\ov{X},F)$ is dense in $\Dom(\db^{\smash{\scriptscriptstyle F}})$ in the graph-norm
of $\db^{\smash{\scriptscriptstyle F}}$, and $B^{0,q}(M,F)$ is dense in $\Dom(\db^{\smash{\scriptscriptstyle{F}},*}_H)$ and in 
$\Dom(\db^{\smash{\scriptscriptstyle F}})\cap\Dom(\db^{\smash{\scriptscriptstyle{F}},*}_H)$ in the graph-norms of 
$\db^{\smash{\scriptscriptstyle{F}},*}$ and $\db^E+\db^{\smash{\scriptscriptstyle{F}},*}$, respectively.
\end{lemma}
%-----------------------------------------------------------------------------
From this we deduce immediately the following (see e.\,g.\, \cite[Proposition 3.5.2]{MM07}).%-----------------------------------------------------------------------------
\begin{proposition}\label{Gaff=Fried}
The Kodaira Laplacian with $\db$-Neumann conditions on $X$
%(i.e. the Friedrichs extension of \eqref{neu-rest})
coincides with the Gaffney extension \eqref{ell-} of the Kodaira Laplacian. 
\end{proposition}
%----------------------------------------------------------------------------
%---------------
We recall now the Bochner-Kodaira-Nakano formula with boundary term
\cite[Theorem 1.4.21]{MM07}.
For $s\in \Omega^{0,q}(\ov{X},F)$ and $y\in \partial X$, set
\begin{align}\label{lm2.86}
\cL_{\varrho}(s,s)
= (\partial\db\varrho) (w_k,\ov{w}_j)\langle \ov{w}^j\wedge i_{\ov{w}_k} s,
s\rangle_{\Lambda^{\bullet,\bullet}\otimes E,y}.
\end{align}
By \cite[Theorem 1.4.21]{MM07} we have 
for any $s\in B^{0,\bullet}(X,F)$,
\begin{equation}\label{bkn-bdy}
\begin{split}
\norm{\overline{\partial}^{\smash{\scriptscriptstyle F}}s}^2_{L^2}+\norm{\overline{\partial}^{\smash{\scriptscriptstyle{F}},*}s}^2_{L^2}
=\norm{(\nabla^{\wi{F}})^{1,0*}\wi{s}}^2_{L^2}+\big\langle
R^{F\otimes K^*_X}(w_j,\overline w_k)
\overline w^k\wedge i_{\overline w_j}s,s\big\rangle\\
-\big\langle\overline{\partial}^{\smash{\scriptscriptstyle F}}s, \Psi^{-1}\overline{\mathcal{T}}\,\wi{s}\big\rangle
-\big\langle\Psi^{-1}\overline{\mathcal{T}}^*\wi{s},\overline{\partial}^{\smash{\scriptscriptstyle{F}},*}s\big\rangle
+\big\langle\mathcal{T}^*\wi{s},(\nabla^{\wi{F}})^{1,0*}\wi{s}\big\rangle\\
+\int_{\partial X}\cL_{\varrho}(s,s)\,dv_{\partial X}\,.
\end{split}
\end{equation}
Especially, we obtain the following Bochner-Kodaira-Nakano inequalitiy
\cite[Corollary 1..4.22]{MM07}. For any $s\in B^{0,q}(X,F)$,
%-----------------------------------------------------------------------
\begin{equation} \label{herm20,111}
\begin{split}
\frac{3}{2}\big(\norm{\overline{\partial}^{\smash{\scriptscriptstyle F}}s}^2_{L^2}+\norm{\overline{\partial}^{\smash{\scriptscriptstyle{F}},*}s}^2_{L^2}\big)
\geqslant\frac{1}{2}\,\norm{(\nabla^{\wi{F}})^{1,0*}\wi{s}}^2_{L^2}
+\big\langle R^{F\otimes K^*_X}(w_j,\ov{w}_k)\ov{w}^k\wedge i_{\ov{w}_j}s,
s\big\rangle\\
+\int_{\partial X}\cL_{\varrho}(s,s)\,dv_{\partial X}
-\frac{1}{2}\big(\norm{\mathcal{T}^*\wi{s}}^2_{L^2}
+\norm{\ov{\mathcal{T}}\wi{s}}^2_{L^2}
+\norm{\ov{\mathcal{T}}^*\wi{s}}^2_{L^2}\big).
\end{split}
\end{equation}

%-----------------------------------------------------------
\begin{theorem}
Let $X$ be a relatively compact pseudoconvex domain with smooth 
boundary in a complex manifold $M$. Let $L$ and $E$ be 
two holomorphic vector bundles on $M$, where $\rank L=1$. 
Assume that $(L,h^L)$ is positive on a neighbourhood of $\overline X$.
Then we have

(1) The Bergman kernel asymptotics for $H^0_{(2)}(X,L^p\otimes E)$
holds on compact sets of $X$.

%{\color{red}We should be able to prove it for $\ov{X}$?}

(2) The Berezin-Toeplitz quantization package holds for the K\"ahler manifold
$(X,\omega)$, where $\omega=\frac{\imat}{2\pi}R^{(L,h^L)}$, the algebra 
$\cC^\infty_{\rm const}(X,\End(E))$ and quantum spaces 
$H^0_{(2)}(X,L^p\otimes E)$.

%{\color{red}Maybe $C^\infty(\ov{X})$?}
\end{theorem}
%-----------------------------------------------------------
%The $L^2$ condition here is understood with respect to 
%a smooth Riemannian metric on $M$. 
\begin{proof}
Since $X$ is pseudoconvex we have $\cL_{\varrho}(s,s)\geq0$ pointwise
on $\partial X$. Moreover the torsion of the the metric $\Theta$ 
and $R^{K^*_X}$ are bounded on $\ov{X}$. Hence \eqref{herm20,111} yields
immediately the spectral gap \eqref{bk1.4} and we can apply Theorem \ref{t2.1}.
%
%For the proof of the spectral gap see also \cite[Theorem 3.5.10]{MM07}, 
%where it was used to prove holomorphic Morse inequalities. 
%In this case the Gaffney extension of the Kodaira Laplacian coincides 
%with the Kodaira Laplacian with $\db$-Neumann conditions 
%\cite[Prop.\,3.5.2]{MM07}.
\end{proof}

\section{Szeg\H{o}-type limit formulas}
\label{S:Szego}

Boutet de Monvel and Guillemin \cite{BdMG81,Guill79} obtained complex variable 
analogues of the classical Szeg\H{o} theorem \cite{Sz:20}.
The analogous result for projective manifolds endowed with the restriction
of the hyperplane bundle was originally proved in \cite[Theorem 13.13]{BdMG81},
\cite[Theorem 1, p.\,248]{Guill79} and for arbitrary positive line bundles in 
\cite{Bern03},
see also \cite{Lind01}. In \cite[Theorem 1.6]{HM17b}
the asymptotics 
are proved for a semi-classical spectral function of the Kodaira Laplacian
on an arbitrary manifold.
%===
\begin{lemma}
Let $(X,\Theta)$ be a Hermitian manifold,	
$(L,h^L)$ and $(E,h^E)$ be holomorphic Hermitian		
vector bundles on $X$ of rank one and $r$, respectively.
Let $f\in\boldsymbol{L}^\infty(X)$ 
be non-negative and have compact
support.
Then the Toeplitz operator $T_{f,p}$ is a compact
operator on $H^0_{(2)}(X, L^p\otimes E)$. 
If the set where $f$ does not vanish has a non-empty interior,
then $T_{f,p}$ is injective.
\end{lemma}
%===
\begin{proof}
$T_{f, p}$ is a positive operator and
	\begin{align}\label{6.20a}
	{\rm Tr}\big[T_{f, p}\big]=\int_{X}f(x){\rm Tr}\big[P_{p}(x, x)\big]
	dv_{X}(x)< +\infty.
	\end{align}
By \eqref{6.20a}, $T_{f, p}$ is of trace class and is therefore compact
(cf.\ \cite[Theorem VI.\,21]{RS})).
On the other hand, if the set where $f$
does not vanish has a non-empty interior, then
for any $s\in H^0_{(2)}(X, L^p\otimes E)$,
we have by definition
\begin{equation}
\langle T_{f,p}s,s\rangle=\int_X\,f\,|s|_{L^p\otimes E}^2\,dv_X>0\,,
\end{equation}
by holomorphicity of $s$ and since $f$ is non-negative by assumption.
This implies that $T_{f,p}$ is injective, hence
concluding the proof of the Lemma.
\end{proof}
%===
We denote by $\spec(T)$ the spectrum of an operator $T$.
Then $\spec(T_{f, p})\subset\big[0,\|f\|_\infty\big]$ and
$\spec(T_{f, p})\cap\big]0,\|f\|_\infty\big]$ consists of at most a 
countable set of eigenvalues of finite multiplicity that can cluster only at $0$.
If $H^0_{(2)}(X, L^p\otimes E)$ is infinite dimensional we have
$0\in\spec(T_{f, p})$.
We denote the positive eigenvalues of $T_{f, p}$ counted with multiplicity by 
\begin{equation}\label{e:specTp}
\lambda_{p, 1}\geq\lambda_{p, 2}\geq\ldots\geq\lambda_{p,j}\geq\ldots,
\end{equation}
so $\spec(T_{f, p})\cap\big]0,\|f\|_\infty\big]=\{\lambda_{p,j}:j\in J_p\}$.
%We have $\lambda_{p,j}\to0$, $j\to\infty$.
The spectral density measure of $T_{f, p}$ on the interval $[0,\infty[$
is 
%\begin{equation}\label{e:specdf1}
$\mu_{f,p}=\sum_{j\in J_p}\delta_{\lambda_{p,j}}$\,.
%\end{equation}
%The normalized spectral density measure of $T_{f, p}$ is given by 
%\begin{equation}\label{e:specdf2}
%v_{f,p}=p^{-n}\mu_{f,p}\,.
%\end{equation}
We define the spectral counting function of $T_{f, p}$ as
follows:
\begin{align}
N_{p}(u)=\#\big\{j; \lambda_{p, j}>u\big\}
=\int_{]u,\|f\|_\infty]}d\mu_{f,p}\,.
\end{align}
%===
\begin{theorem}\label{t6.1}
Under the hypotheses of Theorem \ref{t2.1},
let $f$ be a continuous nonnegative function with compact support.
%Then the Toeplitz operator $T_{f,p}$ is a compact
%operator on $H^0_{(2)}(X, L^p\otimes E)$ and we
%denote by $\lambda_{p, 1}\geq\lambda_{p, 2}\geq\ldots$ its eigenvalues
%counted with multiplicity. 
Then the sequence of normalized spectral measures $p^{-n}\mu_{f,p}$
converges weakly on $]0,\|f\|_\infty]$ to the pushforward of the Liouville
measure $\rank(E)c_1(L,h^L)^{n}/n!$ by $f$.
\begin{equation}\label{e:sz1}
p^{-n}\mu_{f,p}\to \rank(E) f_{\ast}\left(\frac{1}{n!}c_1(L,h^L)^{n}\right)\,,
\:\:\text{on $]0,\|f\|_\infty]$ as $p\to\infty$.}
\end{equation}
The counting function of the spectrum
	of $T_{f,p}$ has the following asymptotics:
	for any $\lambda>0$ as $p\to +\infty$.
	\begin{equation}\label{6.1}
	N_{p}(\lambda)=
	\rank(E) \frac{p^n}{n!}\!\int\limits_{\{f>\lambda\}}\!\!
	\left(\frac{\sqrt{-1}}{2\pi}R^L\right)^{\!\!n}+o(p^n)\,.
	\end{equation}
	%Let $M\subset X$ be a relatively compact open set and denote by
	%$T_{M,p}$ the Toeplitz operator with symbol the characteristic
	%function of $M$. Then we have
If $X$ is compact, the asymptotics in \eqref{e:sz1} hold on $[0,\|f\|_\infty]$
for the full spectral measures $\widetilde{\mu}_{f,p}=
\sum_{\lambda\in\spec{T_{f,p}}}\delta_\lambda$ (multiplicities counted), 
and in \eqref{6.1} also for $\lambda=0$.
\end{theorem}
\begin{proof}
We follow the proof of \cite[Theorem 3.1]{MM08a}
(cf.\ also \cite[Theorem 32]{MS24a}).
We denote the normalized spectral density 
measure on $]0, \|f\|_{\infty}]$ by $v_{p}=p^{-n}\mu_{f,p}$. Then
	\begin{align}\label{6.4}
	v_{p}=-p^{-n}\frac{d}{du}N_{p}(u),\ \ \ u\in \,]0, \|f\|_{\infty}].
	\end{align}
%Note that $\dim H^{0}_{(2)}(X, L^{p}\otimes E)$ can be infinite; thus,
%we only consider on $]0, \|f\|_{\infty}]$. 
Clearly, $v_{p}$ is a sum of Dirac measures
supported on ${\rm Spec}(T_{f, p})\,\cap \,]0, \|f\|_{\infty}]$.
	
We claim that the weak limit of the sequence $\{v_{p}\}_{p\geqslant 1}$
is the direct image measure
${\rm rk}(E)f_{\ast}(\frac{\omega^{n}}{n!})$ with
$\omega=\frac{\sqrt{-1}}{2\pi}R^{L}$,
that is, for every continuous function
$\varphi\in \mathscr{C}_{0}\big(]0,\|f\|_{\infty}]\big)$, we have
\begin{align}\label{6.22a}
\lim_{p\rightarrow +\infty}\int^{\|f\|_{\infty}}_{0}\varphi\,dv_{p}
=\int_{X}(\varphi\circ f)\,\frac{\omega^{n}}{n!}.
\end{align}
	
	\noindent
	By the argument of the proof of \cite[Theorem 3.5]{BMMP14}, which is local, and our assumption
	that $f\in \mathscr{C}^{0}(X, [0, \infty[)$, we have for
	any $m\geqslant 1$,
	\begin{align}\label{6.23a}
	\int^{\|f\|_{\infty}}_{0}x^{m}dv_{p}
	&= p^{-n}{\rm Tr}\big[T^{m}_{f, p}\big]
	=p^{-n}\int_{X}{\rm Tr}\big[\underbrace{T_{f, p}\cdots
		T_{f, p}}_{m\ {\rm times}}(x, x)\big]dv_{X}(x)
	\nonumber \\ &=
	p^{-n}\int_{X}f(x){\rm Tr}\big[P_{p}
	\underbrace{T_{f, p}\cdots T_{f, p}}_{m-1\ {\rm times}}(x, x)\big]
	dv_{X}(x).
	\\& =
	{\rm rk}(E)\int_{X}f(x)^{m}\frac{\omega^{n}}{n!}+o(1).
	\nonumber
	\end{align}
	Now we apply the Weierstrass approximation theorem for
	$\frac{1}{x}\varphi \in \mathscr{C}_{0}\big(]0, \|f\|_{\infty}]\big)$,
	we know that
	any $\varphi\in \mathscr{C}_{0}\big(]0, \|f\|_{\infty}]\big)$
	can be approximated uniformly
	by polynomials without a constant term. Now from (\ref{6.23a}),
	we get (\ref{6.22a}).
	By approximating the characteristic function $1_{]\lambda, \|f\|_{\infty}]}$
	for $\lambda>0$ by continuous functions $f_{k}$, we obtain
	\begin{align}
	\lim_{p\rightarrow +\infty}\int^{\|f\|_{\infty}}_{0}f_{k} dv_{p}
	=\int_{X}(f_{k}\circ f)\,\frac{\omega^{n}}{n!}.
	\end{align}
	Letting $k\rightarrow +\infty$ yields
	\begin{align}\label{6.8a}
	\lim_{p\rightarrow +\infty}\int^{\|f\|_{\infty}}_{0}1_{]\lambda, \|f\|_{\infty}]} dv_{p}
	=\int_{X}(1_{]\lambda, \|f\|_{\infty}]}\circ f)\,\frac{\omega^{n}}{n!}
	=\int\limits_{\{f>\lambda\}}\frac{\omega^{n}}{n!}.
	\end{align}
	By \eqref{6.4}, we find
	\begin{align}\label{6.8b}
	\int^{\|f\|_{\infty}}_{0}1_{]\lambda, \|f\|_{\infty}]} dv_{p}=p^{-n}N_{p}(\lambda).
	\end{align}
	Then (\ref{6.1}) follows from (\ref{6.8a}) and (\ref{6.8b}).
	The proof of Theorem \ref{t6.1} is completed.
	\end{proof}

\bibliographystyle{siam}
%\bibliography{mmbook,bbstat}

\def\cprime{$'$}

\end{document}